\documentclass[11pt,twoside]{article}  

\usepackage[margin = 2.5cm]{geometry}

\usepackage{amsmath,amssymb,amsfonts,amsthm}
\usepackage{appendix}
\usepackage[dvipsnames]{xcolor}
\usepackage{siunitx}
\usepackage{graphicx}
\usepackage[numbers]{natbib}
\RequirePackage[colorlinks,citecolor=blue,urlcolor=blue]{hyperref}

\usepackage{enumitem}

\newtheorem{lem}{Lemma}
\newtheorem{prop}{Proposition}
\newtheorem{thm}{Theorem}
\newtheorem*{thm*}{Theorem}
\newtheorem{cor}{Corollary}
\theoremstyle{definition}
\newtheorem{defn}{Definition}

\newcommand{\E}{\mathbb{E}}

\newcommand{\Var}{\mathrm{Var}}
\renewcommand{\P}{\mathbb{P}}
\newcommand{\R}{\mathbb{R}}

\newcommand{\N}{\mathbb{N}}
\newcommand{\ones}{\boldsymbol{1}}
\newcommand{\bits}{\lbrace 0,~1\rbrace}

\newcommand{\diag}{\mathrm{diag}}
\newcommand{\TV}{\mathrm{TV}}
\newcommand{\pimin}{\pi_{\mathrm{min}}}
\newcommand{\eff}{\mathrm{eff}}
\newcommand{\sigf}{\sigma_{f}}
\newcommand{\sigfasym}{\sigma_{f,\mathrm{asym}}}

\newcommand{\parity}{f_{\mathrm{parity}}}

\newcommand{\alg}{\mathcal{A}}
\newcommand{\algfix}{\mathcal{A}_{\mathrm{fixed}}}
\newcommand{\algseq}{\mathcal{A}_{\mathrm{seq}}}
\newcommand{\algdiff}{\mathcal{A}_{\mathrm{hard}}}
\newcommand{\err}{\mathrm{err}}

\newcommand{\der}{\mathrm{d}}

\newcommand{\abslambda}{\lambda_{\ast}}

\newcommand{\Bern}{\mathrm{Bern}}
\newcommand{\Beta}{\mathrm{Be}}

\newcommand{\Norm}{\mathcal{N}}

\newcommand{\Dpi}{D}

\newcommand{\Tf}{T_{f}}

\newcommand{\df}{d_{f}}

\newcommand{\widgraph}[2]{\includegraphics[keepaspectratio,width=#1]{#2}}

\newcommand{\mydefn}{\ensuremath{: \, =}}

\newcommand{\Jmax}{\ensuremath{J_{\tiny{\mbox{max}}}}}
\newcommand{\real}{\ensuremath{\mathbb{R}}}
\newcommand{\Xtil}{\ensuremath{\widetilde{X}}}

\newcommand{\SPECIAL}{\ensuremath{F}}

\newcommand{\matsnorm}[2]{|\!|\!| #1 | \! | \!|_{{#2}}}

\newcommand{\opnorm}[1]{\ensuremath{\matsnorm{#1}{\tiny{\mbox{op}}}}}

\newcommand{\Event}{\ensuremath{\mathcal{E}}}

\newcommand{\UNIFINT}{I^{\mathrm{unif}}_{N}}
\newcommand{\FUNCINT}{I^{\mathrm{func}}_{N}}

\newcommand{\BEINT}{I^{\mathrm{BE}}_{N}}

\newcommand{\newalpha}{\beta}
\newcommand{\Nat}{\ensuremath{\mathbb{N}}}

\begin{document}

\begin{center}

{\bf{\LARGE{Function-Specific Mixing Times and \\ Concentration Away from
  Equilibrium}}}

\vspace*{.2in}

{\large{
\begin{tabular}{c}
Maxim Rabinovich, Aaditya Ramdas \\
Michael I. Jordan, Martin J. Wainwright \\
\texttt{\{rabinovich,aramdas,jordan,wainwrig\}@berkeley.edu}\\
University of California, Berkeley
\end{tabular}
}}

\vspace*{.2in}
\today
\vspace*{.2in}

\begin{abstract}
Slow mixing is the central hurdle when working with Markov chains,
especially those used for Monte Carlo approximations (MCMC).  In many
applications, it is only of interest to estimate the stationary
expectations of a small set of functions, and so the usual definition
of mixing based on total variation convergence may be too
conservative.  Accordingly, we introduce function-specific analogs of
mixing times and spectral gaps, and use them to prove Hoeffding-like
function-specific concentration inequalities.  These results show that
it is possible for empirical expectations of functions to concentrate
long before the underlying chain has mixed in the classical sense, and
we show that the concentration rates we achieve are optimal up to
constants.  We use our techniques to derive confidence intervals that
are sharper than those implied by both classical Markov chain
Hoeffding bounds and Berry-Esseen-corrected CLT bounds.  For
applications that require testing, rather than point estimation, we
show similar improvements over recent sequential testing results for
MCMC. We conclude by applying our framework to real data examples of
MCMC, providing evidence that our theory is both accurate and relevant
to practice.
\end{abstract}

\end{center}


\section{Introduction}

Methods based on Markov chains play a critical role in statistical
inference, where they form the basis of Markov chain Monte Carlo
(MCMC) procedures for estimating intractable expectations~\cite[see,
  e.g.,][]{Gel13BDA,Rob05MCMC}.  In MCMC procedures, it is the
stationary distribution of the Markov chain that typically encodes the
information of interest. Thus, MCMC estimates are asymptotically
exact, but their accuracy at finite times is limited by the
convergence rate of the chain.

The usual measures of convergence rates of Markov chains---namely, the
total variation mixing time or the absolute spectral gap of the
transition matrix~\citep{Lev08Markov}---correspond to very strong
notions of convergence and depend on global properties of the
chain. Indeed, convergence of a Markov chain in total variation
corresponds to uniform convergence of the expectations of all
unit-bounded function to their equilibrium values. The resulting
uniform bounds on the accuracy of
expectations~\citep{Chu12Hoeffding,Gil98Chernoff,Jou10Curvature,Kon14Uniform,Leo04Hoeffding,
  Lez01Berry,Pau12Conc,Sam00Concentration} may be overly pessimistic---not indicative of
the mixing times of specific expectations such as means and variances
that are likely to be of interest in an inferential setting.

Another limitation of the uniform bounds is that they typically assume
that the chain has arrived at the equilibrium distribution, at least
approximately.  Consequently, applying such bounds requires either
assuming that the chain is started in equilibrium---impossible in
practical applications of MCMC---or that the burn-in period is
proportional to the mixing time of the chain, which is also unrealistic, if
not impossible, in practical settings.

Given that the goal of MCMC is often to estimate specific
expectations, as opposed to obtaining the stationary distribution, in
the current paper we develop a function-specific notion of convergence
with application to problems in Bayesian inference. We define a notion
of ``function-specific mixing time,'' and we develop function-specific
concentration bounds for Markov chains, as well as spectrum-based
bounds on function-specific mixing times.  We demonstrate the utility
of both our overall framework and our particular concentration bounds
by applying them to examples of MCMC-based data analysis from the
literature and by using them to derive sharper confidence intervals
and faster sequential testing procedures for MCMC.


\subsection{Preliminaries}

We focus on discrete time Markov chains on $d$ states given by a $d
\times d$ transition matrix $P$ that satisfies the conditions of
irreducibility, aperiodicity, and reversibility.  These conditions
guarantee the existence of a unique stationary distribution $\pi$.
The issue is then to understand how quickly empirical averages of
functions of the Markov chain, of the form $f: [d] \to [0,1]$,
approach the stationary average, denoted by
\begin{align*}
 \mu & \mydefn \E_{X \sim \pi} [f(X)].
\end{align*}

The classical analysis of mixing defines convergence rate in terms of
the total variation distance:
\begin{align}
\label{eq:sup-TV}
d_{\TV} \big (p,~q \big) = \sup_{f \colon \Omega \rightarrow [0,~1]}
\Big| \E_{X\sim p}\big[f(X)\big] - \E_{Y \sim q} \big [ f(Y) \big]
\Big|,
\end{align}
where the supremum ranges over all unit-bounded functions.  The mixing
time is then defined as the number of steps required to ensure that
the chain is within total-variation distance $\delta$ of the
stationary distribution---that is
\begin{align}
\label{EqnDefnClassicalMixing}
T(\delta) & \mydefn \min \Big \{ n \in \Nat \; \mid \; \max_{i \in
  [d]} d_{\TV}\big(\pi_{n}^{(i)},~\pi\big) \leq \delta \Big \},
\end{align}
where $\Nat = \{1, 2, \ldots \}$ denotes the natural numbers,
and $\pi_{n}^{(i)}$ is the distribution of the chain state $X_{n}$
given the starting state $X_{0} = i$.

Total variation is a worst-case measure of distance, and the resulting
notion of mixing time can therefore be overly conservative when the
Markov chain is being used to approximate the expectation of a fixed
function, or expectations over some relatively limited class of
functions.  Accordingly, it is of interest to consider the following
function-specific discrepancy measure:
\begin{defn}[$f$-discrepancy]
For a given function $f$, the $f$-discrepancy is 
\begin{align}
d_{f} \big(p, ~q \big) = \big|\E_{X \sim p} \big [ f \big( X
  \big) \big] - \E_{Y \sim q} \big [ f \big( Y \big) \big ]
\big |.
\end{align}
\end{defn}
\noindent The $f$-discrepancy leads naturally to a function-specific
notion of mixing time:
\begin{defn}[$f$-mixing time] 
For a given function $f$, the $f$-mixing time is 
\begin{align}
\Tf\big(\delta\big) = \min \Big \{ n \in \Nat \; \mid \; \max_{i \in
  [d]} d_{f}\big(\pi_{n}^{(i)},~\pi\big) \leq \delta \Big \} .
\end{align}
\end{defn}
\noindent In the sequel, we also define function-specific notions of
the spectral gap of a Markov chain, which can be used to bound the
$f$-mixing time and to obtain function-specific concentration
inequalities.

\subsection{Related work}

Mixing times are a classical topic of study in Markov chain theory,
and there is a large collection of techniques for their
analysis~\citep[see,
  e.g.,][]{Ald86Shuffling,Dia90StrongStat,Lev08Markov,Mey12Markov,Oll09Ricci,Sin92Multicomm}.
These tools and the results based on them, however, generally apply
only to worst-case mixing times.  Outside of specific examples~\citep{Con06Riffle,Dia15Heisenberg}, 
relatively little is known about mixing with respect to individual functions or limited classes of
functions. Similar limitations exist in studies of concentration of
measure and studies of confidence intervals and other statistical
functionals that depend on tail probability bounds. Existing bounds
are generally uniform, or non-adaptive, and the rates that are
reported include a factor that encodes the global mixing properties of
the chain and does not adapt to the
function~\citep{Chu12Hoeffding,Gil98Chernoff,Jou10Curvature,Kon14Uniform,Leo04Hoeffding,Lez01Berry,Pau12Conc,Sam00Concentration}.
These factors, which do not appear in classic bounds for independent
random variables, are generally either some variant of the spectral
gap $\gamma$ of the transition matrix, or else a mixing time of the
chain $T\big(\delta_{0}\big)$ for some absolute constant $\delta_{0} >
0$. For example, the main theorem from \cite{Leo04Hoeffding} shows
that for a function $f \colon [d] \rightarrow [0,~1]$ and a sample
$X_{0} \sim \pi$ from the stationary distribution, we have
\begin{align}
\label{eq:unif-hoeffding}
\P \big( \big| \frac{1}{N}\sum_{n = 1}^{N} f\big(X_{n}\big) - \mu\big|
\geq \epsilon\big) \leq 2 \exp \Big \{ -\frac{\gamma_{0}}{2\big(2 -
  \gamma_0\big)} \cdot \epsilon^2N \Big \},
\end{align}
where the eigenvalues of $P$ are given in decreasing order as \mbox{$1
  > \lambda_2(P) \geq \cdots \geq \lambda_d(P)$,} and we denote the
spectral gap of $P$ by
\begin{align*}
\gamma_{0} \mydefn \min \big \{ 1 - \lambda_2(P),~1 \big \}.
\end{align*}
The requirement that the chain start in equilibrium can be relaxed by
adding a correction for the burn-in time~\citep{Pau12Conc}. Extensions
of this and related bounds, including bounded-differences-type
inequalities and generalizations to continuous Markov chains and
non-Markov mixing processes have also appeared in the literature
(e.g.,~\cite{Kon14Uniform,Sam00Concentration}).

The concentration result has an alternative formulation in terms of
the mixing time instead of the spectral gap~\citep{Chu12Hoeffding}.
This version and its variants are weaker, since the mixing time can be
lower bounded as
\begin{align}
\label{eq:Tmix-lbd-gamma}
T\big(\delta\big) \geq \big(\frac{1}{\gamma_{\ast}} - 1\big)
\log\big(\frac{1}{2\delta}\big) \geq \big(\frac{1}{\gamma_0} -
1\big)\log\big(\frac{1}{2\delta}\big),
\end{align}
where we denote the absolute spectral
gap~\citep{Lev08Markov} by
\[\gamma_{\ast} \mydefn \min\big(1 - \lambda_2 ,~1 -
\big|\lambda_{d}\big|\big) ~\leq~ \gamma_{0}.\]

In terms of the minimum probability $\pimin \mydefn \min_i \pi_i$, the
corresponding upper bound is an extra factor of
$\log\big(\frac{1}{\pimin}\big)$ larger, which potentially leads to a
significant gap between $\frac{1}{\gamma_{0}}$ and
$T\big(\delta_{0}\big)$, even for a moderate constant such as
$\delta_{0} = \frac{1}{8}$. Similar distinctions arise in our
analysis, and we elaborate on them at the appropriate junctures.


\subsection{Organization of the paper}

In the remainder of the paper, we elaborate on these ideas and apply
them to MCMC. In Section~\ref{sec:main-results}, we state some
concentration guarantees based on function-specific mixing times, as
well as some spectrum-based bounds on $f$-mixing times, and the
spectrum-based Hoeffding bounds they imply.
Section~\ref{sec:stat-apps} is devoted to further development of these
results in the context of several statistical models.  More
specifically, in Section~\ref{subsec:confidence}, we show how our
concentration guarantees can be used to derive confidence intervals
that are superior to those based on uniform Hoeffding bounds and
CLT-type bounds, whereas in Section~\ref{subsec:testing}, we analyze the consequences 
for sequential testing.  In Section~\ref{subsec:practical}, we show that our mixing time and
concentration bounds improve over the non-adaptive bounds in real
examples of MCMC from the literature.  Finally, the bulk of our proofs
are given in Section~\ref{sec:proofs}, with some more technical
aspects of the arguments deferred to the appendices.


\section{Main results}
\label{sec:main-results}

We now present our main technical contributions, starting with a set
of ``master'' Hoeffding bounds with exponents given in terms of
$f$-mixing times. As we explain in
Section~\ref{subsec:derived-hoeffding}, these mixing time bounds can
be converted to spectral bounds bounding the $f$-mixing time in terms
of the spectrum.  (We give some techniques for the latter in
Section~\ref{subsec:mixing-bounds}).

Recall that we use \mbox{$\mu \mydefn \E_{\pi} [f]$} to denote the
mean. Moreover, we follow standard conventions in setting
\begin{align*}
\lambda_{\ast} \mydefn \max \big \{ \lambda_2(P),~ \big|\lambda_{d}(P)
\big| \big \} , \quad \mbox{and} \quad \lambda_{0} \mydefn \max \big
\{ \lambda_2(P),~ 0 \big \}.
\end{align*}
so that the absolute spectral gap and the
(truncated) spectral gap introduced earlier are given by
$\gamma_{\ast} \mydefn 1 - \lambda_{\ast}, \quad \mbox{and} \quad
\gamma_{0} \mydefn 1 - \lambda_{0}.$
In Section~\ref{subsec:mixing-bounds}, we define and analyze
corresponding function-specific quantities, which we introduce as
necessary.

\subsection{Master Hoeffding bound}
\label{subsec:master-hoeffding}

In this section, we present a master Hoeffding bound that provides
concentration rates that depend on the mixing properties of the chain
only through the $f$-mixing time $T_{f}$. The only hypotheses on
burn-in time needed for the bounds to hold are that the chain has been
run for at least $N \geq T_{f}$ steps---basically, so that thinning is
possible---and that the chain was started from a distribution
$\pi_{0}$ whose $f$-discrepancy distance from $\pi$ is small---so that
the expectation of each $f\big(X_n\big)$ iterate is close to
$\mu$---even if its total-variation discrepancy from $\pi$ is
large. Note that the latter requirement imposes only a very mild
restriction, since it can always be satisfied by first running the
chain for a burn-in period of $T_{f}$ steps and then beginning to
record samples.

\begin{thm} 
\label{thm:hoeffding-eps2}
Given any fixed $\epsilon>0$ such that $\df \big( \pi_{0},~\pi \big)
\leq \frac{\epsilon}{2}$ and $N \geq \Tf \big(\frac{\epsilon
}{2}\big)$, we have
\begin{align}
\label{eq:hoeffding-eps2}
\P \Big[ \frac{1}{N}\sum_{n = 1}^{N} f\big(X_n\big) \geq \mu +
  \epsilon \Big] \leq \exp \left \{ -\frac{\epsilon^2
  N}{8\Tf\big(\frac{\epsilon}{2}\big)} \right \}.
\end{align}
\end{thm}


Compared to the bounds in earlier work~\cite[e.g.,][]{Leo04Hoeffding},
the bound~\eqref{eq:hoeffding-eps2} has several distinguishing
features. The primary difference is that the ``effective'' sample size
\begin{subequations}
\begin{align}
\label{EqnEffective}
N_{\eff} & \mydefn \frac{N}{\Tf(\epsilon /2)},
\end{align}
is a function of $f$, which can lead to significantly sharper bounds
on the deviations of the empirical means than the earlier uniform
bounds can deliver. Further, unlike the uniform results, we do not
require that the chain has reached equilibrium, or even approximate
equilibrium, in a total variation sense. Instead, the result applies
provided that the chain has equilibrated only approximately, and only
with respect to $f$.
  
The reader might note that if one actually has access to a
distribution $\pi_{0}$ that is $\epsilon/2$-close to $\pi$ in
$f$-discrepancy, then an estimator of $\mu$ with tail bounds similar
to those guaranteed by Theorem~\ref{thm:hoeffding-eps2} can be
obtained as follows: first, draw $N$ i.i.d.  samples from $\pi_{0}$,
and second, apply the usual Hoeffding inequality for i.i.d. variables.
However, it is essential to realize that
Theorem~\ref{thm:hoeffding-eps2} does not require that such a
$\pi_{0}$ be available to the practitioner. Instead, the theorem
statement is meant to apply in the following way: suppose
that---starting from \emph{any} initial distribution---we run an
algorithm for $N \geq T_f(\epsilon/2)$ steps, and then use the last of
$N - T_f(\epsilon/2)$ samples to form an empirical average.  Our
concentration result then holds with an effective sample size of
\begin{align}
\label{EqnEffectiveWithBurnIn}
N_{\eff}^{\mathrm{burnin}} & \mydefn \frac{N -
  T_f(\epsilon/2)}{\Tf\left(\epsilon/2\right)} ~=~
\frac{N}{\Tf\left(\epsilon/2\right)} - 1.
\end{align}
In other words, the result can be applied with an arbitrary initial
$\pi_{0}$, and accounting for burn-in merely reduces the effective
sample size by one. By contrast, such an interpretation does not
actually hold for the original result of \cite{Leo04Hoeffding}: it
requires an initial sample $X_1 \sim \pi$, but such an exact sample is
not attainable after any finite burn-in period.

The appearance of the function-specific mixing time $\Tf$ in the
bounds comes with both advantages and disadvantages. A notable
disadvantage, shared with the mixing time versions of the uniform
bounds, is that spectrum-based bounds on the mixing time (including
our $f$-specific ones) introduce a $\log\big(\frac{1}{\pimin}\big)$
term that can be a significant source of looseness. On the other hand,
obtaining rates in terms of mixing times comes with the advantage that
any bound on the mixing time translates directly into a version of the
concentration bound (with the mixing time replaced by its upper
bound). Moreover, since the $\pimin^{-1}$ term is likely to be an
artifact of the spectrum-based approach, and possibly even just of the
proof method, it may be possible to turn the mixing time based bound
into a stronger spectrum-based bound with a more sophisticated
analysis. We go part of the way toward doing this, albeit without
completely removing the $\pimin^{-1}$ term.

An analysis based on mixing time also has the virtue of better
capturing the non-asymptotic behavior of the rate. Indeed, as a
consequence of the link~\eqref{eq:Tmix-lbd-gamma} between mixing and
spectral graph (as well as matching upper bounds~\citep{Lev08Markov}),
for any fixed function $f$, there exists a function-specific
spectral-gap $\gamma_{f} > 0$ such that
\begin{align}
\label{eq:Tf-asym-gamma}
\Tf \big(\frac{\epsilon }{2}\big) \approx \frac1{\gamma_{f}} \log
\Big( \frac{1}{\epsilon } \Big) + O \big( 1 \big), \quad \mbox{for}
\quad \epsilon \ll 1.
\end{align}
\end{subequations}
These asymptotics can be used to turn our aforementioned theorem into
a variant of the results of~\citet{Leo04Hoeffding}, in which
$\gamma_{0}$ is replaced by a value $\gamma_{f}$ that (under mild
conditions) is at least as large as $\gamma_{0}$. However, as we
explore in Section~\ref{subsec:practical}, such an asymptotic
spectrum-based view loses a great deal of information needed to deal
with practical cases, where often $\gamma_{f} =\gamma_{0}$ and yet
\mbox{$\Tf(\delta) \ll T(\delta)$} even for very small values of
$\delta > 0$.  For this reason, part of our work is devoted to
deriving more fine-grained concentration inequalities that capture
this non-asymptotic behavior.

By combining our definition~\eqref{EqnEffective} of the effective
sample size $N_{\eff}$ with the asymptotic
expansion~\eqref{eq:Tf-asym-gamma}, we arrive at an intuitive
interpretation of Theorem~\ref{thm:hoeffding-eps2}: it dictates that
the effective sample size scales as $N_{\eff} ~\approx~
\frac{\gamma_{f} N }{ \log(1/\epsilon)}$ in terms of the
function-specific gap $\gamma_f$ and tolerance $\epsilon$.  This
interpretation is backed by the Hoeffding bound derived in
Corollary~\ref{cor:hoeffding-derived-eps2} and it is useful as a
simple mental model of these bounds. On the other hand, interpreting
the theorem this way effectively plugs in the asymptotic behavior of
$T_{f}$ and does not account for the non-asymptotic properties of the
mixing time; the latter may actually be more favorable and lead to
substantially smaller effective sample sizes than the naive asymptotic
interpretation predicts.  From this perspective, the master bound has
the advantage that any bound on $T_{f}$ that takes advantage of
favorable non-asymptotics translates directly into a stronger version
of the Hoeffding bound. We investigate these issues empirically in
Section~\ref{subsec:practical}.

Based on the worst-case Markov Hoeffding
bound~\eqref{eq:unif-hoeffding}, we might hope that the
\mbox{$\Tf(\frac{\epsilon}{2})$} term in
Theorem~\ref{thm:hoeffding-eps2} is spurious and removable using
improved techniques. Unfortunately, it is fundamental. This conclusion
becomes less surprising if one notes that even if we start the chain
in its stationary distribution and run it for \mbox{$N <
  \Tf(\epsilon)$} steps, it may still be the case that there is a
large set $\Omega_{0}$ such that for $i \in \Omega_0$ and $1 \leq n
\leq N$,
\begin{align}
\label{eq:lower-bd-intuition}
\left|f(X_{n}) - \mu\right| \gg \epsilon~~ \text{a.s. if}~ X_{0} = i .
\end{align}
This behavior is made possible by the fact that large positive and
negative deviations associated with different values in $\Omega_{0}$
can cancel out to ensure that $\E\left[f\left(X_{n}\right)\right] =
\mu$ marginally. However, the lower
bound~\eqref{eq:lower-bd-intuition} guarantees that
\begin{align*}
\P\left(\frac{1}{N}\sum_{n = 1}^{N} f\left(X_{n}\right) \geq \mu +
\epsilon\right) & \geq \sum_{i \in \Omega_0} \pi_{i} \cdot
\P\left(\frac{1}{N}\sum_{n = 1}^{N} f\left(X_{n}\right) \geq \mu +
\epsilon~|~X_{0} = i\right) \\ & \geq \pi\left(\Omega_0\right),
\end{align*}
so that if $\pi\left(\Omega_0\right) \gg 0$, we have no hope of
controlling the large deviation probability unless $N \gtrsim
\Tf\left(\epsilon\right)$. We make this intuitive argument precise in
Section~\ref{sec:lower}.


\subsection{Bounds on $f$-mixing times}
\label{subsec:mixing-bounds}

We generally do not have direct access either to the mixing time
$T\big(\delta\big)$ or the $f$-mixing time
$\Tf\big(\delta\big)$. Fortunately, any bound on $\Tf$ translates
directly into a variant of the tail bound~\eqref{eq:hoeffding-eps2}.
Accordingly, this section is devoted to methods for bounding these
quantities. Since mixing time bounds are equivalent to bounds on
$d_{\TV}$ and $d_{f}$, we frame the results in terms of distances
rather than times.  These results can then be inverted in order to
obtain mixing-time bounds in applications.

The simplest bound is simply a uniform bound on total variation
distance, which also yields a bound on the $f$-discrepancy.  In
particular, if the chain is started with distribution $\pi_0$, then we
have
\begin{align}
\label{eq:mix-TV}
d_{\TV}\big(\pi_{n},~\pi\big) \leq \frac{1}{\sqrt{\pimin}} \cdot
\abslambda^{n} \cdot d_{\TV}\big(\pi_{0},~\pi\big).
\end{align}
In order to improve upon this bound, we need to develop
function-specific notions of spectrum and spectral gaps.  The simplest
way to do this is simply to consider the (left) eigenvectors to which
the function is not orthogonal and define a spectral gap restricted
only to the corresponding eigenvectors.

\begin{defn}[$f$-eigenvalues and spectral gaps] 
For a function \mbox{$f \colon [d] \rightarrow \R$,} we define
\begin{subequations}
\begin{align}
J_{f} & \mydefn \Big \{ j \in [d] \, \mid \, \lambda_{j} \neq 1
~\text{and}~ q_{j}^{T}f \neq 0 \Big \},
\end{align}
where $q_{j}$ denotes a left eigenvector associated with
$\lambda_{j}$.  Similarly, we define
\begin{align}
\lambda_{f} = \max_{j \in J_{f}} \big|\lambda_{j}\big|, \quad
\mbox{and} \quad \gamma_{f} = 1 - \lambda_{f}.
\end{align}
\end{subequations}
\end{defn}

Using this notation, it is straightforward to show that if the chain
is started with the distribution $\pi_{0}$, then
\begin{align}
\label{eq:mix-gap}
 d_{f} \big( \pi_{n},~\pi\big) \leq \sqrt{ \frac{\E_{\pi}\big[f^2
       \big]}{\pimin}} \cdot \lambda_{f}^{n} \cdot
 d_{f}\big(\pi_{0},~\pi\big) .
\end{align}
This bound, though useful in many cases, is also rather brittle: it
requires $f$ to be exactly orthogonal to the eigenfunctions of the
transition matrix.  For example, a function $f_{0}$ with a good value
of $\lambda_{f}$ can be perturbed by an arbitrarily small amount in a
way that makes the resulting perturbed function $f_{1}$ have
$\lambda_{f} = \lambda_{\ast}$. More broadly, the bound is of little
value for functions with a small but nonzero inner product with the
eigenfunctions corresponding to large eigenvalues (which is likely to
occur in practice; cf.\ Section \ref{subsec:practical}), or in
scenarios where $f$ lacks symmetry (cf.\ the random function example
in Section~\ref{subsec:example-cycle}).

In order to address these issues, we now derive a more fine-grained
bound on $\df$. The basic idea is to split the lower $f$-spectrum
$J_{f}$ into a ``bad'' piece $J$, whose eigenvalues are close to $1$
but whose eigenvectors are approximately orthogonal to $f$, and a
``good'' piece $J_{f} \setminus J$, whose eigenvalues are far from $1$
and which therefore do not require control on the inner products of
their eigenvectors with $f$.  More precisely, for a given set $J
\subset J_{f}$, let us define
\begin{align*}
 \Delta_{J}^{\ast} \mydefn 2 \big |J \big| \times \max_{j \in J} \|
 h_{j} \|_{\infty} \times \max_{j \in J} \big | q_{j}^{T} f \big|, &
 \qquad \lambda_{J} \mydefn \max \Big \{ \big|\lambda_{j}\big| \, \mid
 \, j \in J \Big \}, \quad \mbox{and} \\
\lambda_{-J} \mydefn \max \Big \{ \big|\lambda_{j}\big| \; \mid \; j
\in J_{f} \setminus J \Big \}.
\end{align*}
We obtain the following bound, expressed in terms of $\lambda_{-J}$
and $\lambda_{J}$, which we generally expect to obey the relation $1 -
\lambda_{-J} \ll 1 - \lambda_{J}$.

\begin{lem}[Sharper $f$-discrepancy bound]
Given $f \colon [d] \rightarrow [0,~1]$ and a subset $J \subset
J_{f}$, we have
\begin{align}
\df \big(\pi_{n},~\pi \big ) & \leq \Delta^{\ast}_J \: \lambda_{J}^{n}
\cdot d_{\TV}(\pi_0,~\pi) + \sqrt{\frac{\E_{\pi}\big [ f^2 \big]
  }{\pimin}} \cdot \lambda_{-J}^{n} \, \df( \pi_{0}, ~\pi).
\end{align}
\label{lem:mix-gap-all}
\end{lem}

The above bound, while easy to apply and comparatively easy to
estimate, can be loose when the first term is a poor
estimate of the part of the discrepancy that comes from the $J$ part
of the spectrum. We can get a still sharper estimate by instead 
making use of the following vector quantity that 
more precisely summarizes the interactions between $f$ and $J$:
\begin{align*}
h_{J}\big(n\big) \mydefn \sum_{j \in J} \big(q_{j}^{T}f \cdot
\lambda_{j}^{n}\big) h_{j}.
\end{align*}
This quantity leads to what we refer to as an \emph{oracle adaptive
  bound}, because it uses the exact value of the part of the
discrepancy coming from the $J$ eigenspaces, while using the same
bound as above for the part of the discrepancy coming from
$J_{f}\setminus J$.

\begin{lem}[Oracle $f$-discrepancy bound]
\label{lem:mix-gap-all-oracle}
Given \mbox{$f \colon [d] \rightarrow [0,~1]$} and a subset \mbox{$J
  \subset J_{f}$,} we have
\begin{align}
\df \big(\pi_{n},~\pi \big) & \leq \big|\big(\pi_0 -\pi\big)^{T} h_{J}
\big (n \big ) \big| + \sqrt{\frac{\E_{\pi}\big [f^2 \big] }{\pimin}}
\cdot \lambda_{-J}^{n} \cdot \df \big( \pi_{0}, ~\pi \big).
\end{align}
\end{lem}

We emphasize that, although Lemma \ref{lem:mix-gap-all-oracle} is
stated in terms of the initial distribution $\pi_{0}$, when we apply
the bound in the real examples we consider, we replace all
quantities that depend on $\pi_0$ by their worst cases values, in
order to avoid dependence on initialization; this results in a $\|
h_{J} \big ( n \big) \|_{\infty}$ term instead of the dot product in
the lemma.


\subsection{Concentration bounds}
\label{subsec:derived-hoeffding}

The mixing time bounds from Section~\ref{subsec:mixing-bounds} allow
us to translate the master Hoeffding bound into a weaker but more
interpretable---and in some instances, more directly
applicable---concentration bound. The first result we prove along
these lines applies meaningfully only to functions $f$ whose absolute
$f$-spectral gap $\gamma_{f}$ is larger than the absolute spectral gap
$\gamma_{\ast}$. It is a direct consequence of the master Hoeffding
bound and the simple spectral mixing bound \eqref{eq:mix-gap}, and it
delivers the asymptotics in $N$ and $\epsilon$ promised in
Section~\ref{subsec:master-hoeffding}.

\begin{cor}
\label{cor:hoeffding-derived-eps2}
Given any $\epsilon > 0$ such that $d_{f} \big( \pi_{0},~\pi\big) \leq
\frac{\epsilon}{2}$ and $N \geq \Tf \big (\frac{\epsilon}{2}\big)$, we
have
\begin{align*}
\P \left[ \frac{1}{N} \sum_{n = 1}^{N} f(X_n) \geq \mu + \epsilon
  \right] & \leq
\begin{cases}
 \exp \left ( - \frac{\epsilon^2}{8} \: \frac{\gamma_{f} N}{ \log
   \big(\frac{2}{\epsilon \sqrt{\pimin}} \big) } \right ) &
 \text{ if } \epsilon \leq \frac{2\lambda_f}{\sqrt{\pimin}},\\
\exp \left ( - \frac{\epsilon^2  N}{8} \right ) & \text{ otherwise.}
\end{cases}
\end{align*}
\end{cor}

Deriving a Hoeffding bound using the sharper $f$-mixing bound given in
Lemma~\ref{lem:mix-gap-all} requires more care, both because of the
added complexity of managing two terms in the bound and because one of
those terms does not decay, meaning that the bound only holds for
sufficiently large deviations $\epsilon > 0$.

The following result represents one way of articulating the bound
implied by Lemma~\ref{lem:mix-gap-all}; it leads to improvements over
the previous two results when the contribution from the bad part of
the spectrum $J$---that is, the part of the spectrum that brings
$\gamma_{f}$ closer to $1$ than we would like---is negligible at the
scale of interest. Recall that Lemma~\ref{lem:mix-gap-all} expresses
the contribution of $J$ via the quantity $\Delta_{J}^{\ast}$.

\begin{cor} 
\label{cor:hoeffding-derived-Jf}
Given a triple of positive numbers $(\Delta, \Delta_J,
\Delta_{J}^\ast)$ such that $\Delta_{J} \geq \Delta_{J}^{\ast}$ and $N
\geq \Tf \big(\Delta_{J} + \Delta\big)$, we have
\begin{align}
\P \left[ \frac{1}{N}\sum_{n = 1}^{N} f\big(X_n\big) \geq \mu + 
  2\left(\Delta_{J} + \Delta\right)\right] & \leq
\begin{cases}        
   \exp\big(- \frac{\big(\Delta_{J} + \Delta\big)^{2}}{2} \:
   \frac{\big(1 - \lambda_{-J}\big)N}{ \log\big(\frac{1}{\Delta
       \sqrt{\pimin}}\big) }\big) \text{ if } \Delta \leq
   \frac{\lambda_{-J}}{\sqrt{\pimin}}, \\ 
\exp \big(-\frac{\big(\Delta_{J} + \Delta\big)^{2}N}{2}\big) \text{ if }
\Delta > \frac{\lambda_{-J}}{\sqrt{\pimin}}.
 \end{cases}
\end{align}
\end{cor}

Similar arguments can be applied to combine the master Hoeffding
bounds with the oracle $f$-mixing bound
Lemma~\ref{lem:mix-gap-all-oracle}, but we omit the corresponding
result for the sake of brevity. The proofs for both aforementioned corollaries are in Section \ref{subsec:proofs-conc-derived}.


\subsection{Example: Lazy random walk on $C_{2d}$}
\label{subsec:example-cycle}

In order to illustrate the mixing time and Hoeffding bounds from
Section~\ref{subsec:mixing-bounds}, we analyze their predictions for
various classes of functions on the $2d$-cycle $C_{2d}$, identified
with the integers modulo $2d$. In particular, consider the Markov
chain corresponding to a lazy random walk on $C_{2d}$; it has
transition matrix
\begin{align}
\label{eq:P-C2d}
P_{uv} = \begin{cases} \frac{1}2 &~ \text{if}~ v = u, \\ \frac{1}{4}
  &~ \text{if}~ v = u + 1 \mod{2d}, \\ \frac{1}{4} &~ \text{if}~ v = u
  - 1 \mod{2d}, \\ 0 &~ \text{otherwise.}
\end{cases}
\end{align}

It is easy to see that the chain is irreducible, aperiodic, and
reversible, and its stationary distribution is uniform. It can be
shown~\citep{Lev08Markov} that its mixing time scales proportionally
\mbox{to $d^2$.} However, as we now show, several interesting classes
of functions mix much faster, and in fact, a ``typical'' function,
meaning a randomly chosen one, mixes much faster than the naive mixing
bound would predict.

\paragraph{Parity function.} The epitome of a rapidly mixing function is the parity function:
\begin{align}
\parity (u)) \mydefn \begin{cases} 
1 &~ \text{if}~ u~ \text{is odd}, \\
0 &~ \text{otherwise.} 
\end{cases}
\end{align}
It is easy to see that no matter what the choice of initial
distribution $\pi_0$ is, we have $\E\big[\parity(X_1)\big] = \frac{1}2
$, and thus $\parity$ mixes in a single step.

\paragraph{Periodic functions.} A more general class of examples arises from considering the
eigenfunctions of $P$, which are given by $g_{j}\big(u\big) =
\cos\big(\frac{\pi j u}{d}\big)$; \citep[see, e.g.,][]{Lev08Markov}.
We define a class of functions of varying regularity by setting
\begin{align*}
f_{j} = \frac{1 + g_{j}}2 , \quad \mbox{for each $j = 0, 1, \ldots,
  d$.}
\end{align*}
Here we have limited $j$ to $0 \leq j \leq d$ because $f_{j}$ and $f_{2d -
  j}$ behave analogously.  Note that the parity function $\parity$
corresponds to $f_d$.

Intuitively, one might expect that some of these functions mix well
before $d^2$ steps have elapsed---both because the vectors $\{f_j, \;
j \neq 1 \}$ are orthogonal to the non-top eigenvectors with
eigenvalues close to $1$ and because as $j$ gets larger, the periods
of $f_{j}$ become smaller and smaller, meaning that their global
behavior can increasingly be well determined by looking at local
snapshots, which can be seen in few steps.

Our mixing bounds allow us to make this intuition precise, and our
Hoeffding bounds allow us to prove correspondingly improved
concentration bounds for the estimation of $\mu =
\E_{\pi}\big[f_{j}\big] = 1/2$. Indeed, we have
\begin{align}
\label{eq:gamma-trig}
\gamma_{f_j} = \frac{1 - \cos\big(\frac{\pi j}{d}\big)}2 
\geq 
\begin{cases} \frac{\pi^2 j^2 }{24d^2 } & \text{if}~ j \leq
  \frac{d}2 , \\
\frac{1}2  & \text{if}~ \frac{d}2  < j \leq d.
\end{cases}
\end{align}
Consequently, equation~\eqref{eq:mix-gap} predicts that
\begin{align}
\label{eq:Tf-trig}
T_{f_j} \big(\delta\big) \leq \tilde{T}_{f_j}\big(\delta\big) =
\begin{cases}
\frac{24}{\pi^2 } \big[\frac{1}2 \log 2d+ \log \big(\frac{1}{\delta}
  \big) \big] \cdot \frac{d^2 }{j^2 } &~ \text{if}~ j \leq
\frac{d}2 , \\
\log 2d + 2 \log\big
(\frac{1}{\delta}\big) &~ \text{if}~ \frac{d}2 < j \leq d,
\end{cases} 
\end{align}
where we have used the trivial bound $\E_{\pi}\big[f^2 \big] \leq 1$
to simplify the inequalities. Note that this yields an improvement
over $\asymp d^2 $ for $j \gtrsim \log{d}$. Moreover, the
bound~\eqref{eq:Tf-trig} can itself be improved, since each $f_{j}$ is
orthogonal to all eigenfunctions other than $\ones$ and $g_{j}$, so
that the $\log{d}$ factors can all be removed by a more carefully
argued form of Lemma~\ref{lem:mix-gap-all}. It thus follows directly
from the bound~\eqref{eq:gamma-trig} that if we draw $N +
\tilde{T}_{f_j}\big(\frac{\epsilon}{2}\big)$ samples, we obtain the
tail bound
\begin{align}
\P \Big[ \frac{1}{N_0} \sum_{n = N_{\mathrm{b}}}^{N +
    N_{\mathrm{b}}} f_{j}\big(X_n\big) \geq \frac{1}2  +
  \epsilon \Big] & \leq
\begin{cases}
\exp \big(-\frac{3 d^2 }{\pi^2  j^2 } \cdot \frac{\epsilon^2  N}{
  \log\big(2\sqrt{2d}/\epsilon\big)}\big) &~ \text{if}~ j \leq
\frac{d}2 , \\
\exp \big(-\frac{\epsilon^2 N}{16 \log \big(
    2\sqrt{2d}/\epsilon \big)} \big) &~ \frac{d}2 < j
\leq d ,
\end{cases}
\end{align}
where the burn-in time is given by $N_{\mathrm{b}} =
\tilde{T}_{f_j}\big(\epsilon / 2\big)$.  Note again that the
sharper analysis mentioned above would allow us to remove the
$\log{2d}$ factors.

\paragraph{Random functions.} 
A more interesting example comes from considering a randomly chosen
function $f \colon C_{2d} \rightarrow [0,~1]$. Indeed, suppose that
the function values are sampled iid from some distribution $\nu$ on
$[0,~1]$ whose mean $\mu^{\ast}$ is $1/2$:
\begin{align}\label{eq:f-random}
\{ f(u), ~u \in C_{2d} \} ~ \overset{\text{iid}}{\sim}  ~ \nu . 
\end{align}
We can then show that for any fixed $\delta^{\ast} > 0$, with high
probability over the randomness of $f$, have
\begin{align}
\label{eq:Tf-random}
T_{f} (\delta) & \lesssim \frac{d \log{d} \big[\log{d} + \log \big(
    \frac{1}{\delta} \big) \big]}{\delta^2 }, \qquad \mbox{for all
  $\delta \in (0, \delta^\ast]$.}
\end{align}
For $\delta \gg \frac{\log{d}}{\sqrt{d}}$, this scaling is an
improvement over the global mixing time of order $d^2 \log(1/\delta)$.

The core idea behind the proof of equation~\eqref{eq:Tf-random} is to
apply Lemma~\ref{lem:mix-gap-all} with
\begin{align}
\label{eq:J-random}
J_{\delta} & \mydefn \left\{ j \in \Nat \cap [1, 2d-1] \; \mid \; j
\leq 4 \delta \sqrt{\frac{d}{\log{d}}}~~ \text{or} ~~ j \geq 2d - 4
\delta \sqrt{\frac{d}{\log{d}}} \right\} .
\end{align}
It can be shown that $\| h_{j}\|_{\infty} = 1$ for all $0 \leq j < 2d$
and that with high probability over $f$, $|q_{j}^{T}f| \lesssim
\sqrt{\frac{\log{d}}{d}}$ simultaneously for all $j \in J_{\delta}$,
which suffices to reduce the first part of the sharper $f$-discrepancy
bound to order $\delta$.

In order to estimate the rate of concentration, we proceed as
follows. Taking $\delta = c_{0} \epsilon$ for a suitably chosen
universal constant $c_{0} > 0$, we show that $\Delta_{J} \mydefn
\frac{\epsilon}{4} \geq \Delta_{J}^{\ast}$. We can then set $\Delta =
\frac{\epsilon}{4}$ and observe that with high probability over $f$,
the deviation in Corollary~\ref{cor:hoeffding-derived-Jf} satisfies
the bound $ 2\left(\Delta_{J} + \Delta\right) \leq \epsilon$.  With
$\delta$ as above, we have $1 - \lambda_{-J} \geq
\frac{c_{1}\epsilon^{2}}{d\log{d}}$ for another universal constant
$c_{1} > 0$. Thus, if we are given $N + \Tf\big(\epsilon/2\big)$
samples for some $N \geq \Tf\big(\frac{\epsilon}{2}\big)$, then we
have
\begin{align}
\label{eq:conc-random}
\P \left[ \frac{1}{N} \sum_{n = \Tf(\epsilon/2)}^{N + \Tf (
    \epsilon/2)} f(X_n) \geq \mu + \epsilon \right] \leq \exp
\left \{ -\frac{c_2 \epsilon^{4}N}{d \log{d }\big[ \log \big(
    \frac{4}{\epsilon} \big) + \log{2d} \big ] } \right \},
\end{align}
for some $c_2 > 0$.  Consequently, it suffices for the sample size to be
lower bounded by
\begin{align*}
N \gtrsim \frac{d \log{d}\big[\log\big(1/\epsilon\big) +
    \log{d}\big]}{\epsilon^{4}},
\end{align*}
in order to achieve an estimation accuracy of $\epsilon$.  Notice that
this requirement is an improvement over the $\frac{d^2 }{\epsilon^2 }$
from the uniform Hoeffding bound provided that $\epsilon \gg (\frac{
  \log^2 {d}} {d})^{1/2}$.  Proofs of all these claims can be found in
Appendix~\ref{app:proof-example-cycle}.


\subsection{Lower bounds}
\label{sec:lower}

Let us now make precise the intuitive argument set forth at the end of
Section~\ref{subsec:master-hoeffding}.  The basic idea is to start
with an arbitrary candidate function $\delta : \left(0,~1\right)
\rightarrow \left(0,~1\right)$ such that
$\Tf\left(\frac{\epsilon}{2}\right)$ in the denominator of the
function-specific Hoeffding bound \eqref{eq:hoeffding-eps2} can be
replaced by $\Tf\left(\delta\left(\epsilon\right)\right)$ and show
that if $\delta\left(\epsilon\right) \geq \epsilon$, the replacement
is not actually possible. We prove this fact by constructing a Markov
chain (which is independent of $\epsilon$) and a function (which
depends on both $\epsilon$ and $\delta$) such that the Hoeffding bound
is violated for the Markov chain-function pair for some value of $N$
(which in general depends on the chain and $\epsilon$).

As the following precise result shows, our lower bound continues to
hold for an arbitrary constant in the exponent of the Hoeffding bound,
meaning that Theorem~\ref{thm:hoeffding-eps2} is optimal up to
constants. We give the proof in Section~\ref{subsec:proofs-lower}.

\begin{prop}
\label{prop:lower-bound}
For every constant $c_{1} > 0$ and $\epsilon \in (0,1)$, there exists
a Markov chain $P_{c_1}$, a number of steps $N = N(c_1,\epsilon)$ and
a function $f = f_{\epsilon}$ such that
\begin{align}
\P_{\pi}\left(\left|\frac{1}{N} \sum_{n = 1}^{N} f(X_{n}) -
\frac{1}{2}\right| \geq \epsilon\right) > 2 \cdot \exp \left(
-\frac{c_{1} N \epsilon^{2}}{\Tf \left (\delta
  (\epsilon)\right)}\right).
\end{align}
\end{prop}


\section{Statistical applications}
\label{sec:stat-apps}

We now consider how our results apply to Markov chain Monte Carlo
(MCMC) in various statistical settings.  Our investigation proceeds
along three connected avenues.  We begin by showing, in
Section~\ref{subsec:confidence}, how our concentration bounds can be
used to provide confidence intervals for stationary expectations that
avoid the over-optimism of pure CLT predictions without incurring the
prohibitive penalty of the Berry-Esseen correction---or the global
mixing rate penalty associated with spectral-gap-based confidence
intervals. Then, in Section~\ref{subsec:testing}, we show how our results
allow us to improve on recent sequential hypothesis testing
methodologies for MCMC, again replacing the dependence on the spectral
gap by a dependence on the $f$-mixing time. Later, in
Section~\ref{subsec:practical}, we illustrate the practical
significance of function-specific mixing properties by using our
framework to analyze three real-world instances of MCMC, basing both
the models and datasets chosen on real examples from the literature.


\subsection{Confidence intervals for posterior expectations}
\label{subsec:confidence}

In many applications, a point estimate of $\E_{\pi}\big[f\big]$ does
not suffice; the uncertainty in the estimate must be quantified, for
instance by providing $(1-\alpha)$ confidence intervals for some pre-specified constant $\alpha$.  In this section, we
discuss how improved concentration bounds can be used to obtain
sharper confidence intervals. In all cases, we assume the Markov chain
is started from some distribution $\pi_0$ that need not be the
stationary distribution, meaning that the confidence intervals must
account for the burn-in time required to get close to equilibrium.

We first consider a bound that is an immediate consequence of the
uniform Hoeffding bound given by~\cite{Leo04Hoeffding}. As one would
expect, it gives contraction at the usual Hoeffding rate but with an
effective sample size of $N_{\eff} \approx \gamma_{0} (N - T_{0})$,
where $T_{0}$ is the tuneable burn-in parameter. Note that this means
that no matter how small $T_{f}$ is compared to the global mixing time
$T$, the effective size incurs the penalty for a global burn-in and
the effective sample size is determined by the global spectral
parameter $\gamma_{0}$.  In order to make this precise, for a fixed
burn-in level $\alpha_0 \in (0, \alpha)$, define
\begin{subequations}
\begin{align}
\epsilon_{N}( \alpha,~\alpha_{0}) & \mydefn \sqrt{2\big(2 -
  \gamma_{0}\big)} \cdot \sqrt{\frac{\log\big(2/\big[\alpha -
      \alpha_0\big]\big)}{\gamma_{0}\big[N -
      T\big(\alpha_{0}\big)\big]}}.
\end{align}
Then the uniform Markov Hoeffding bound~\cite[Thm. 1]{Leo04Hoeffding}
implies that the set
\begin{align}
\label{eq:confidence-unif-hoeffding}
\UNIFINT \big( \alpha,~\alpha_{0} \big) = \left[ \frac{1}{N - T \big(
    \alpha_0/2 \big)} \sum_{n = T \big( \alpha_{0}/ 2 \big) + 1}^{N} f
  \big( X_{n} \big) \pm \epsilon_{N} \big( \alpha, ~\alpha_{0}\big)
  \right]
\end{align}
\end{subequations}
is a $1-\alpha$ confidence interval.  Full details of the proof are
given in Appendix~\ref{appsec:confidence-unif-hoeffding}.

Moreover, given that we have a family of confidence intervals---one
for each choice of $\alpha_0 \in (0, \alpha)$---we can obtain the
sharpest confidence interval by computing the infimum
$\epsilon_{N}^{\ast}\big(\alpha\big) \mydefn \inf \limits_{0 <
  \alpha_{0} < \alpha} \epsilon_{N}\big(\alpha,~\alpha_{0}\big)$.
Equation~\eqref{eq:confidence-unif-hoeffding} then implies that
\begin{align*}
\UNIFINT\big(\alpha\big) = \Big[\frac{1}{N - T
    \big(\alpha_0\big)}\sum_{n = T(\alpha_{0}/2) + 1}^{N} f (X_{n})
  \pm \epsilon_{N}^{\ast}( \alpha) \Big]
\end{align*}
is a $1 - \alpha$ confidence interval for $\mu$.

We now consider one particular application of our Hoeffding bounds to
confidence intervals, and find that the resulting interval adapts to
the function, both in terms of burn-in time required, which now falls
from a global mixing time to an $f$-specific mixing time, and in terms
of rate, which falls from $\frac{1}{\gamma_{0}}$ to $T_{f}(\delta)$
for an appropriately chosen $\delta > 0$.  We first note that the
one-sided tail bound of Theorem~\ref{thm:hoeffding-eps2} can be
written as $e^{-r_{N}(\epsilon)/8}$, where
\begin{align}
\label{eq:rN-def}
r_{N} (\epsilon) & \mydefn \epsilon^2 \left [ \frac{N}{T_{f} \big(
  \frac{\epsilon}{2} \big)} - 1 \right ].
\end{align}
If we wish for each tail to have probability mass that is at most
$\alpha/2$, we need to choose $\epsilon > 0$ so that $r_{N} \big (
\epsilon \big) \geq 8 \log\frac2 {\alpha}$, and conversely any such
$\epsilon$ corresponds to a valid two-sided $\big(1 - \alpha\big)$
confidence interval.  Let us summarize our conclusions:
\begin{thm}
\label{thm:confidence-adaptive-hoeffding} 
For any width $\epsilon_{N} \in r_{N}^{-1}\big(\big[8
  \log\big(2/\alpha\big),~\infty\big)\big)$, the set
\begin{align*}
\FUNCINT & \mydefn \left[ \frac{1}{N - T_{f}\big(\frac{\epsilon
    }{2} \big)} \sum_{n = T_{f} \big( \frac{\epsilon}{2}
    \big)}^{N} f\big(X_{n}\big) \pm \epsilon_{N} \right]
\end{align*}
is a $1 - \alpha$ confidence interval for the mean $\mu =
\E_{\pi}\big[f\big]$.
\end{thm}

In order to make the result more amenable to interpretation, first
note that for any $0 < \eta < 1$, we have
\begin{align}
\label{eq:rN-eta-def}
r_{N}\big(\epsilon\big) \geq \underbrace{\epsilon^2
  \left[\frac{N}{T_{f}\big(\frac{\eta}{2}\big)} - 1\right]}_{
  r_{N,\eta}(\epsilon)} \quad \mbox{valid for all $\epsilon \geq
  \eta$.}
\end{align}
Consequently, whenever $r_{N,\eta}(\epsilon_{N}) \geq 8 \log\frac2
{\alpha}$ and $\epsilon_{N} \geq \eta$, we are guaranteed that a
symmetric interval of half-width $\epsilon_{N}$ is a valid $\big(1 -
\alpha\big)$-confidence interval.  Summarizing more precisely, we
have:
\begin{cor}
\label{cor:confidence-adaptive-hoeffding-concrete} 
Fix $\eta > 0$ and let
\begin{align*}
\epsilon_{N} = r_{N,\eta}^{-1}\big(8 \log \frac2 {\alpha}\big) & = 2\sqrt{2}
\sqrt{\frac{T_{f}\big(\frac{\eta}{2}\big) \cdot \log\big (2/
    \alpha \big)}{N - T_{f} \big( \frac{\eta}{2} \big)}} .
\end{align*}
If $N \geq \Tf\big(\frac{\eta}{2}\big)$, then $\FUNCINT$ is a $1 -
\alpha$ confidence interval for $\mu = \E_{\pi}\big[f\big]$.
\end{cor}

Often, we do not have direct access to $T_{f}\big(\delta\big)$, but we
can often obtain an upper bound $\tilde{T}_{f}\big(\delta\big)$ that
is valid for all $\delta > 0$. In
Section~\ref{subsec:proofs-confidence}, therefore, which contains the
proofs for this section, we prove a strengthened form of Theorem
\ref{thm:confidence-adaptive-hoeffding} and its corollary in that
setting.

A popular alternative strategy for building confidence intervals using
MCMC depends on the Markov central limit theorem
(e.g.,~\citep{Fle08Trust,Jon01Honest,Gly09BatchMeans,Rob05MCMC}). If
the Markov CLT held exactly, it would lead to appealingly simple
confidence intervals of width
\begin{align*}
\tilde{\epsilon}_{N} = \sigfasym \: \sqrt{ \frac{\log (2/\alpha)}{N}},
\end{align*}
where $\sigfasym^2 \mydefn \lim_{N \rightarrow \infty} \frac{1}{N}
\Var_{X_{0} \sim \pi} \big [ \sum_{n = 1}^{N} f \big( X_{n} \big)
  \big]$ is the asymptotic variance of $f$.

Unfortunately, the CLT does not hold exactly, even after the burn-in
period. The amount by which it fails to hold can be quantified using a
Berry-Esseen bound for Markov chains, as we now discuss. Let us adopt
the compact notation
$\tilde{S}_{N} = \sum_{n = 1}^{N} \big[ f \big( X_{n} \big) - \mu
  \big].$
We then have the bound~\citep{Lez01Berry}
\begin{align}
\label{eq:markov-clt-be}
\big| \P \big( \frac{ \tilde{S}_{N}}{\sigfasym \sqrt{N}} \leq s \big)
- \Phi \big(s\big)\big| \leq \frac{e^{- \gamma_{0} N } }{3
  \sqrt{\pimin}} + \frac{13}{\sigfasym \sqrt{\pimin}} \cdot
\frac{1}{\gamma_{0}\sqrt{N}},
\end{align}
where $\Phi$ is the standard normal CDF.  Note that this bound
accounts for both the non-stationarity error and for non-normality
error at stationarity.  The former decays rapidly at the rate $e^{-\gamma_{0}N}$, while the latter decays far more slowly, at the rate
$\frac{1}{\gamma_{0}\sqrt{N}}$.

While the bound~\eqref{eq:markov-clt-be} makes it possible to prove a
corrected CLT confidence interval, the resulting bound has two
significant drawbacks. The first is that it only holds for extremely
large sample sizes, on the order of $\frac{1}{\pimin\gamma_{0}^2 }$,
compared to the order $\frac{\log\big(1/\pimin\big)}{\gamma_{0}}$
required by the uniform Hoeffding bound. The second, shared by the
uniform Hoeffding bound, is that it is non-adaptive and therefore
bottlenecked by the global mixing properties of the chain.  For
instance, if the sample size is bounded below as
\begin{align*}
N \geq \max\big(\frac{1}{\gamma_{0}}\log\big(\frac2 {\sqrt{\pimin}\alpha}\big) 
 ~, ~\frac{1}{\gamma_{0}^2 }\frac{6084}{\sigfasym^2 \pimin \alpha^2 } \big),
\end{align*}
then both terms of equation~\eqref{eq:confidence-unif-hoeffding} are
bounded by $1/6$, and the confidence intervals take the form
\begin{align}
\label{eq:confidence-clt-be} 
\BEINT = \left[\frac{1}{N}\sum_{n = 1}^{N} f\big(X_{n}\big) \pm
  \sigfasym \sqrt{\frac{2 \log \big (6 / \alpha \big)} {N}} \right].
\end{align}
See Appendix~\ref{appsec:confidence-clt-be} for the justification of
this claim.

It is important to note that the width of this confidence interval
involves a hidden form of mixing penalty. Indeed, defining the
variance $\sigf^2 = \Var_{\pi}\big[f\big(X\big)\big]$ and $\rho_{f}
\mydefn \frac{\sigf^2 }{\sigfasym^2 }$, we can rewrite the width as
\begin{align*}
\epsilon_{N} & = \sigf \sqrt{ \frac{2 \log \big ( 6 /\alpha \big)}{
    \rho_{f} N} }.
\end{align*}
Thus, for this bound, the quantity $\rho_{f}$ captures the penalty due
to non-independence, playing the role of $\gamma_{0}$ and $\gamma_{f}$
in the other bounds. In this sense, the CLT bound adapts to the
function $f$, but only when it applies, which is at a sample-size
scale dictated by the global mixing properties of the chain (i.e.,
$\gamma_{0}$).



\subsection{Sequential testing for MCMC}
\label{subsec:testing}

For some applications, full confidence intervals may be unnecessary;
instead, a practitioner may merely want to know whether $\mu =
\E_{\pi}[f]$ lies above or below some threshold $0 < r < 1$. In these
cases, we would like to develop a procedure for distinguishing between
the two possibilities, at a given tolerable level $0 < \alpha < 1$ of
combined Type I and II error. The simplest approach is, of course, to
choose $N$ so large that the $1 - \alpha$ confidence interval built
from $N$ MCMC samples lies entirely on one side of $r$, but it may be
possible to do better by using a sequential test. This latter idea was
recently investigated in \citet{Gyo15Test}, and we consider the same
problem settings that they did:
\begin{enumerate}[leftmargin=*]
\item[(a)] Testing with (known) indifference region, involving a
  choice between
\begin{align*}
H_{0} & \colon \mu \geq r + \delta \\
H_{1} & \colon \mu \leq  r - \delta;
\end{align*}
\item[(b)] Testing with no indifference region---that is, the same as
  above but with $\delta = 0$.
\end{enumerate}

For the first setting (a), we always assume $0 < \delta < \nu \mydefn \min(\mu,1-\mu)$, and the
algorithm is evaluated on its ability to correctly choose between
$H_{0}$ and $H_{1}$ when one of them holds, but it incurs no penalty
for either choice when $\mu$ falls in the indifference region $\big(r
- \delta,~r + \delta\big)$. The error of a procedure $\alg$ can thus
be defined as
\begin{align*}
\err\big(\alg,~f\big) = \begin{cases}
  \P\big(\alg\big(X_{1:\infty}\big) = H_{1}\big) & ~ \text{if}~ \mu
  \in H_{0}, \\ \P\big(\alg\big(X_{1:\infty}\big) = H_{0}\big) & ~
  \text{if}~ \mu \in H_{1}, \\ 0 &~ \text{otherwise.}
\end{cases} 
\end{align*}

The rest of this subsection is organized as follows. For the first
setting (a), we analyze a procedure $\algfix$ that makes a decision
after a fixed number $N := N(\alpha)$ of samples. We also analyze a
sequential procedure $\algseq$ that chooses whether to reject
at a sequence $N_0,\ldots,N_k,\ldots$ of decision times. For the
second, more challenging, setting (b), we analyze $\algdiff$, which also
rejects at a sequence of decision times. For both $\algseq$ and
$\algdiff$, we calculate the expected stopping times of the
procedures.

As mentioned above, the simplest procedure $\algfix$ would choose a
fixed number $N$ of samples to be collected based on the target level
$\alpha$. After collecting $N$ samples, it forms the empirical average
$\hat{\mu}_{N} = \frac{1}{N}\sum_{n = 1}^{N} f\big(X_{n}\big)$ and
outputs $H_{0}$ if $\hat{\mu}_{N} \geq r + \delta$, $H_{1}$ if
$\hat{\mu}_{N} \leq r - \delta$, and outputs a special indifference
symbol, say $\mathrm{I}$, otherwise.

The sequential algorithm $\algseq$ makes decisions as to whether to
output one of the hypotheses or continue testing at a fixed sequence
of decision times, say $N_{k}$.  These times are defined recursively
by
\begin{align}
N_{0} & = \big\lfloor M \cdot \min\big(\frac{1}{r},~\frac{1}{1 -
  r}\big) \big\rfloor, \label{eq:seq-seq-N0} \\ N_{k} & = \big\lfloor
N_{0}\big(1 + \xi\big)^{k} \big\rfloor \label{eq:seq-seq-Nk},
\end{align}
where $M > 0$ and $0 < \xi < 2/5$ are parameters of the algorithm. At
each time $N_{k}$ for $k \geq 1$, the algorithm $\algseq$ checks if
\begin{align}
\label{eq:seq-indiff-stopping}
\hat{\mu}_{N_k} \in \big(r - \frac{M}{N_k}, r + \frac{M}{N_k} \big) .
\end{align}
If the empirical average lies in this interval, then the algorithm
continues sampling; otherwise, it outputs $H_0$ or $H_1$ accordingly
in the natural way.

For the sequential algorithm $\algdiff$, let $N_{0} > 0$ be chosen
arbitrarily,\footnote{In \citet{Gyo15Test}, the authors set $N_{0}
  = \big\lfloor \frac{100}{\gamma_{0}}\big\rfloor$, but this is
  inessential.} and let $N_{k}$ be defined in terms of $N_{0}$ as in
\eqref{eq:seq-seq-Nk}. It once again decides at each $N_{k}$ for $k
\geq 1$ whether to output an answer or to continue sampling, depending
on whether
\begin{align*}
\hat{\mu}_{N_k} \in \big(r - \epsilon_{k}\big(\alpha\big),~r +
\epsilon_{k}\big(\alpha\big)\big) .
\end{align*}
When this inclusion holds, the algorithm continues; when it doesn't
hold, the algorithm outputs $H_0$ or $H_1$ in the natural way.
The following result is restricted to the stationary case; later in
the section, we turn to the question of burn-in.  
\begin{thm}
\label{thm:seq-err-all-adapt} 
Assume that $\alpha \leq \frac2 {5}$.  For $\algfix, \algseq,
\algdiff$ to all satisfy $ \err\big(\alg,~f\big) ~\leq~ \alpha $, it
suffices to (respectively) choose
\begin{align}
 N &= \frac{2\Tf\big(\delta\big)
   \log\big(\frac{1}{\alpha}\big)}{\delta^2 }, \\ M &=
 \frac{8\Tf\big(\frac{\delta}2\big) \log\big(\frac2
   {\sqrt{\alpha\xi}}\big)}{\delta},~ \text{and} \\
\label{eq:seq-seq-diff-adapt-eps2-epsk}
\epsilon_{k}\big(\alpha\big) & = \inf \Big \{ \epsilon > 0 \colon
\frac{\epsilon^2 }{8\Tf\big(\frac{\epsilon}{2}\big)} \geq
\frac{ \log \big(1 / \alpha \big) + 1 + 2 \log{k}}{N_{k}} \Big \},
\end{align}
where we let $\inf \emptyset = \infty$.
\end{thm}
Our results differ from those of~\cite{Gyo15Test} because the latter
implicitly control the worst-case error of the algorithm
\begin{align*}
\err\big(\alg\big) = \sup_{f \colon \Omega \rightarrow [0,~1]}
\err\big(\alg,~f\big),
\end{align*}
while our analysis controls $\err\big(\alg,~f\big)$ directly. The corresponding choices made in \cite{Gyo15Test} are
\begin{align*}
N = \frac{\log(1/\alpha)}{\gamma_{0}\delta^2 }, 
M = \frac{\log (\frac2
  {\sqrt{\alpha\xi}})}{\gamma_{0}\delta} , \text{ and }
\epsilon_{k}(\alpha) = \sqrt{\frac{\log (1/\alpha) + 1 + 2   \log{k}}{\gamma_{0}N_{k}}}.
\end{align*}
Hence, the $\Tf$ parameter in
our bounds plays the same role that $\frac{1}{\gamma_0}$
plays in their uniform bounds.
As a result of this close correspondence, we easily see that our
results improve on the uniform result for a fixed function $f$
whenever it converges to its stationary expectation faster than the
chain itself converges--- i.e., whenever 
$\Tf\big(\delta\big) \leq \frac{1}{2\gamma_{0}}$.

The value of the above tests depends substantially on their sample
size requirements. In setting (a), algorithm $\algseq$ is only
valuable if it reduces the number of samples needed compared to
$\algfix$. In setting (b), algorithm $\algdiff$ is valuable because of
its ability to test between hypotheses separated only by a point, but
its utility is limited if it takes too long to run. Therefore, we now
turn to the question of bounding expected stopping times.

In order to o carry out the stopping time analysis, we introduce the
true margin $\Delta = |r - \mu |$.  First, let us introduce some
useful notation.  Let $N(\alg)$ be the number of sampled collected by
$\alg$. Given a margin schedule $\big(\epsilon_{k}\big)$, let
\begin{align*}
 k_{0}^{\ast}\big(\epsilon_{1:\infty}\big) \mydefn \min \big \{ k \geq 1
 \colon \epsilon_{k} \leq \frac{\Delta}2 \big \}, \text{ and } N_{0}^{\ast}\big(\epsilon_{1:\infty}\big) \mydefn
 N_{k_{0}^{\ast}\big(\epsilon_{1:\infty}\big)}.
\end{align*}

We can bound the expected stopping times of $\algseq,\algdiff$ in
terms of $\Delta$ as follows: 
\begin{thm}
\label{thm:seq-stopping-time-all}
Assume either $H_{0}$ or $H_{1}$ holds. Then,
\begin{align}
\E\big[N(\algseq)\big] &\leq \big(1 + \xi\big)\big[\frac{M}{\Delta} +
  \frac4 {\Delta}\sqrt{\frac{2\Tf\big(\delta/2\big)M}{\Delta} +
    8\Tf\big(\delta/2\big)} + 1\big];\\ \E\big[N(\algdiff)\big]
&\leq \big(1 + \xi\big)\big(N_{0}^{\ast} + 1\big) + \frac{32\alpha
  T_{f}\big(\Delta/4\big)}{\Delta^2 }.
\end{align}
\end{thm}

With minor modifications to the proofs in~\cite{Gyo15Test}, we can
bound the expected stopping times of their procedures as
\begin{align*}
\E [N(\algseq)] &\leq (1 + \xi) \Big \{ \frac{M}{\Delta} + \frac2
   {\Delta}\sqrt{\frac{M}{\gamma_{0}\Delta} + \frac{4}{\gamma_{0}}} +
   1 \Big \};\\
\E [ N (\algdiff)] & \leq (1 + \xi) \big(N_{0}^{\ast} + 1\big) +
\frac{4\alpha}{\gamma_{0}\Delta^2 }.
\end{align*}

In order to see how the uniform and adaptive bounds compare, it is
helpful to first note that, under either $H_{0}$ or $H_{1}$, we have
the lower bound $\Delta \geq \delta$. Thus, the dominant term in the
expectations in both cases is $(1 + \xi)M/\Delta$.  Consequently, the
ratio between the expected stopping times is approximately equal to the
ratio between the $M$ values---viz.,
\begin{align}
 \frac{M_{\mathrm{adapt}}}{M_{\mathrm{unif}}} \approx \gamma_{0}\Tf\big(\delta/2\big).
\end{align}
As a result, we should expect a significant improvement in terms of
number of samples when the relaxation time $\frac{1}{\gamma_{0}}$ is
significantly larger than the \mbox{$f$-mixing} time
$\Tf\big(\delta/2\big)$. Framed in absolute terms, we
can write
\begin{align*}
\bar{N}_{\mathrm{unif}} (\algseq) \approx
\frac{\log\big(2/\sqrt{\alpha\xi}\big)}{\gamma_{0}\delta\Delta} \quad
\mbox{and} \quad \bar{N}_{\mathrm{adapt}} (\algseq) \approx
\frac{\Tf\big ( \delta / 2 \big) \log \big( 2 / \sqrt{\alpha
    \xi} \big)}{\delta \Delta}.
\end{align*}
Up to an additive term, the bound for $\algdiff$ is also qualitatively
similar to earlier ones, with $\frac{1}{\delta\Delta}$ replaced by
$\frac{1}{\Delta^2 }$.


\section{Analyzing mixing in practice}
\label{subsec:practical}

We analyze several examples of MCMC-based Bayesian analysis
from our theoretical perspective. These examples demonstrate that
convergence in discrepancy can in practice occur much faster than
suggested by naive mixing time bounds and that our bounds help narrow the gap between theoretical predictions and
observed behavior.

\begin{figure}[h]
\begin{center}
\begin{tabular}{cc}
\widgraph{0.5\linewidth}{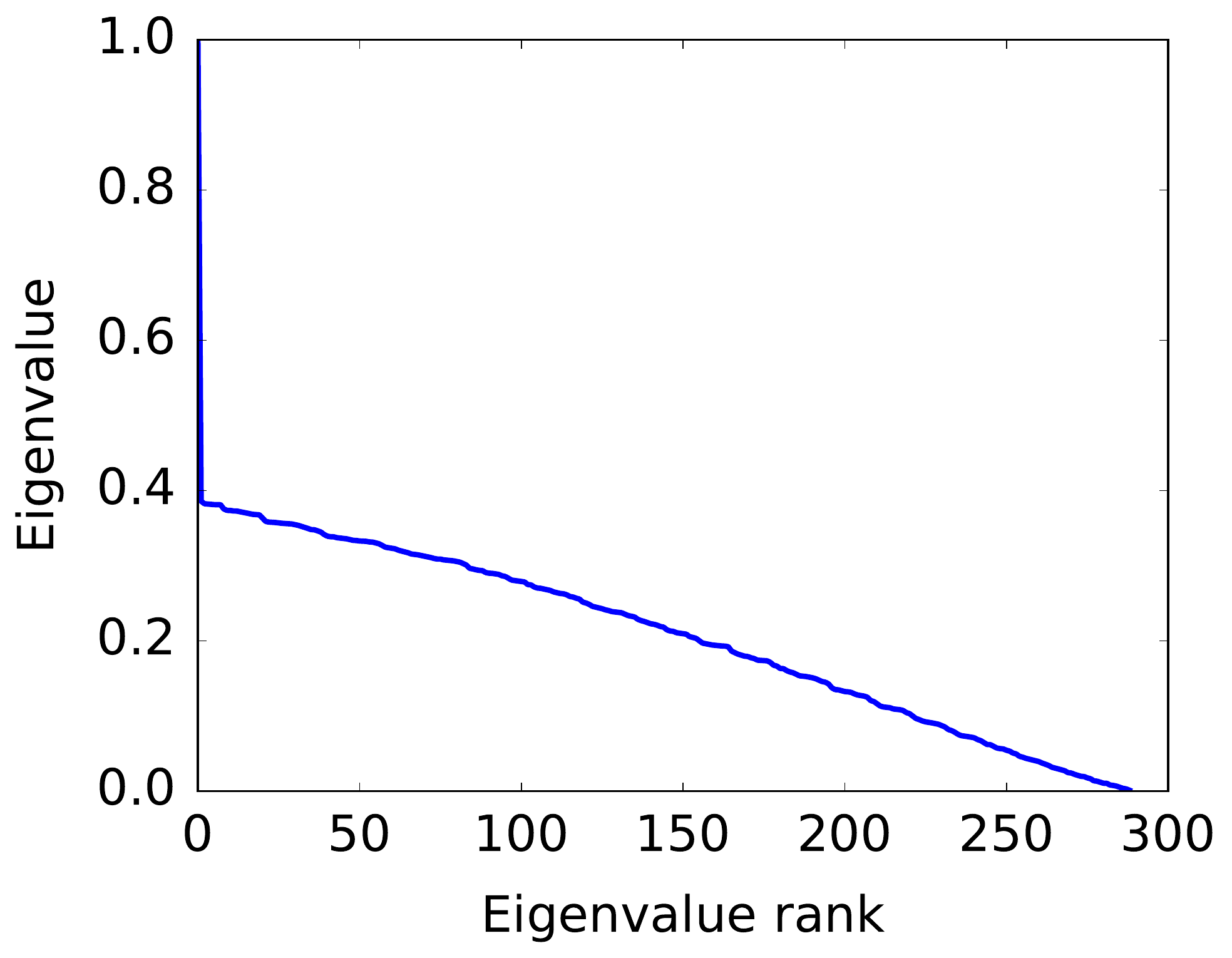} &
\widgraph{0.5\linewidth}{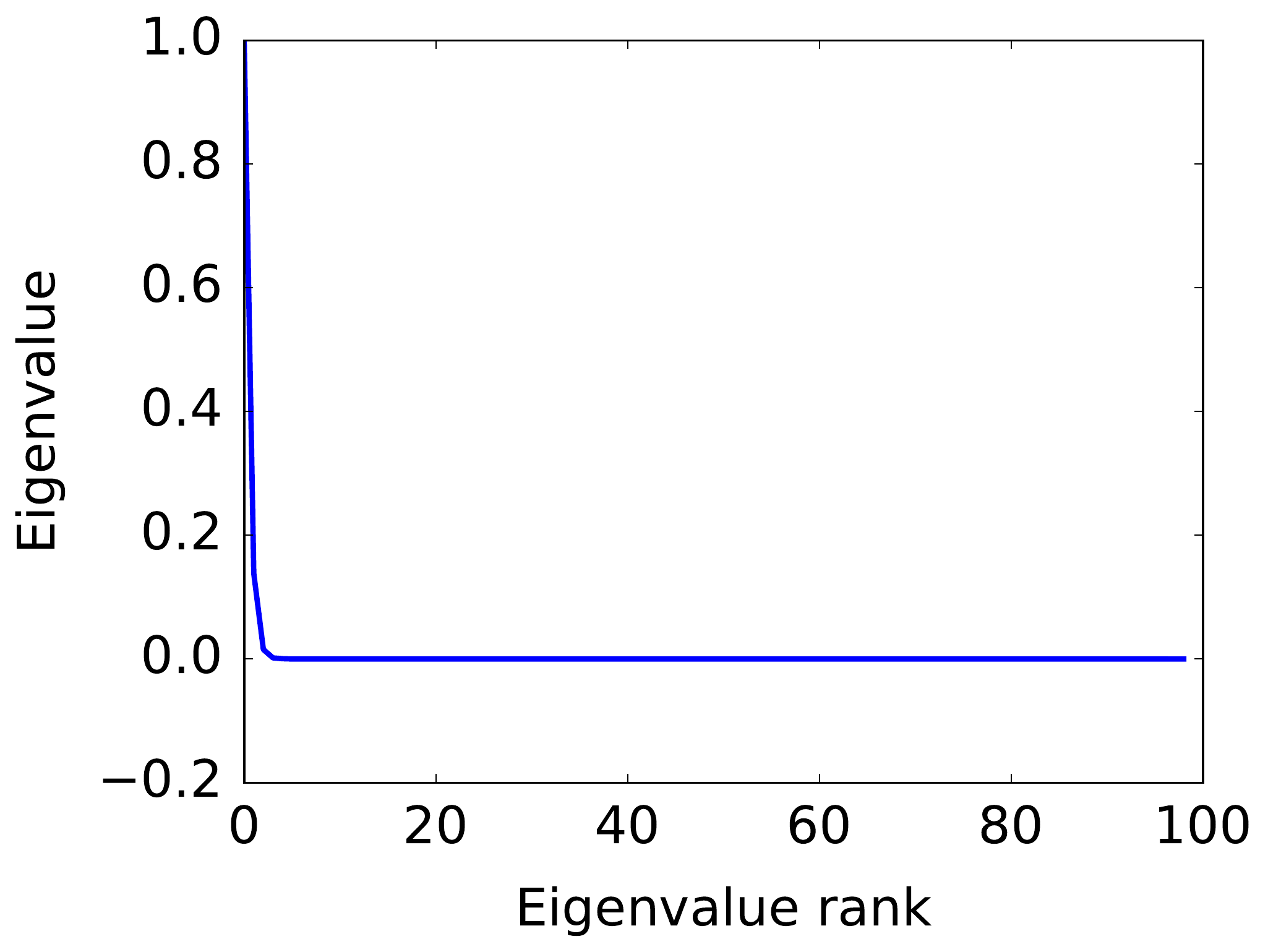} \\
(a) & (b)
\end{tabular}
\begin{tabular}{c}
\widgraph{0.5\linewidth}{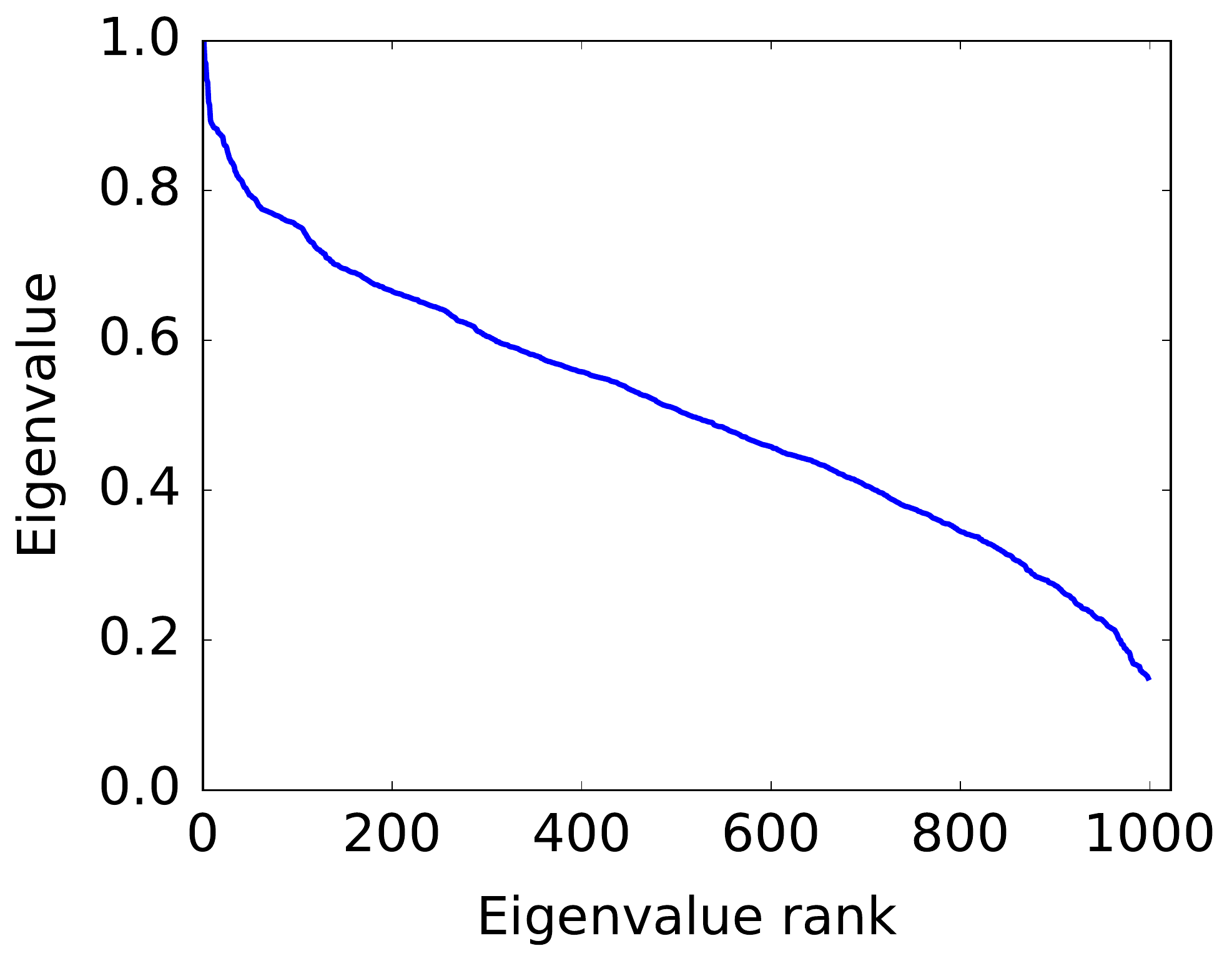} \\
(c)
\end{tabular}
\caption{Spectra for three example chains: (a) Metropolis-Hastings for
  Bayesian logistic regression; (b) collapsed Gibbs sampler for
  missing data imputation; and (c) collapsed Gibbs sampler for a
  mixture model.
\label{fig:example-spectra}}
\end{center}
\end{figure}


\subsection{Bayesian logistic regression}

Our first example is a Bayesian logistic regression problem
introduced by \citet{Rob05MCMC}. The data consists of $23$
observations of temperatures (in Fahrenheit, but normalized by
dividing by $100$) and a corresponding binary outcome---failure $(y =
1)$ or not $(y = 0)$ of a certain component; the aim is to fit a
logistic regressor, with parameters $\big(\alpha,~\beta\big) \in \R^2 $, to the
data, incorporating a prior and integrating over the model uncertainty
to obtain future predictions. More explicitly, following the analysis
in \citet{Gyo12Nonasym}, we consider the following model:
\begin{align*}
p \big( \alpha, ~\beta ~| ~b \big) & = \frac{1}{b} \cdot e^{\alpha}
\exp \big ( -e^{\alpha} / b \big) \\
p\big(y \, \mid \, \alpha,~\beta,~x\big) & \propto \exp
\big(\alpha + \beta x \big),
\end{align*}
which corresponds to an exponential prior on $e^{\alpha}$, an
improper uniform prior on $\beta$ and a logit link for prediction. As
in \citet{Gyo12Nonasym}, we target the posterior by running
a Metropolis-Hastings algorithm with a Gaussian proposal with covariance matrix
$\Sigma ~=~ \begin{pmatrix} 4 & 0 \\ 0 & 10 
\end{pmatrix}.$
Unlike in their paper, however, we discretize the state space to
facilitate exact analysis of the transition matrix and to make our
theory directly applicable. The resulting state space is given by
\begin{align*}
\Omega = \Big \{ \big(\hat{\alpha} \pm i \cdot \Delta,~\hat{\beta} \pm
j \cdot \Delta\big) \, \mid \, 0 \leq i,~j \leq 8 \Big \},
\end{align*}
where $\Delta = 0.1$ and \mbox{$(\hat{\alpha},~\hat{\beta})$} is the
MLE. This space has $d = 17^2 = 289$ elements, resulting in a $289
\times 289$ transition matrix that can easily be diagonalized.

\citet{Rob05MCMC} analyze the probability of failure when the
temperature $x$ is $\SI{65}{\degree}\mathrm{F}$; it is specified by
the function
\begin{align*}
f_{65} \big( \alpha,~\beta \big) = \frac{\exp \big( \alpha + 0.65
  \beta \big)}{1 + \exp \big( \alpha + 0.65 \beta \big)} .
\end{align*}
Note that this function fluctuates significantly under the posterior,
as shown in Figure~\ref{fig:oring-p65-hist}.
\begin{figure}[h!]
\centering
\widgraph{0.5\linewidth}{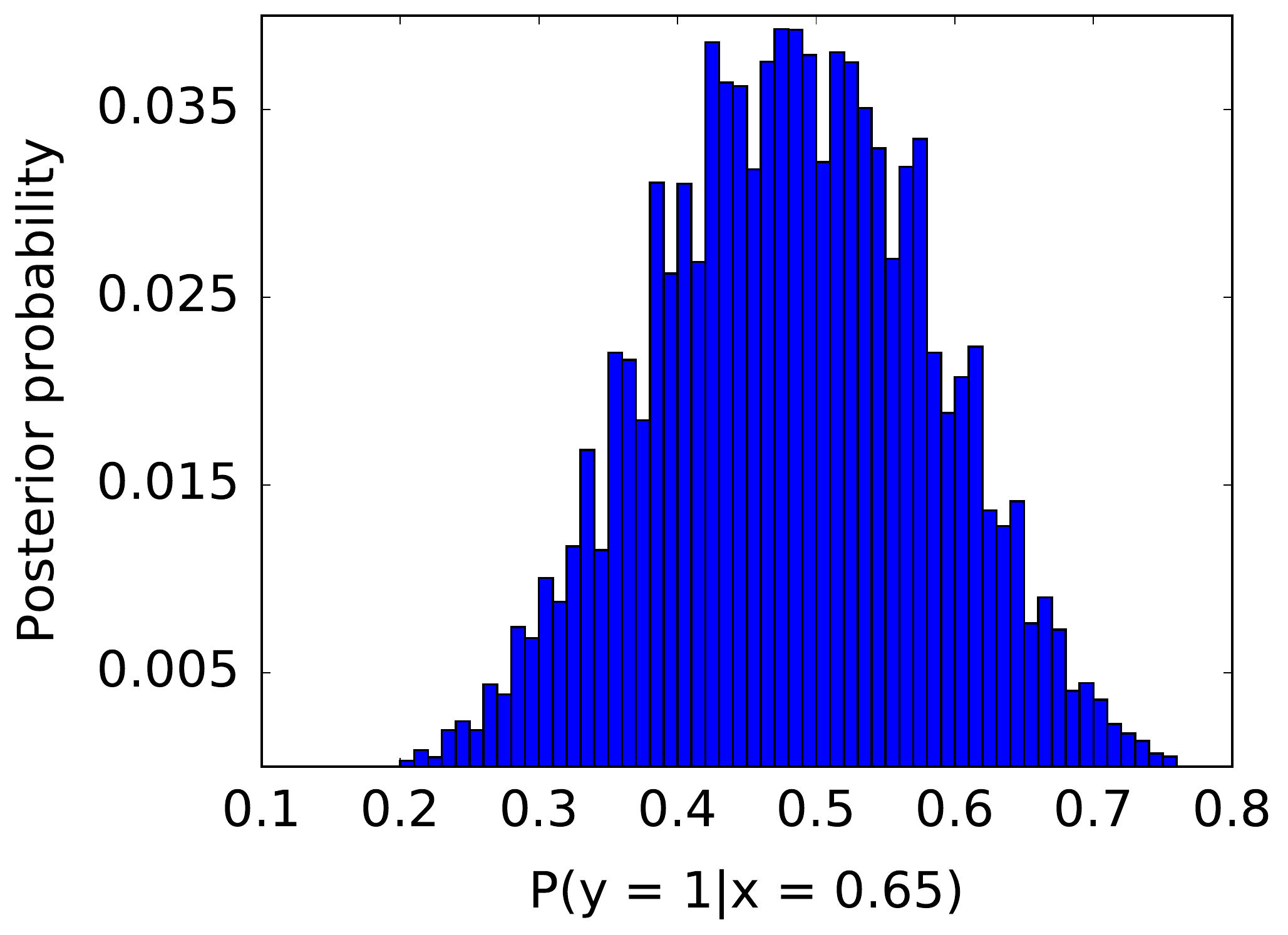}
\caption{Distribution of $f_{65}$ values under the posterior. Despite
  the discretization and truncation to a square, it generally matches
  the one displayed in Figure 1.2 in~\citet{Rob05MCMC}.
\label{fig:oring-p65-hist}}
\end{figure}

We find that this function also happens to exhibit rapid mixing. The discrepancy
$d_{f_{65}}$, before entering an asymptotic regime in which it decays
exponentially at a rate $1 - \gamma^{\ast} \approx 0.386$, first drops
from about $0.3$ to about $0.01$ in just $2$ iterations, compared to
the predicted $10$ iterations from the naive bound
$
d_{f} \big( n \big) \leq d_{\TV} \big( n \big) \leq
\frac{1}{\sqrt{\pimin}} \cdot \big(1 - \gamma^{\ast}\big)^{n} .
$
Figure~\ref{fig:oring-discrep-1} demonstrates this on a log scale,
comparing the naive bound to a version of the bound in
Lemmas~\ref{lem:mix-gap-all} and~\ref{lem:mix-gap-all-oracle}. Note
that the oracle $f$-discrepancy bound improves significantly over the
uniform baseline, even though the non-oracle version does not.  In
this calculation, we took $J = \big \{ 2, \dots, 140 \big \}$ to
include the top half of the spectrum excluding $1$ and computed $\ \|
h_{j} \|_{\infty}$ directly from $P$ for $j \in J$ and likewise for
$q_{j}^{T} f_{65}$. The oracle bound is given by
Lemma~\ref{lem:mix-gap-all-oracle}.  As shown in
panel (b) of Figure~\ref{fig:oring-discrep-1}, this decay is also faster than that
of the total variation distance.

\begin{figure}[h!]
\begin{tabular}{cc}
\widgraph{0.5\linewidth}{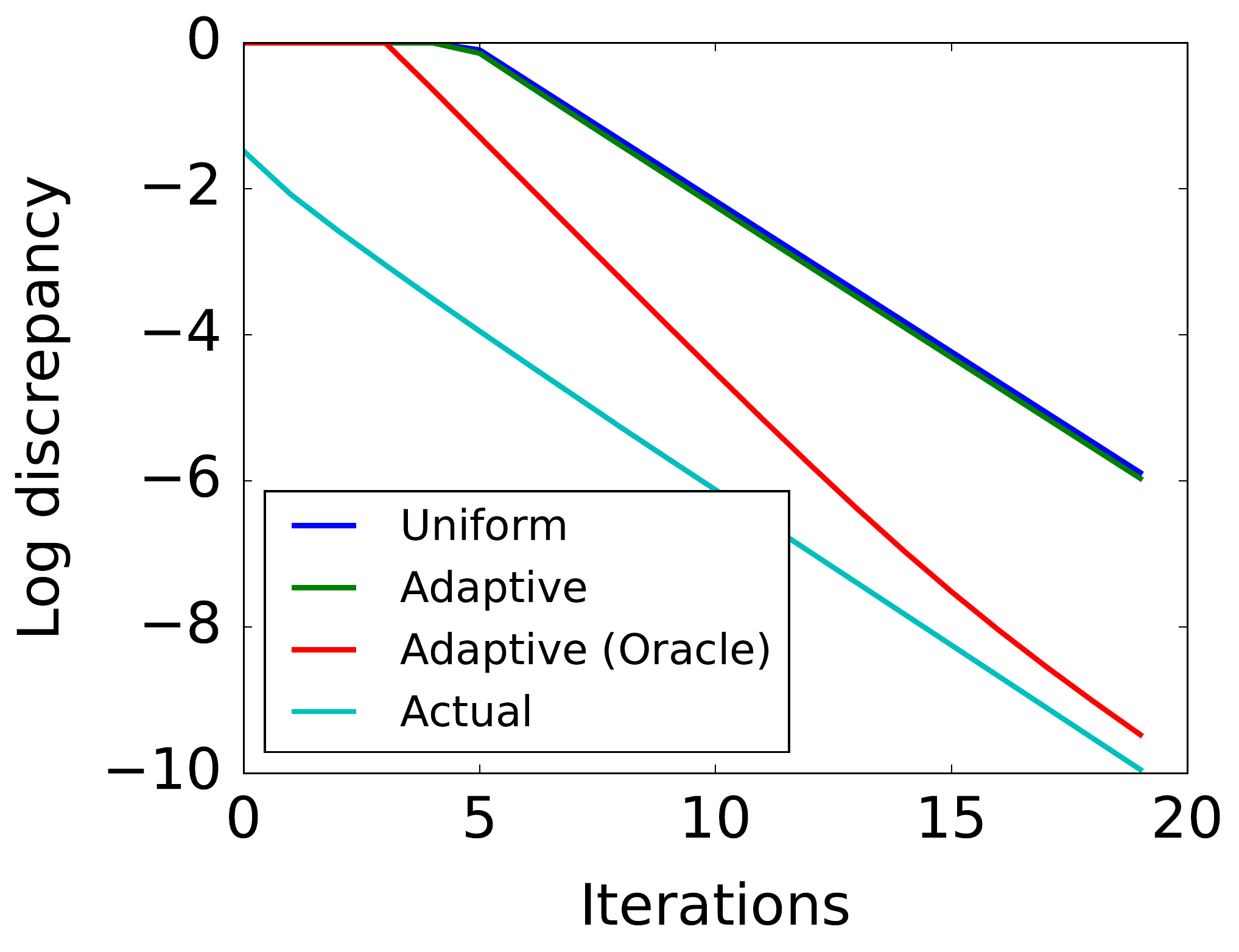} &
\widgraph{0.5\linewidth}{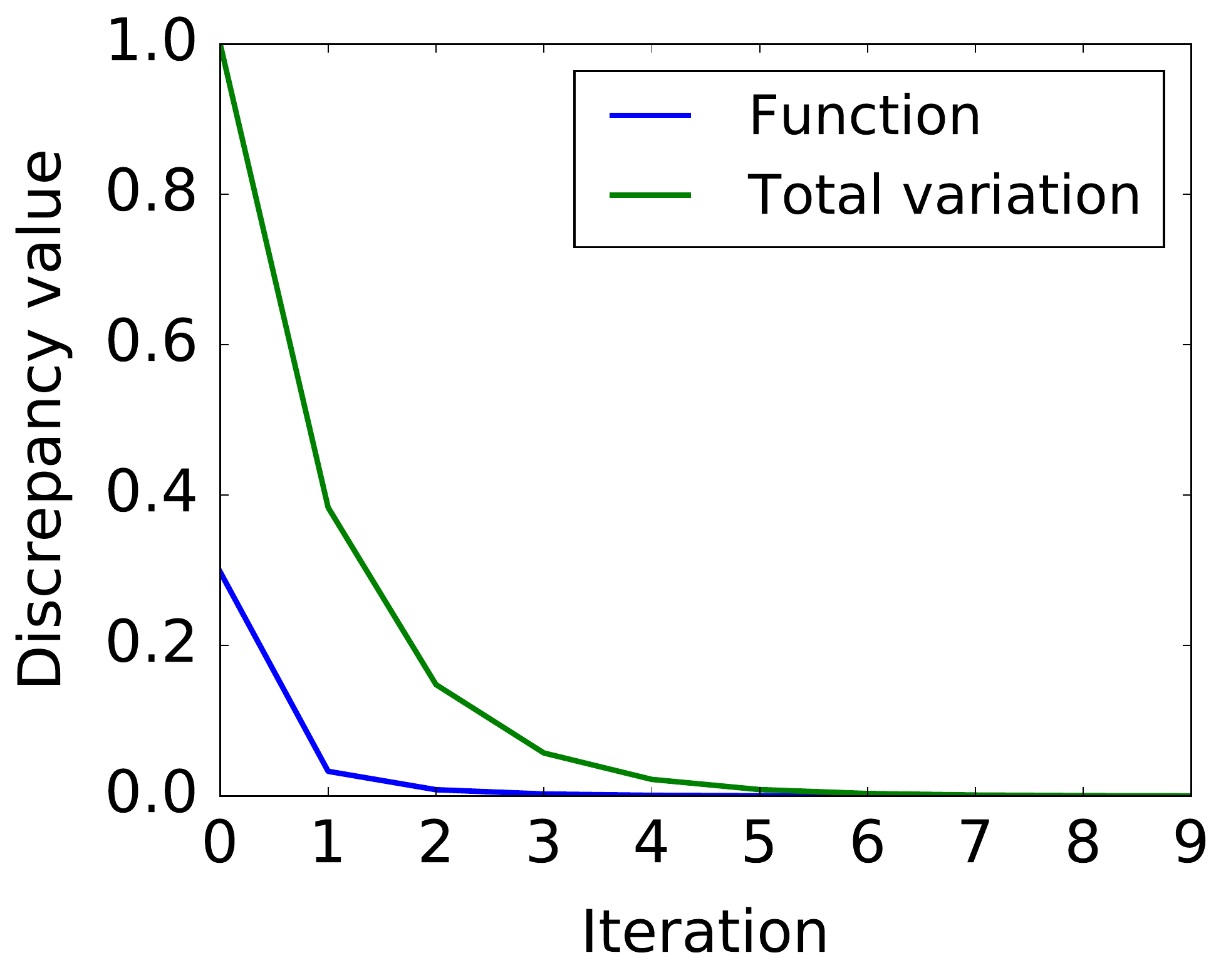}\\
(a) & (b)
\end{tabular}
\caption{(a) Discrepancies (plotted on log-scale) for $f_{65}$ as a
  function of iteration number. The prediction of the naive bound is
  highly pessimistic; the $f$-discrepancy bound goes part of the way toward
  closing the gap and the oracle version of the $f$-discrepancy bound nearly
  completely closes the gap in the limit and also gets much closer to
  the right answer for small iteration
  numbers. (b) Comparison of the function discrepancy $d_{f_{65}}$ and the
  total variation discrepancy $d_{\TV}$. They both decay fairly
  quickly due to the large spectral gap, but the function discrepancy
  still falls much faster.
  \label{fig:oring-discrep-1}}
\end{figure}
%

An important point is that the quality of the $f$-discrepancy bound depends
significantly on the choice of $J$. In the limiting case where $J$
includes the whole spectrum below the top eigenvalue, the oracle bound
becomes exact. Between that and $J = \emptyset$, the oracle bound
becomes tighter and tighter, with the rate of tightening depending on
how much power the function has in the higher versus lower
eigenspaces. Figure \ref{fig:oring-discrep-Js} illustrates this for a
few settings of $J$, showing that although for this function and this
chain, a comparatively large $J$ is needed to get a tight bound, the
oracle bound is substantially tighter than the uniform and non-oracle
$f$-discrepancy bounds even for small $J$.

\begin{figure}[h!]
\begin{center}
\begin{tabular}{cc}
\widgraph{0.5 \linewidth}{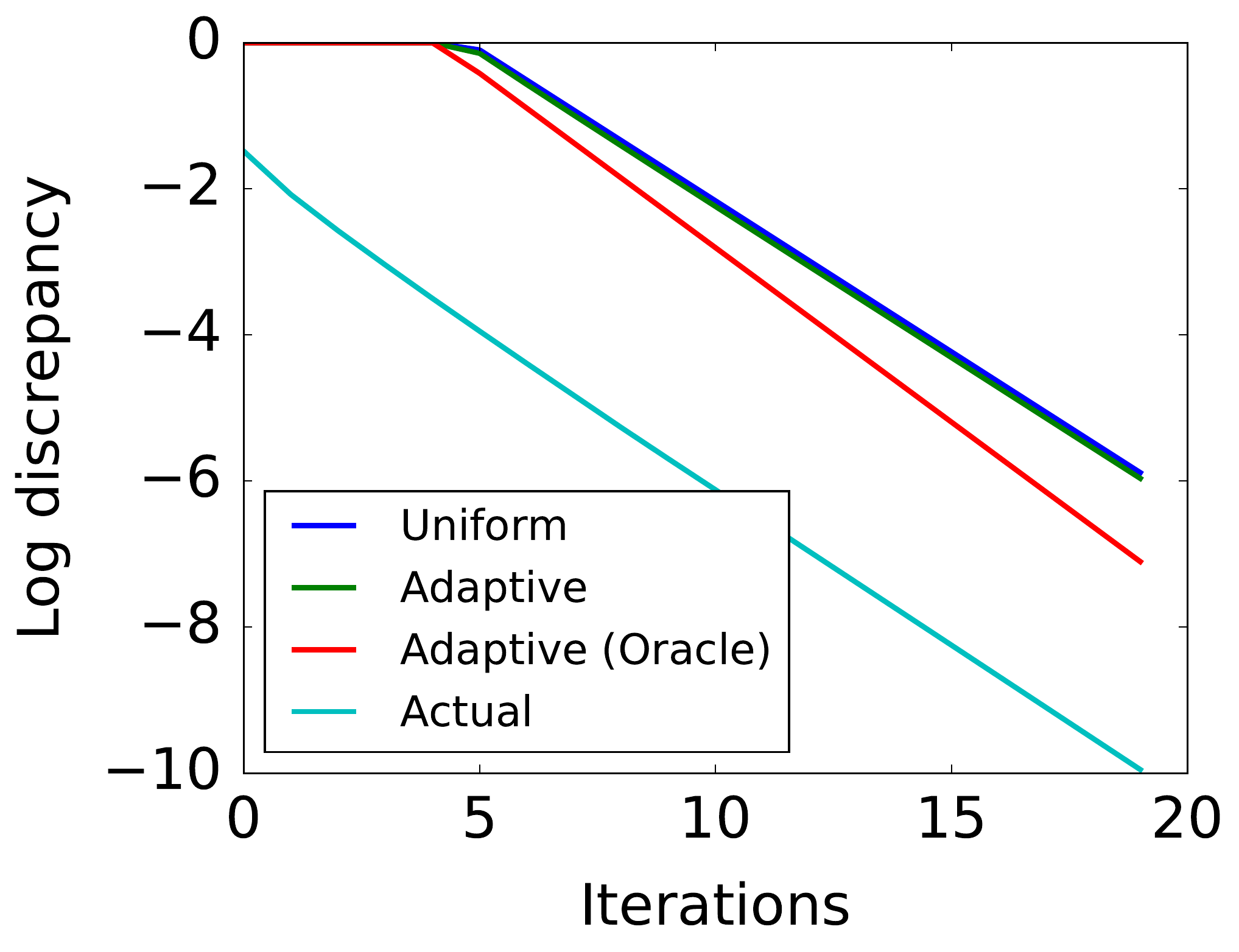} &
\widgraph{0.5 \linewidth}{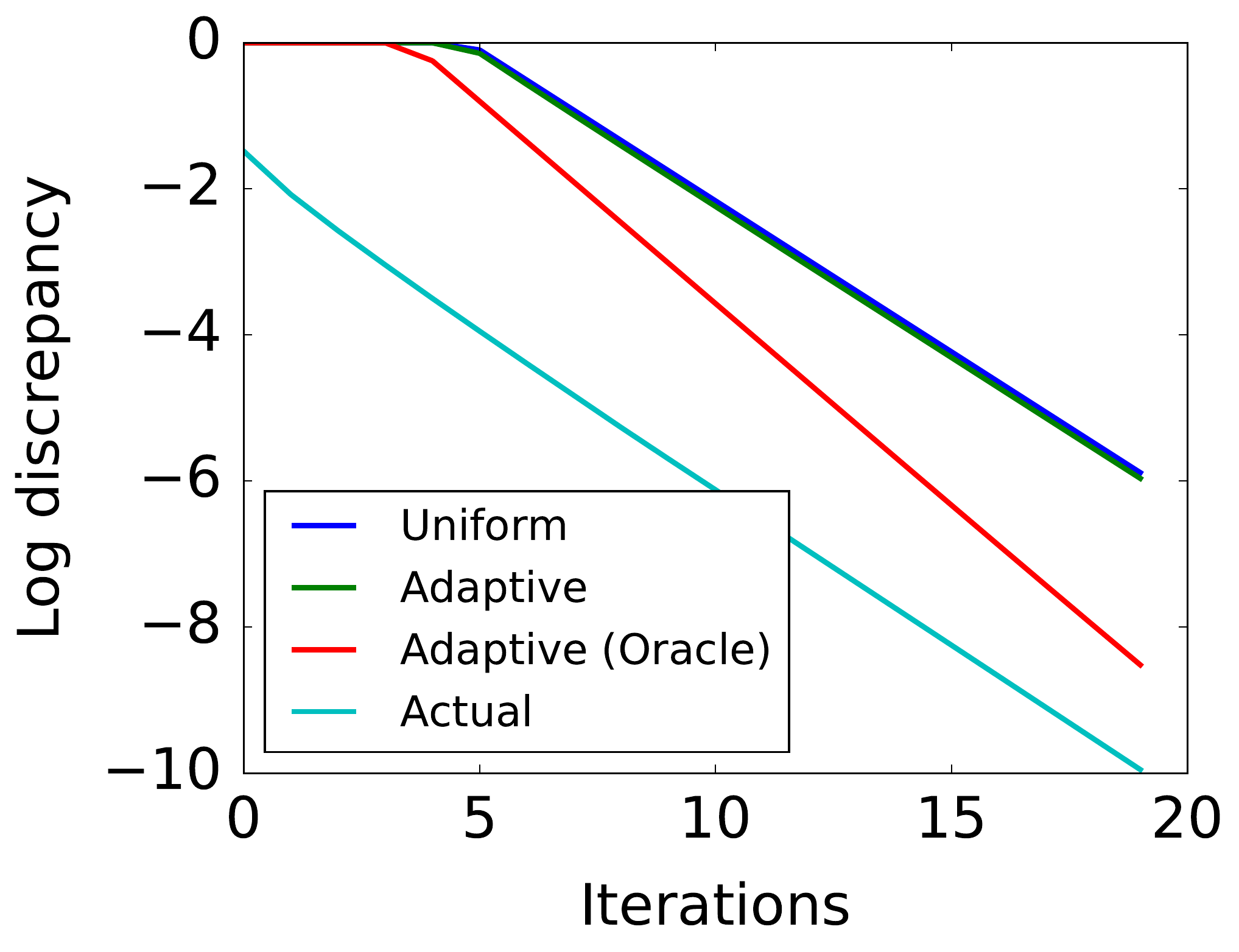} \\
(a) & (b) \\
\widgraph{0.5 \linewidth}{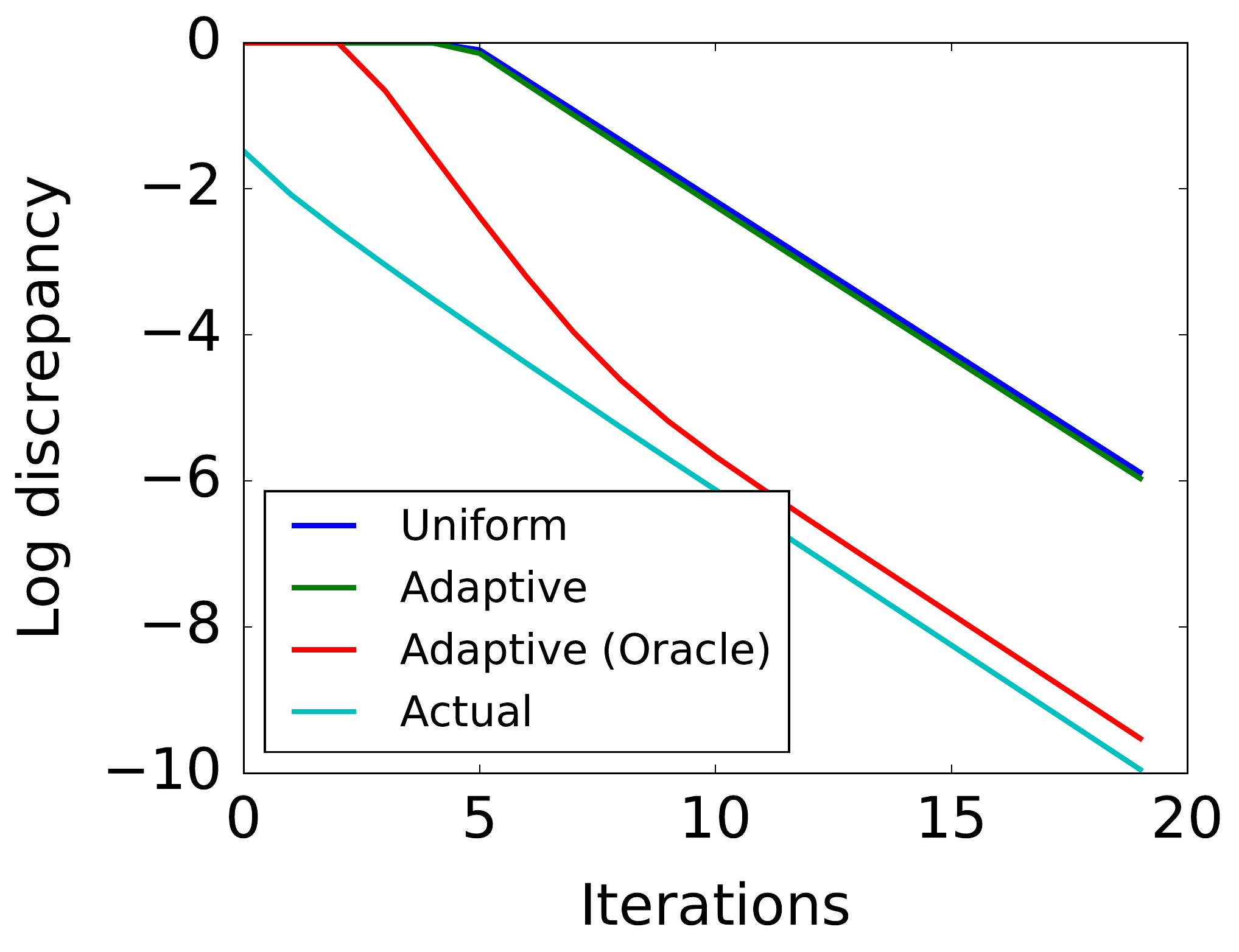} &
\widgraph{0.5 \linewidth}{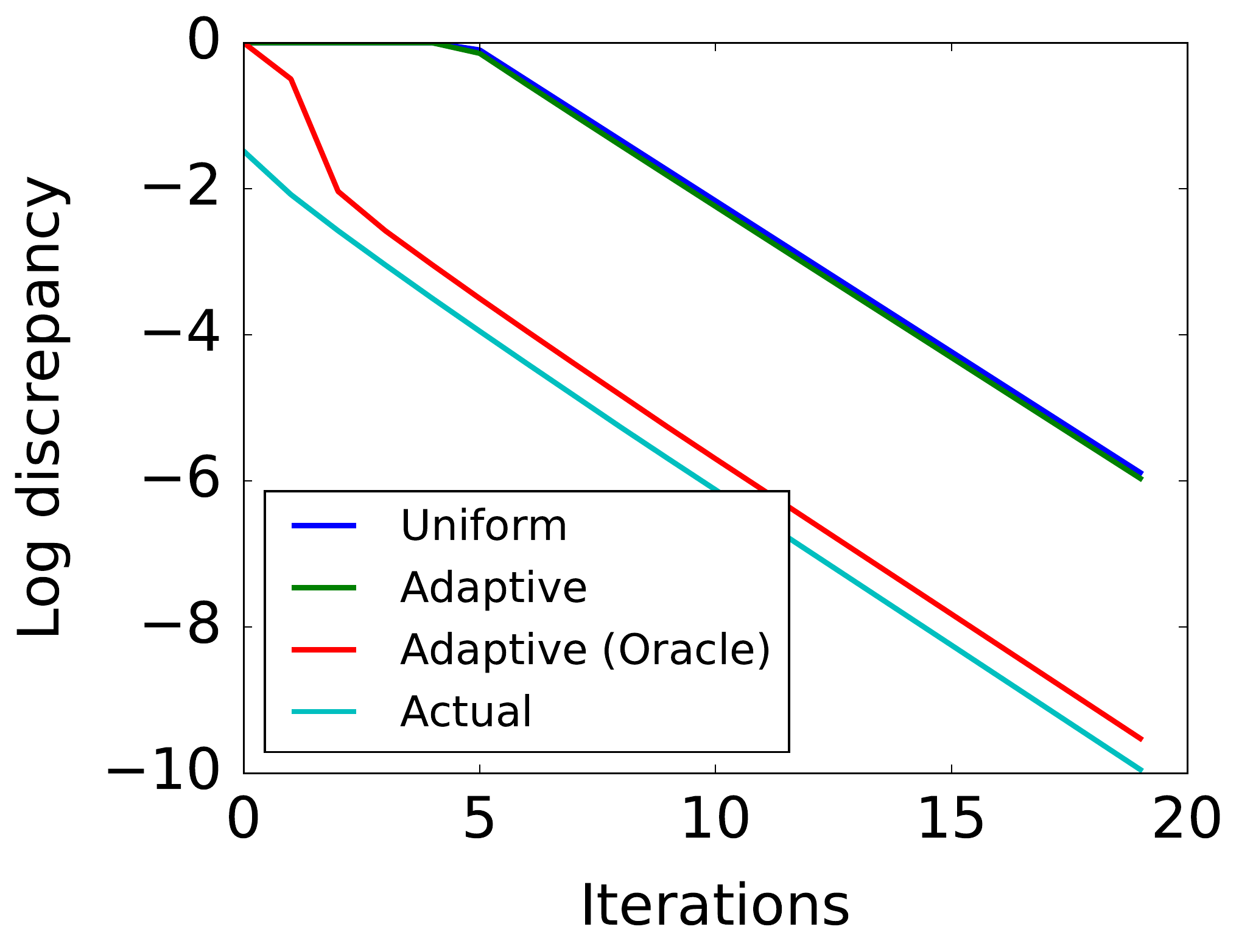} \\
(c) & (d)
\end{tabular}
\end{center}
\caption{Comparisons of the uniform, non-oracle function-specific, and oracle
  function-specific bounds for various choices of $J$. In each case, $J = \{ 2,
  \dots, \Jmax \}$, with $\Jmax = 50$ in panel (a), $\Jmax = 100$ in
  panel (b), $\Jmax = 200$ in panel (c), and $\Jmax = 288$ in panel
  (d). The oracle bound becomes tight in the limit as $\Jmax$ goes to
  $d = 289$, but it offers an improvement over the uniform bound
  across the board.}
\label{fig:oring-discrep-Js}
\end{figure}


\subsection{Bayesian analysis of clinical trials}

The problem of missing data often necessitates Bayesian analysis,
particularly in settings where uncertainty quantification is
important, as in clinical trials. We illustrate how our framework
would apply in this context by considering a clinical trials
dataset~\citep{Ber10Clinical,Gyo12Nonasym}.

The dataset consists of $n = 50$ patients, some of whom participated
in a trial for a drug and exhibited early indicators ($Y_{i}$) of
success/failure and final indicators ($X_{i}$) of success/failure. Among
the 50 patients, both indicator values are available for $n_{X} = 20$ patients;
early indicators are available for $n_{Y} = 20$ patients; and no indicators are
available for $n_{0} = 10$ patients. The analysis depends on the following
parameterization:
\begin{align*}
\P\big(X_{i} = 1~ \mid ~Y_{i} = 0\big) & = \gamma_{0}, \\ \P\big(X_{i}
= 1~ \mid ~Y_{i} = 1\big) & = \gamma_{1}, \\ 
\P\big(X_{i} = 1~ \mid ~Y_{i}~\text{missing}\big) & = p .
\end{align*}
Note that, in contrast to what one might expect, $p$ is to be
interpreted as the marginal probability that $X_{i} = 1$, so that in
actuality $p = \P\big(X_{i} = 1\big)$ unconditionally; we keep the
other notation, however, for the sake of consistency with past
work~\citep{Ber10Clinical,Gyo12Nonasym}.  Conjugate uniform (i.e.,
$\Beta\big(1,~1\big)$) priors are placed on all the model parameters.

The unknown variables include the parameter triple
$\big(\gamma_{0},~\gamma_{1},~p\big)$ and the unobserved $X_{i}$
values for $n_{Y} + n_{0} = 30$ patients, and the full sample space is
therefore $\tilde{\Omega} = [0,~1]^{3} \times \big \{ 0,~1
\big \}^{30}$. We cannot estimate the transition matrix for this
chain, even with a discretization with as coarse a mesh as $\Delta =
0.1$, since the number of states would be $d = 10^{3} \times 2^{30}
\sim 10^{12}$. We therefore make two changes to the original MCMC
procedure. First, we collapse out the $X_{i}$ variables to bring the
state space down to $[0,~1]^{3}$; while analytically collapsing out
the discrete variables is impossible, we can estimate the transition
probabilities for the collapsed chain analytically by sampling the
$X_{i}$ variables conditional on the parameter values and forming a
Monte Carlo estimate of the collapsed transition
probabilities. Second, since the function of interest in the original
work---namely, $f\big(\gamma_0,~\gamma_1,~p\big) = \ones\big(p >
0.5\big)$---depends only on $p$, we fix $\gamma_{0}$ and $\gamma_{1}$
to their MLE values and sample only $p$, restricted to the unit
interval discretized with mesh $\Delta = 0.01$.

\begin{figure}[h!]
\centering 
\widgraph{0.5\linewidth}{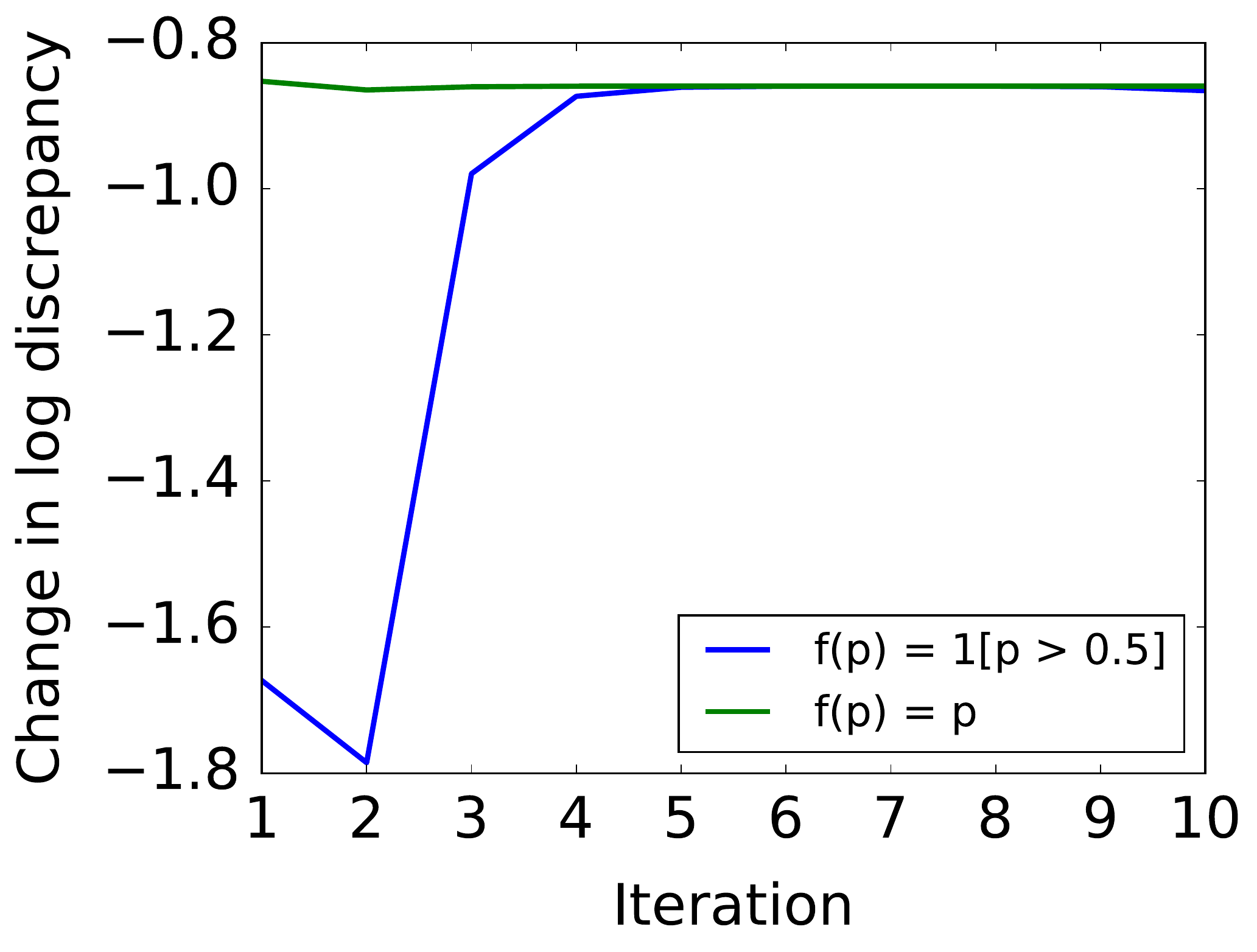}
\caption{Change in log discrepancy for the two functions $f(p) =
  \ones\big(p \geq 0.5\big)$ and $f(p) = p$ considered above. Whereas
  $f(p) = p$ always changes at the constant rate dictated by the
  spectral gap, the indicator discrepancy decays more quickly in the
  first few iterations.\label{fig:clinical-rates}}
\end{figure}

As Figure~\ref{fig:example-spectra} shows, eigenvalue decay occurs
rapidly for this sampler, with $\gamma^{\ast} \approx 0.86$. Mixing
thus occurs so quickly that none of the bounds---uniform or
function-specific---get close to the truth, due to the presence of the constant
terms (and specifically the large term $\frac{1}{\sqrt{\pimin}}
\approx 2.14 \times 10^{33}$). Nonetheless, this example still
illustrates how in actual fact, the choice of target function can make
a big difference in the number of iterations required for accurate
estimation; indeed, if we consider the two functions
\begin{align*}
f_1(p) \mydefn \ones\big(p > 0.5\big), \quad \mbox{and} \quad f_2(p)
\mydefn p,
\end{align*}
we see in Figure~\ref{fig:clinical-rates} that the mixing behavior
differs significantly between them: whereas the discrepancy for the
second decays at the asymptotic exponential rate from the
outset, the discrepancy for the first decreases faster
 (by about an order of magnitude) for the first few
iterations, before reaching the asymptotic rate dictated by the spectral gap.


\subsection{Collapsed Gibbs sampling for mixture models}

Due to the ubiquity of clustering problems in applied statistics and machine
learning, Bayesian inference for mixture models (and their
generalizations) is a widespread application of MCMC
\citep{Gha05IBP,Gri04LDA,Jai07SplitMerge,Mim12SSVI,Nea00DP}. We
consider the mixture-of-Gaussians model, applying it to a
subset of the schizophrenic reaction time data analyzed in
\citet{Bel95Schizophrenia}. The subset of the data we consider
consists of $10$ measurements, with $5$ coming from healthy subjects
and $5$ from subjects diagnosed with schizophrenia. Since our
  interest is in contexts where uncertainty is high, we chose the $5$
  subjects from the healthy group whose reaction times were greatest
  and the $5$ subjects from the schizophrenic group whose reaction
  times were smallest. We considered a mixture with $K = 2$
components, viz.:
\begin{align*}
\mu_{b} & \sim \Norm\big(0,~\rho^2 \big),~ b = 0,~1, \\
\omega  &  \sim \Beta\big(\alpha_0,~\alpha_1\big) \\
Z_{i} ~|~ \omega & \sim \Bern\big(\omega\big) \\
X_{i} ~|~ Z_{i} = b,~\mu & \sim \Norm\big(\mu_{b},~\sigma^2 \big) .
\end{align*}
We chose relatively uninformative priors, setting $\alpha_0 = \alpha_1
= 1$ and $\rho = 237$. Increasing the value chosen in the original
analysis \citep{Bel95Schizophrenia}, we set $\sigma \approx 70$; we
found that this was necessary to prevent the posterior from being too
highly concentrated, which would be an unrealistic setting for
MCMC. We ran collapsed Gibbs on the indicator variables $Z_{i}$ by
analytically integrating out $\omega$ and $\mu_{0:1}$.

As Figure~\ref{fig:example-spectra} illustrates, the spectral gap for
this chain is small---namely, $\gamma_{\ast} \approx 3.83 \times
10^{-4}$---yet the eigenvalues fall off comparatively quickly after
$\lambda_2 $, opening up the possibility for improvement over the
uniform $\gamma_{\ast}$-based bounds. In more detail, define 
\begin{align*}
z^{\ast}_{b} & \mydefn \begin{pmatrix} b & b & b & b & b & 1 - b & 1 -
  b & 1 - b & 1 - b & 1 - b \end{pmatrix},
\end{align*}
corresponding to the cluster assignments in which the patient and
control groups are perfectly separated (with the control group being
assigned label $b$).  We can then define the indicator for exact
recovery of the ground truth by
\begin{align*}
f(z) = \ones\big(z \in \big \{ z_{0}^{\ast},~z_{1}^{\ast}\big \}\big).
\end{align*}

As Figure~\ref{fig:mixture-DFvsTV} illustrates, convergence in terms
of $f$-discrepancy occurs much faster than convergence in total
variation, meaning that predictions of required burn-in times and
sample size based on global metrics of convergence drastically
overestimate the computational and statistical effort required to
estimate the expectation of $f$ accurately using the collapsed Gibbs
sampler. This behavior can be explained in terms of the interaction
between the function $f$ and the eigenspaces of $P$.  Although the
pessimistic constants in the bounds from the uniform
bound~\eqref{eq:mix-TV} and the non-oracle function-specific bound
(Lemma~\ref{lem:mix-gap-all}) make their predictions overly
conservative, the oracle version of the function-specific bound
(Lemma~\ref{lem:mix-gap-all-oracle}) begins to make exact predictions
after just a hundred iterations when applied with $J = \big \{ 1,
\dots, 25 \big \}$; this corresponds to making exact predictions of
$T_{f}\big(\delta\big)$ for $\delta \leq \delta_{0} \approx 0.01$,
which is a realistic tolerance for estimation of $\mu$.
Panel (b) of Figure~\ref{fig:mixture-DFvsTV} documents this by
plotting the $f$-discrepancy oracle bound against the actual value of
$d_{f}$ on a log scale.

\begin{figure}[h!]
\begin{center}
\begin{tabular}{cc}
\widgraph{0.45\linewidth}{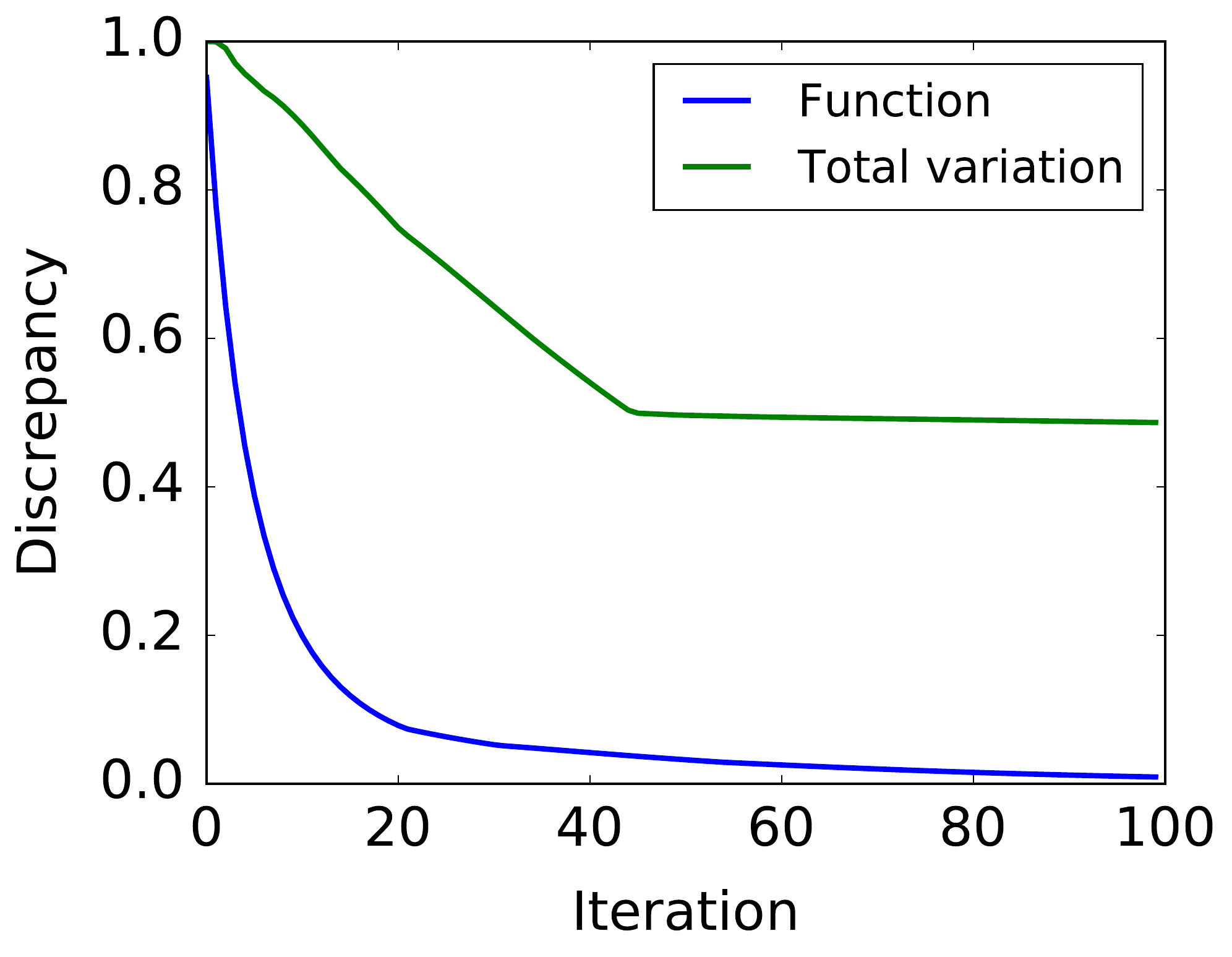} & 
\widgraph{0.45\linewidth}{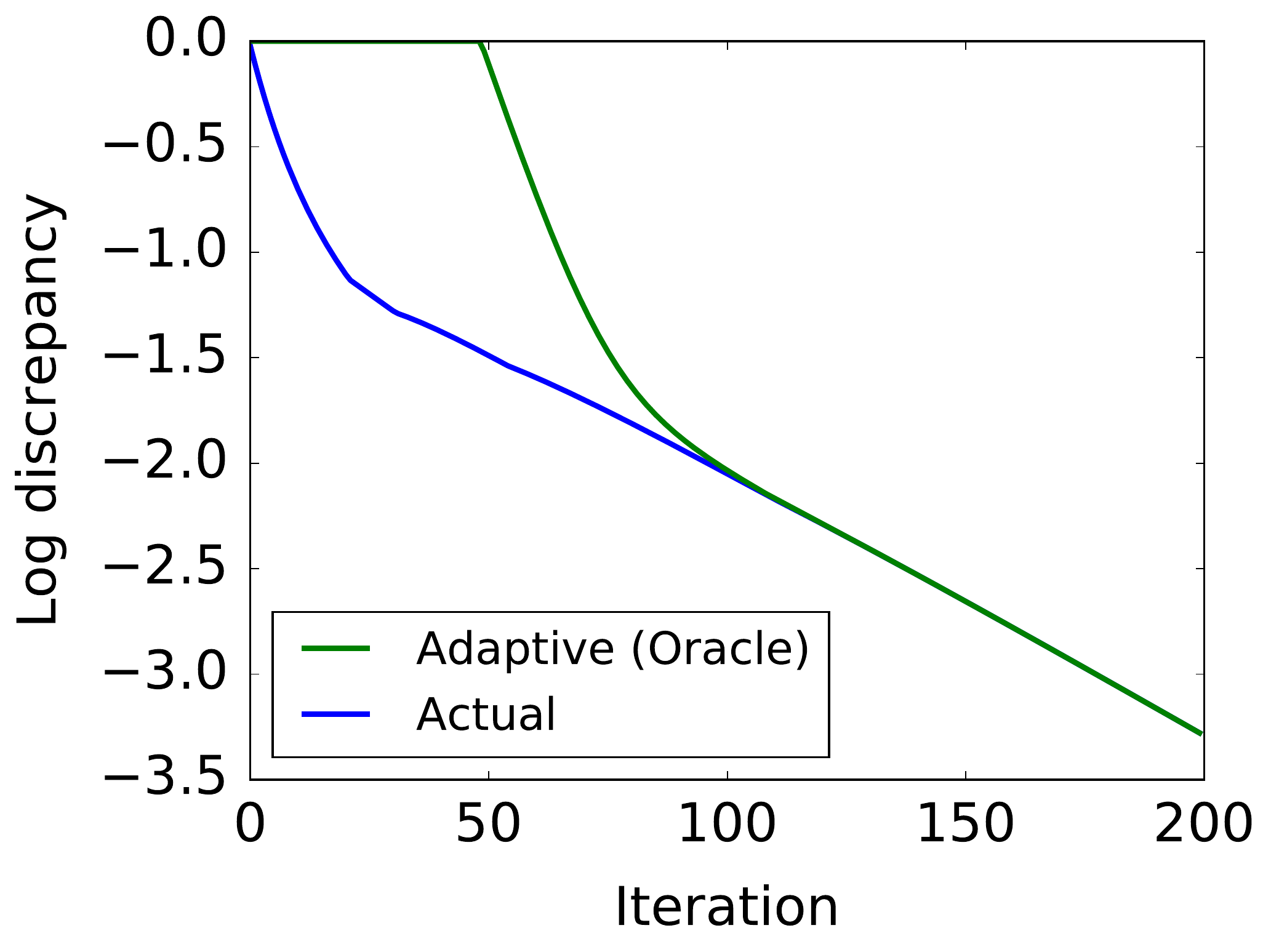}\\
(a) &  (b)
\end{tabular}
\caption{(a) Comparison of the $f$-discrepancy $d_{f}$ and the total
  variation discrepancy $d_{\TV}$ over the first $100$ iterations of
  MCMC. Clearly the function mixes much faster than the overall
  chain. (b) The predicted value of $\log{d_{f}}$ (according to the
  $f$-discrepancy oracle bound---Lemma \ref{lem:mix-gap-all-oracle})
  plotted against the true value. The predictions are close to sharp
  throughout and become sharp at around $100$ iterations.}
\label{fig:mixture-DFvsTV}
\end{center}
\end{figure}

\begin{table}[h]
\begin{center}
\begin{tabular}{c|c|c}
Bound type & $T_{f}\big(0.01\big)$ & $T_{f}\big(10^{-6}\big)$
\\ \hline Uniform & 31,253 & 55,312 \\ Function-Specific & 25,374 & 49,434
\\ Function-Specific (Oracle) & 98 & 409 \\ Actual & 96 & 409
\end{tabular}
\caption{Comparison of bounds on $T_{f}\big(\delta\big)$ for different
  values of $\delta$. The uniform bound corresponds to the bound
  $T_{f}\big(\delta\big) \leq T\big(\delta\big)$, the latter of which
  can be bounded by the total variation bound.  The function-specific bounds
  correspond to Lemmas~\ref{lem:mix-gap-all}
  and~\ref{lem:mix-gap-all-oracle}, respectively. Whereas the uniform
  and non-oracle $f$-discrepancy bounds make highly conservative predictions,
  the oracle $f$-discrepancy bound is nearly sharp even for $\delta$ as large
  as $0.01$.\label{tab:mixture-Tf-bds}}
\end{center}
\end{table}
The mixture setting also provides a good illustration of how the
function-specific Hoeffding bounds can substantially improve on the
uniform Hoeffding bound.  In particular, let us compare the
$T_{f}$-based Hoeffding bound (Theorem~\ref{thm:hoeffding-eps2}) to
the uniform Hoeffding bound established by~\citet{Leo04Hoeffding}. At
equilibrium, the penalty for non-independence in our bounds is
$(2T_{f}(\epsilon/2))^{-1}$ compared to roughly
$\gamma^{-1}_{\ast}$ in the uniform bound. Importantly, however, our
concentration bound applies unchanged even when the chain has not
equilibrated, provided it has approximately equilibrated with respect
to $f$.  As a consequence, our bound only requires a burn-in of
$T_{f}(\epsilon/2)$, whereas the uniform Hoeffding
bound does not directly apply for any finite burn-in. Table~\ref{tab:mixture-Tf-bds} illustrates the 
size of these burn-in times in practice. This issue can
be addressed using the method of \citet{Pau12Conc}, but at the cost of
a burn-in dependent penalty $d_{\TV}(T_0) = \sup_{\pi_{0}}
d_{\TV}(\pi_{n},~\pi)$:
\begin{align}
\label{eq:uniform-hoeffding-burn-in}
\P \Big[ \frac{1}{N - T_{0}} \sum_{n = T_0}^{N} f(X_{n}) \geq
  \mu + \epsilon \Big] \leq d_{\TV}\big(T_{0}\big) +
 \exp \Big \{ -\frac{\gamma_{0}}{2 \big(1 - \gamma_{0} \big)} \cdot
\epsilon^2  \big[N - T_0 \big] \Big \},
\end{align}where we have let $T_{0}$ denote the burn-in time. Note that a matching
bound holds for the lower tail.
For our experiments, we computed the tightest version of the
bound~\eqref{eq:uniform-hoeffding-burn-in}, optimizing $T_0$ in the
range $\big[0,~10^{5}\big]$ for each value of the deviation
$\epsilon$. Even given this generosity toward the uniform bound, the
function-specific bound still outperforms it substantially, as
Figure~\ref{fig:mixture-conc-bds} shows. 

\begin{figure}[h!]
\centering
\widgraph{0.5\linewidth}{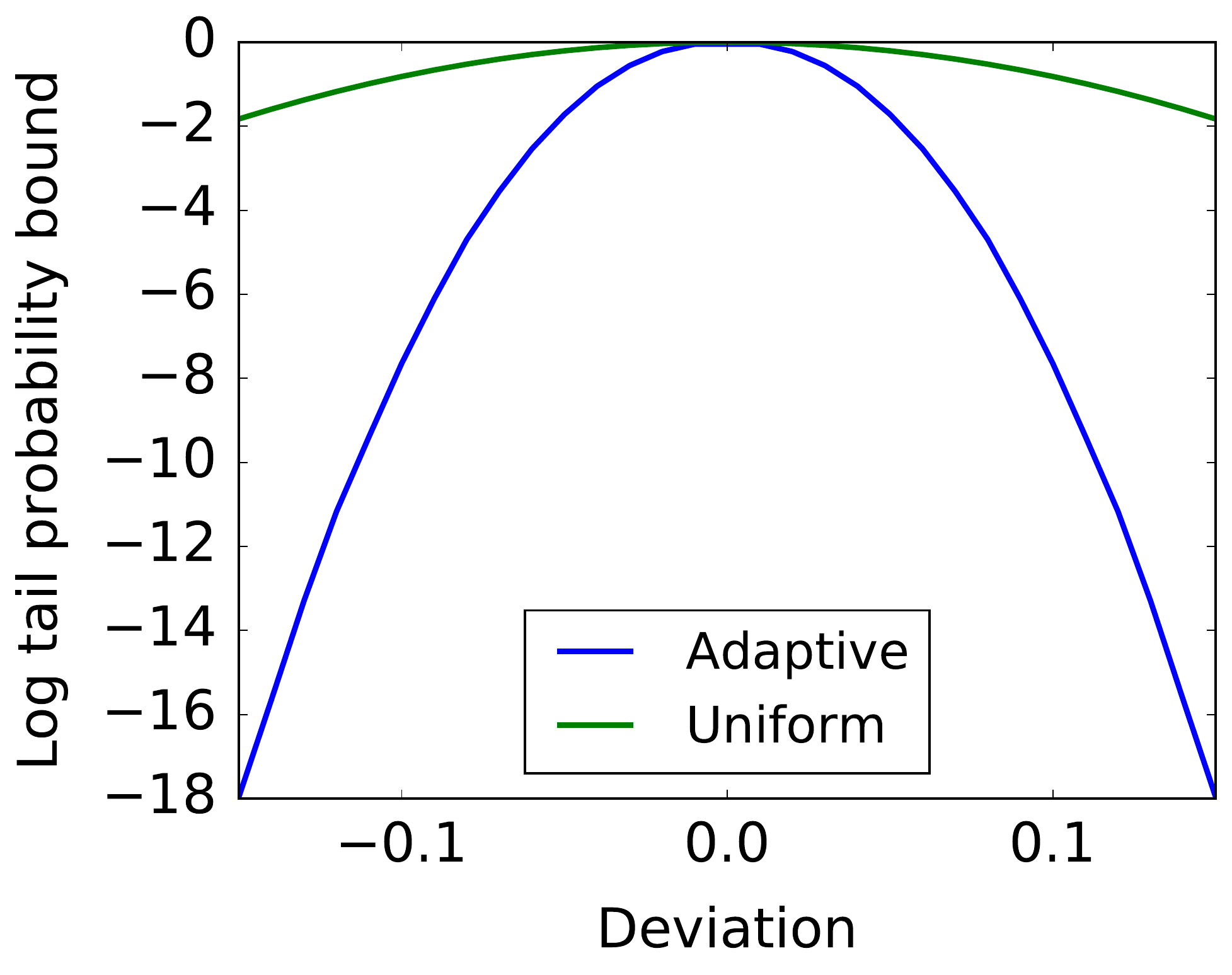}
\caption{Comparison of the (log) tail probability bounds provided by
  the uniform Hoeffding bound due to~\cite{Leo04Hoeffding} with one
  version of our function-specific Hoeffding bound
  (Theorem~\ref{thm:hoeffding-eps2}).  Plots are based on $N = 10^{6}$
  iterations, and choosing the optimal burn-in for the uniform bound
  and a fixed burn-in of $409 \geq T_{f} \big( 10^{-6} \big)$
  iterations for the function-specific bound. The function-specific
  bound improves over the uniform bound by orders of
  magnitude.\label{fig:mixture-conc-bds}}
\end{figure}

For the function-specific
bound, we used the function-specific oracle bound
(Lemma~\ref{lem:mix-gap-all-oracle}) to bound
$T_{f}\big(\frac{\epsilon}{2}\big)$; this nearly coincides with
the true value when $\epsilon \approx 0.01$ but deviates slightly for
larger values of $\epsilon$.



\section{Proofs of main results}
\label{sec:proofs}

This section is devoted to the proofs of the main results of this
paper.

\subsection{Proof of Theorem~\ref{thm:hoeffding-eps2}} 
\label{subsec:proofs-conc-master}

We begin with the proof of the master Hoeffding bound from
Theorem~\ref{thm:hoeffding-eps2}.  At the heart of the proof is the
following bound on the moment-generating function (MGF) for the sum of
an appropriately thinned subsequence of the function values
$\{f(X_n)\}_{n=1}^\infty$.  In particular, let us introduce the
shorthand notation $\Xtil_{m,t} \mydefn X_{(m - 1) \Tf(\epsilon/2) +
  t}$ and $N_0 \mydefn N / T_f(\frac{\epsilon}{2})$.  With this
notation, we have the following auxiliary result:
\begin{lem}[Master MGF bound]
For any scalars $\newalpha \in \real$, $\epsilon \in (0,1)$, and
integer $t \in \big [0, \Tf(\frac{\epsilon}{2}) \big)$, we have
\begin{align}
 \E \left [ \exp \Big( \newalpha \sum_{m = 1}^{N_0} f \big(\Xtil_{m,t}
   \big) \Big) \right] & \leq \exp \left \{\big[
   \frac{1}{2}\newalpha\epsilon + \newalpha \mu + \frac{1}2
   \newalpha^2 \big] \cdot N_0 \right \}.
\end{align}
\label{lem:mgf-master}
\end{lem}
\noindent See Section~\ref{SecProofLemMGFMaster} for the proof of this
claim. 
Recalling the definition of $\Xtil_{m,t}$, we have
\begin{align*}
\E \left[ e^{\alpha\sum_{n = 1}^N f(X_n) } \right] & = \E \left[ \exp
  \Big \{ \alpha \sum_{t = 1}^{\Tf(\epsilon/2)} \sum_{m=1}^{N_0}
  f(\Xtil_{m,t}) \Big \} \right] \\
& = \E \left [ \exp \Big \{ \alpha \Tf(\epsilon/2) \big[
    \frac{1}{\Tf(\epsilon/2)} \sum_{t = 1}^{\Tf(\epsilon/2)} \sum_{m =
      1}^{N_0} f(\Xtil_{m,t}) \big] \Big \} \right ] \\
& \leq \frac{1}{\Tf \big( \epsilon/2 \big)} \sum_{t = 0}^{ \Tf \big(\epsilon/2
  \big) - 1} \E \left[ \exp \Big \{ \alpha \Tf \big(\epsilon/2\big)
  \sum_{m = 1}^{N_0} f \big( \Xtil_{m,t} \big) \Big \} \right],
\end{align*}
where the last inequality follows from Jensen's inequality, as applied
to the exponential function.
Applying Lemma~\ref{lem:mgf-master} with $\newalpha = \alpha\Tf\big(\epsilon/2\big)$, we conclude
\begin{align*}
\E \big[e^{\alpha\sum_{n = 1}^{N} f(X_{n}) }\big] \leq \exp \Big \{
\big[\frac{1}{2}\alpha\Tf\big(\frac{\epsilon}{2}\big)\epsilon + \alpha\Tf\big(\frac{\epsilon}{2}\big)\mu + \frac{1}2
  \alpha^2 \Tf^2\big(\frac{\epsilon}{2}\big) \big] \cdot N_0 \Big \},
\end{align*}
valid for $\alpha > 0$.  By exponentiating and applying Markov's
inequality, it follows that
\begin{align*}
 \P \left[ \frac{1}{N}\sum_{n = 1}^{N} f\big(X_n\big) \geq \mu +
   \epsilon \right] & \leq e^{-\alpha(\mu + \epsilon)} \E
 \big[e^{\alpha\sum_{n = 1}^{N} f(X_{n}) }\big] \\
& \leq \exp \left \{\frac{1}{2} \cdot \Big[ -\alpha \:
   \Tf(\frac{\epsilon}{2}) \epsilon + \alpha^2 \:
   \Tf^2(\frac{\epsilon}{2}) \Big] \: N_0(\frac{\epsilon}{2}) \right
 \}.  \nonumber
\end{align*}

The proof of Theorem~\ref{thm:hoeffding-eps2} follows by taking
$\alpha = \frac{\epsilon}{2\Tf\left(\frac{\epsilon}{2}\right)}$ since
\begin{align*}
\P \left[ \frac{1}{N}\sum_{n = 1}^{N} f\big(X_n\big) \geq \mu +
  \epsilon \right] & \leq \exp \left \{ \frac{1}{2} \cdot \big[ -\frac{\epsilon^{2}}{2}
  + \frac{\epsilon^2}{4} \big] \cdot N_0 \right \} \\
& \leq \exp \left \{ -\frac{\epsilon^2 N_0}{8} \right \} \\
& = \exp \left \{ - \frac{\epsilon^2  N}{8 \Tf
  \big(\frac{\epsilon}{2}\big)} \right \}.
\end{align*}


\subsubsection{Proof of Lemma~\ref{lem:mgf-master}}
\label{SecProofLemMGFMaster}

For the purposes of the proof, fix $t$ and let $W_{m} =
\tilde{X}_{m,t}$. For convenience, also define a dummy constant random
variable $W_{0} \mydefn 0$. Now, by assumption, we have
\begin{align*}
\left|\E\left[f\left(W_{1}\right)\right] - \mu\right| \leq
\frac{\epsilon}{2} ~~ \text{and} ~~ \left|\E\left[f\left(W_{m +
    1}\right)~|~W_{m}\right] - \mu\right| \leq \frac{\epsilon}{2}.
\end{align*}
We therefore have the bound
\begin{align}
\label{eq:master-MGF-martingale}
\E \left[e^{\alpha\sum_{m} f\left(W_{m}\right)}\right] & \leq \E
\left[\prod_{m = 1}^{N_0} e^{\alpha\left[f\left(W_{m}\right) -
      \E\left[f\left(W_{m}\right)~|~W_{m - 1}\right]\right]}\right]
\cdot e^{\alpha\mu N_{0} + \frac{\alpha\epsilon N_{0}}{2}}.
\end{align}
But now observe that the random variables $\Delta_{m} =
f\left(W_{m}\right) - \E\left[f\left(W_{m}\right)~|~W_{m - 1}\right]$
are deterministically bounded in $[-1,~1]$ and zero mean conditional
on $W_{m - 1}$. Moreover, by the Markovian property, this implies that
the same is true conditional on \mbox{$W_{< m} \mydefn W_{0:\left(m -
    1\right)}$.} It follows by standard MGF bounds that
\begin{align*}
\E\left[e^{\alpha\Delta_{m}}~|~W_{< m}\right] \leq
e^{\frac{\alpha^{2}}{2}} .
\end{align*}
Combining this bound with inequality~\eqref{eq:master-MGF-martingale},
we conclude that
\begin{align*}
\E\left[e^{\alpha\sum_{m} f\left(W_{m}\right)}\right] & \leq
e^{\frac{\alpha^{2}}{2} \cdot N_{0}}\cdot e^{\alpha\mu N_{0} +
  \frac{\alpha\epsilon N_{0}}{2}},
\end{align*}
as claimed.


\subsection{Proofs of 
Corollaries~\ref{cor:hoeffding-derived-eps2}
and~\ref{cor:hoeffding-derived-Jf}}
\label{subsec:proofs-conc-derived}

In this section, we prove the derived Hoeffing bounds stated in
Corollaries~\ref{cor:hoeffding-derived-eps2}
and~\ref{cor:hoeffding-derived-Jf}.


\subsubsection{Proof of Corollary~\ref{cor:hoeffding-derived-eps2}}
\label{SecProofThmHoeffdingDerivedEps2}

The proof is a direct application of
Theorem~\ref{thm:hoeffding-eps2}. Indeed, it suffices to note that if
$\epsilon \leq \frac{2\lambda_{f}}{\sqrt{\pimin}}$, then
\begin{align*}
\Tf \big(\frac{\epsilon}{2}\big) \leq \frac{\log \big(
  \frac{2}{\epsilon\sqrt{\pimin}}\big)}{\log \big(
  \frac{1}{\lambda_f} \big)} = \frac{ \log \big( \frac{2}{\epsilon}
  \big) + \frac{1}2  \log \big(
  \frac{1}{\pimin} \big)}{\log \big(\frac{1}{\lambda_f}\big)},
\end{align*}
which yields the first bound. Turning to the second bound, note that
if $\epsilon  > \frac{2\lambda_{f}}{\sqrt{\pimin}}$, then
equation~\eqref{eq:mix-gap} implies that
$\Tf\big(\frac{\epsilon}{2}\big) = 1$, which establishes the
claim.


\subsubsection{Proof of Corollary~\ref{cor:hoeffding-derived-Jf}}

The proof involves combining Theorem~\ref{thm:hoeffding-eps2} with
Lemma~\ref{lem:mix-gap-all}, using the setting \mbox{$\epsilon =
  2\left(\Delta + \Delta_{J}\right)$.} We begin by combining the
bounds $\lambda_{J} \leq 1$, $d_{\TV}\left(\pi_{0},~\pi_{n}\right)
\leq 1$, $d_{f}\left(\pi_{0},~\pi_{n}\right) \leq 1$, and
$\E_{\pi}\left[f^{2}\right] \leq 1$ with the claim of
Lemma~\ref{lem:mix-gap-all} so as to find that
\begin{align*}
d_{f}\left(\pi_{0},~\pi_{n}\right) & \leq \Delta_{J}^{\ast} +
\frac{\lambda_{-J}^{n}}{\sqrt{\pimin}} \; \leq \; \Delta_{J} +
\frac{\lambda_{-J}^{n}}{\sqrt{\pimin}} .
\end{align*}
It follows that
\begin{align*}
\Tf (\frac{\epsilon}{2}) & = \Tf\left(\Delta_{J} + \Delta\right) \leq
\frac{\log \left(\frac{1}{\Delta}\right) + \frac{1}{2} \log
  \left(\frac{1}{\pimin}\right)}{\log(\frac{1}{\lambda_{-J}})} \qquad
\mbox{whenever $\Delta \leq \frac{\lambda_{-J}}{\sqrt{\pimin}}$.}
\end{align*}
Plugging into Theorem~\ref{thm:hoeffding-eps2} now yields the first
part of the bound. On the other hand, if $\Delta >
\frac{\lambda_{-J}}{\sqrt{\pimin}}$, then Lemma~\ref{lem:mix-gap-all}
implies that $\Tf\left(\Delta_{J} + \Delta\right) = 1$, which proves
the bound in the second case.

%
%
%
%

\subsection{Proof of Proposition~\ref{prop:lower-bound}}
\label{subsec:proofs-lower}

In order to prove the lower bound in
Proposition~\ref{prop:lower-bound}, we first require an auxiliary
lemma:

\begin{lem}
\label{lem:lower-bound-helper}
Fix a function $\delta \colon \left(0,~1\right) \rightarrow
\left(0,~1\right)$ with $\delta\left(\epsilon\right) > \epsilon$. For
every constant $c_{0} \geq 1$, there exists a Markov chain $P_{c_0}$
and a function $f_{\epsilon}$ on it such that $\mu = \frac{1}{2}$,
yet, for $N = c_{0}\Tf\left(\delta\left(\epsilon\right)\right)$, and
starting the chain from stationarity,
\begin{align*}
\P_{\pi}\left(\left|\frac{1}{N}\sum_{n = 1}^{N}
f_{\epsilon}\left(X_{n}\right) - \frac{1}{2}\right| \geq
\epsilon\right) \geq \frac{1}{3} .
\end{align*}
\end{lem}

Using this lemma, let us now prove Proposition~\ref{prop:lower-bound}.
Suppose that we make the choices
\begin{align*}
c_{0} & \mydefn \left\lceil\frac{\log
  7}{c_{1}\epsilon^{2}}\right\rceil \geq 1, \quad \mbox{and} \quad
N_{c_{1},\epsilon} \mydefn c_{0}\Tf\left(\epsilon\right),
\end{align*}
in Lemma~\ref{lem:lower-bound-helper}.  Letting $P_{c_0}$ be the
corresponding Markov chain and $f_{\epsilon}$ the function guaranteed
by the lemma, we then have
\begin{align*}
\P_{\pi}\left(\left|\frac{1}{N_{c_1,\epsilon}}\sum_{n =
  1}^{N_{c_1,~\epsilon}} f_{\epsilon}\left(X_{n}\right) -
\frac{1}{2}\right| \geq \epsilon \right) & \geq \frac{1}{3} \; > \;
\frac{2}{7} \; \geq \; 2 \cdot \exp\left(-\frac{c_{1}
  N_{c_1,\epsilon}\epsilon^{2}}{\Tf\left(\delta(\epsilon)\right)}
\right).
\end{align*}

\paragraph{Proof of Lemma~\ref{lem:lower-bound-helper}:}
It only remains to prove Lemma~\ref{lem:lower-bound-helper}, which we
do by constructing pathological function on a chain graph, and letting
our Markov chain be the lazy random walk on this graph.  For the
proof, fix $\epsilon > 0$, let $\delta = \delta\left(\epsilon\right)$
and let $\Tf = \Tf\left(\delta\right)$. Now choose an integer $d > 0$
such that $d > 2c_{0}$ and let the state space be $\Omega$ be the line
graph with $2d$ elements with the standard lazy random walk defining
$P$. We then set
\begin{align*}
f\left(i\right) & = \begin{cases}
\frac{1}{2} - \delta & ~ 1 \leq i \leq d, \\
\frac{1}{2} + \delta &~d + 1 \leq i \leq 2d . 
\end{cases}
\end{align*}
It is then clear that $\Tf = 1$.

Define the bad event
\begin{align*}
\mathcal{E} = \left\lbrace X_{1} \in \left[0,~\frac{d}{2}\right] \cup
\left[\frac{3d}{2},~2d\right]\right\rbrace .
\end{align*}
When this occurs, we have 
\begin{align*}
\left|\frac{1}{N'}\sum_{n = 1}^{N'} f\left(X_{n}\right) -
\frac{1}{2}\right| \geq \delta > \epsilon \qquad \mbox{with
  probability one},
\end{align*}
for all $N' < \frac{d}{2}$. Since $N = c_{0} < \frac{d}{2}$, we can
set $N' = N$.

On the other hand, under $\pi$, the probability of $\mathcal{E}$ is
$\geq \frac{1}{3}$. (It is actually about $\frac{1}{2}$, but we want
to ignore edge cases.) The claim follows immediately.


\subsection{Proofs of confidence interval results}
\label{subsec:proofs-confidence}

Here we provide the proof of the confidence interval corresponding to
our bound (Theorem~\ref{thm:confidence-adaptive-hoeffding}). Proofs of
the claims~\eqref{eq:confidence-unif-hoeffding}
and~\eqref{eq:confidence-clt-be} can be found in
Appendix~\ref{app:proof-confidence}.

As discussed in Section~\ref{subsec:confidence}, we actually prove a
somewhat stronger form of
Theorem~\ref{thm:confidence-adaptive-hoeffding}, in order to guarantee
that the confidence interval can be straightforwardly built using an
upper bound $\tilde{T}_{f}$ on the $f$-mixing time rather than the
true value. Setting $\tilde{T}_{f} = \Tf$ recovers the original
theorem.

Specifically, suppose $\tilde{T}_{f} \colon \N \rightarrow \R_{+}$ is
an upper bound on $\Tf$ and note that the corresponding tail bound
becomes $e^{-\tilde{r}_{N}(\epsilon)/8}$, where
\begin{align*}
\tilde{r}_{N}(\epsilon) = \epsilon^2
\left[\frac{N}{\tilde{T}_{f}\big(\frac{\epsilon}{2}\big)} - 1 \right]
.
\end{align*}
This means that, just as before we wanted to make the rate $r_{N}$ in
equation~\eqref{eq:rN-def} at least as large as $8 \log\frac2
{\alpha}$, we now wish to do the same with $\tilde{r}_{N}$, which
means choosing $\epsilon_{N}$ with
$\tilde{r}_{N}\big(\epsilon_{N}\big) \geq 8 \log \frac2 {\alpha}$. We
therefore have the following result.

\begin{prop}
\label{prop:confidence-adaptive-hoeffding-bdd}
For any width $\epsilon_{N} \in \tilde{r}_{N}^{-1}\big(\big[8
  \log\big(2/\alpha\big),~\infty\big)\big)$, the set
\begin{align*}
\FUNCINT = \left[ \frac{1}{N - T_{f}\big(\frac{\epsilon}{2} \big)}
  \sum_{n = \tilde{T}_{f} \big( \frac{\epsilon}{2} \big)}^{N}
  f\big(X_{n}\big) \pm \epsilon_{N} \right]
\end{align*}
is a $1 - \alpha$ confidence interval for $\mu = \E_{\pi}\big[f\big]$.
\end{prop}
\begin{proof}
For notational economy, let us introduce the shorthands
$\tau_{f}(\epsilon) = \Tf\big(\frac{\epsilon}{2}\big)$ and
$\tilde{\tau}_{f}\big(\epsilon\big) =
\tilde{T}_{f}\big(\frac{\epsilon}{2}\big)$. Theorem~\ref{thm:hoeffding-eps2}
then implies
\begin{align*}
\P \left[ \frac{1}{N - \tilde{\tau}_{f}}\sum_{n =
    \tilde{\tau}_{f}}^{N} f\big(X_n\big) \geq \mu + \epsilon \right] &
\leq \exp\big(-\frac{N - \tau_{f}}{4\tau_{f}} \cdot \epsilon^2 \big)
\\
& \leq \exp\big(-\frac{N - \tilde{\tau}_{f}}{4\tilde{\tau}_{f}}
\cdot \epsilon^2 \big) \\ & =
\exp\big(-\frac{\tilde{r}_{N}\big(\epsilon\big)}{4}\big) .
\end{align*}
Setting $\epsilon = \epsilon_{N}$ yields
\begin{align*}
\P \left[ \frac{1}{N - \tilde{\tau}_{f}}\sum_{n =
    \tilde{\tau}_{f}}^{N} f\big(X_n\big) \geq \mu + \epsilon_{N}
  \right] & \leq \frac{\alpha}{2}.
\end{align*}
The corresponding lower bound leads to an analogous bound on the lower
tail.
\end{proof}

As we did with Corollary~\ref{cor:confidence-adaptive-hoeffding-concrete}, we can
derive a more concrete, though slightly weaker, form of this result
that is more amenable to interpretation. We derive the corollary from the
specialized bound by setting $\tilde{T}_{f} = T_{f}$.

To obtain this bound, define the
following lower bound, in parallel with
equation~\eqref{eq:rN-eta-def}:
\begin{align*}
\tilde{r}_{N}\big(\epsilon\big) \geq
\tilde{r}_{N,\eta}\big(\epsilon\big) \mydefn \epsilon^2
\big[\frac{N}{\tilde{T}_{f}\big(\frac{\eta}{2}\big)} - 1\big],~
\epsilon \geq \eta .
\end{align*}
Since this is a lower bound, we see that whenever $\epsilon_{N} \geq
\eta$ and $\tilde{r}_{N,\eta}\big(\epsilon_{N}\big) \geq 8 \log\frac2
    {\alpha}$, $\epsilon_{N}$ is a valid half-width for a $\big(1 -
    \alpha\big)$-confidence interval for the stationary mean centered
    at the empirical mean.  More formally, we have the following:
\begin{prop}
\label{prop:confidence-adaptive-hoeffding-concrete-bdd} 
Fix $\eta > 0$ and let
\begin{align*}
\epsilon_{N} = \tilde{r}_{N,\eta}^{-1}\big(8 \log \frac2 {\alpha}\big)
& = 2\sqrt{2}\sqrt{\frac{\tilde{T}_{f}\big(\frac{\eta}{2}\big) \cdot
    \log\big (2/ \alpha \big)}{N - \tilde{T}_{f} \big( \frac{
      \eta}{2} \big)}} .
\end{align*}
If $N \geq \tilde{T}_{f}\big(\frac{\eta}{2}\big)$, then $\FUNCINT$
is a $1 - \alpha$ confidence interval for $\mu = \E_{\pi}\big[f\big]$.
\end{prop}
\begin{proof}
By assumption, we have
\begin{align*}
 \eta \leq \epsilon_{N}\big(\eta\big) = 2
 \sqrt{\frac{\tilde{T}_{f}\big(\frac{\eta}{2}\big) \cdot
     \log\big(2/\alpha\big)}{N - \tilde{T}_{f}\big(\frac{\eta
     }{2}\big)}}.
\end{align*}
This implies $\tilde{T}_{f}\big(\frac{\epsilon_{N}}{2}\big) \geq
\tilde{T}_{f}\big(\frac{\eta}{2}\big)$, which yields
\begin{align*}
\tilde{r}_{N} \big(\epsilon_{N}\big) = \epsilon_{N}^2
\big[\frac{N}{\tilde{T}_{f}\big(\frac{\epsilon_{N}}{2}\big)} -
  1\big] \geq \epsilon_{N}^2
\big[\frac{N}{\tilde{T}_{f}\big(\frac{\eta}{2}\big)} - 1\big] =
8 \log\big(2/\alpha\big).
\end{align*}
But now Proposition~\ref{prop:confidence-adaptive-hoeffding-bdd} applies,
so that we are done.
\end{proof}


\subsection{Proofs of sequential testing results}
\label{subsec:proofs-testing}

In this section, we collect various proofs associated with our
analysis of the sequential testing problem.


\subsubsection{Proof of Theorem~\ref{thm:seq-err-all-adapt} for $\algfix$}

We provide a detailed proof when $H_1$ is true, in which case we have
$\mu \leq r - \delta$; the proof for the other case is analogous.
When $H_1$ is true, we need to control the probability
$\P\big(\algfix\big(X_{1:N}\big) = H_{0}\big)$.  In order to do so,
note that Theorem~\ref{thm:hoeffding-eps2} implies that
\begin{align*}
\P\big(\algfix\big(X_{1:N}\big) = H_{0}\big) & =
\P\big(\frac{1}{N}\sum_{n = 1}^{N} f\big(X_{n}\big) \geq r +
\delta\big) \\ & \leq \P\big(\frac{1}{N}\sum_{n = 1}^{N}
f\big(X_{n}\big) \geq \mu + 2\delta\big) \\ & \leq
\exp\big(-\frac{\delta^2 N}{2\Tf\big(\delta\big)}\big).
\end{align*}
Setting $N = \frac{2T_{f}\left(\delta\right)
  \log\big(\frac{1}{\alpha}\big)}{\delta^2 }$ yields the bound $\P
\big(\algfix\big(X_{1:N}\big) = H_{0}\big) \leq \alpha$, as claimed.


\subsubsection{Proof of Theorem \ref{thm:seq-err-all-adapt} for $\algseq$}

The proof is nearly identical to that given by~\cite{Gyo15Test}, with
$\Tf\big(\delta/2\big)$ replacing $\frac{1}{\gamma_{0}}$. We
again assume that $H_{1}$ holds, so $\mu \leq r - \delta$. In this
case, it is certainly true that
\begin{align*}
\err\big(\algseq,~f\big) & = \P\big(\exists k \colon
\algseq\big(X_{1:N_{k}}\big) = H_{0}\big) \\ 
& = \P\big(\exists k \colon \frac{1}{N_{k}}\sum_{n = 1}^{N_{k}}
f\big(X_{n}\big) \geq r + \frac{M}{N_{k}}\big) \\ 
& \leq \sum_{k = 1}^{\infty} \P\big(\frac{1}{N_{k}}\sum_{n = 1}^{N_k}
f\big(X_{n}\big) \geq r + \frac{M}{N_{k}}\big) .
\end{align*}
It follows by Theorem \ref{thm:hoeffding-eps2}, with $\epsilon_{k} =
\delta + \frac{M}{N_{k}}$, that
\begin{align*}
\P\big(\frac{1}{N_{k}}\sum_{n = 1}^{N_k} f\big(X_{n}\big) \geq r +
\frac{M}{N_{k}}\big) & \leq \P\big(\frac{1}{N_{k}}\sum_{n = 1}^{N_k}
f\big(X_{n}\big) \geq \mu + \delta + \frac{M}{N_{k}}\big) \\ & \leq
\exp\big(-\frac{\epsilon_{k}^2 N_{k}}{8\Tf\big(\frac{\epsilon_{k}
  }{2}\big)}\big) \\ & \leq \exp\big(-\frac{\epsilon_{k}^2
  N_{k}}{8\Tf\big(\frac{\delta}{2}\big)}\big).
\end{align*}

In order to simplify notation, for the remainder of the proof, we
define \mbox{$\tau \mydefn 4\Tf(\delta/2)$,} \mbox{$\beta \mydefn
  \frac{\sqrt{\alpha\xi}}{2}$,} and \mbox{$\zeta_{k} \mydefn
  \frac{\delta^2 N_{k}}{2 \tau \log(1/\beta)}$.}  In terms of this
notation, we have $M = \frac{2\tau \log(1/\beta)}{\delta}$, and hence
that
\begin{align*}
\exp\big(-\frac{\epsilon_{k}^2 N_{k}}{2\tau}\big) & =
\exp\big(-\frac{1}{2\tau} \cdot \big(\delta^2 N_{k} + 2\delta M +
\frac{M^2 }{N_{k}}\big)\big) \\ & = \exp\big(-\big[\frac{\delta^2
    N_{k}}{2\tau} + \log\big(1/\beta\big) + \frac{2\tau \log^2
    \big(1/\beta\big)}{\delta^2 N_{k}}\big]\big) \\ & =
\exp\big(-\log\big(1/\beta\big)\big[1 + \zeta_{k} +
  \zeta_{k}^{-1}\big]\big) \\ & = \beta \cdot
\exp\big(-\log\big(1/\beta\big)\big[\zeta_{k} +
  \zeta_{k}^{-1}\big]\big) .
\end{align*}
It follows that the error probability is at most
\begin{align*}
 \beta\sum_{k = 1}^{\infty}
 \exp\big(-\log\big(1/\beta\big)\big[\zeta_{k} +
   \zeta_{k}^{-1}\big]\big).
\end{align*}
We now finish the proof using two small technical lemmas, whose proofs
we defer to Appendix~\ref{app:proofs-proofs-testing}.

\begin{lem}
\label{lem:seq-seq-indiff-adapt-sum} In the above notation, we have
\begin{align*}
\sum_{k = 1}^{\infty} \exp \Big \{ -\log (1/\beta) \big[\zeta_{k} +
  \zeta_{k}^{-1}\big] \Big \}  & \leq 4\sum_{\ell = 0}^{\infty} \exp
\Big \{ -\log(1/\beta) \Big[\big(1 + \xi\big)^{\ell} + \big(1 +
  \xi\big)^{-\ell} \Big] \Big \}.
\end{align*}
\end{lem}

\begin{lem}
\label{lem:seq-seq-indiff-adapt-group} 
For any integer $c \geq 0$, we have
\begin{align*}
 (1 + \xi)^{\ell} + (1 + \xi)^{-\ell} & \geq 2 ( c + 1 ) \quad
  \mbox{for all $\ell \in \Big[\frac{9c}{5\xi}, \frac{9 (c + 1)}{5\xi}
      \Big)$.}
\end{align*}
\end{lem}

Using this bound, and grouping together terms in blocks of size
$\frac{9}{5\xi}$, we find that the error is at most
\begin{align*}
4 \sum_{\ell = 0}^{\infty} \exp\big(-\log\big(1/\beta\big)\big[\big(1
  + \xi\big)^{\ell} + \big(1 + \xi\big)^{-\ell}\big]\big) & \leq
\frac{36}{5\xi} \cdot \sum_{c = 0}^{\infty} \beta^{2\big(c + 1\big)}.
\end{align*}
Since both $\alpha$ and $\xi$ are at most $\frac{2}{5}$, we have
$\beta = \frac{\sqrt{\alpha \xi}}{2} \leq \frac{1}{5}$, and hence the
error probability is bounded as
\begin{align*}
\frac{36\beta}{5\xi}\sum_{c = 0}^{\infty} \beta^{2\big(c + 1\big)}
\leq \frac{36\beta^{3}}{5\xi\big(1 - \beta^2 \big)} \leq
\frac{36\beta^2 }{25\xi\big(1 - \beta^2 \big)}\leq \frac{3\beta^2
}{2\xi} = \frac{3\alpha}{4} < \alpha .  
\end{align*}


\subsubsection{Proof of Theorem~\ref{thm:seq-err-all-adapt} for $\algdiff$}

We may assume that $H_{1}$ holds, as the other case is
analogous. Under $H_{1}$, letting $k_{0}$ be the smallest $k$ such
that $\epsilon_{k} < \infty$, we have
\begin{align*}
\err (\algdiff,~f) & \leq \sum_{k = k_{0}}^{\infty}
\P\left(\hat{\mu}_{N_k} \geq r + \epsilon_{k}\right) \; \leq \;
\sum_{k = k_{0}}^{\infty} \P(\hat{\mu}_{N_k} \geq \mu +
2\epsilon_{k}).
\end{align*}
By Theorem~\ref{thm:hoeffding-eps2}, and the definition of
$\epsilon_{k}$, we thus have
\begin{align*}
\err \left(\algdiff,~f\right) \leq \sum_{k = k_0}^{\infty}
\exp\left(-\frac{N_{k}\epsilon_{k}^{2}}{8T_{f}\left(\frac{\epsilon_{k}}{2}\right)}\right)
& \leq \frac{\alpha}{2}\sum_{k = k_0}^{\infty} \frac{1}{k^2} \\ & =
\frac{\pi^{2}}{12} \, \alpha \; < \; \alpha,
\end{align*}
as claimed.


\subsubsection{Proof of Theorem~\ref{thm:seq-stopping-time-all} for $\algseq$}

We may assume $H_{1}$ holds; the other case is analogous. Note that
\begin{align*}
\E[N] & \leq N_{1} + \sum_{k = 1}^{\infty} \big(N_{k + 1} -
N_{k}\big)\P\big(N > N_{k}\big) \\
& \leq N_{1} + \sum_{k = 1}^{\infty} \big(N_{k + 1} -
N_{k}\big)\P\big(\frac{1}{N_{k}}\sum_{n = 1}^{N_{k}} f\big(X_{n}\big)
\in \big(r - \frac{M}{N_{k}},~r + \frac{M}{N_{k}}\big)\big) \\
& \leq N_{1} + \sum_{k = 1}^{\infty} \big(N_{k + 1} -
N_{k}\big)\P\big(\frac{1}{N_{k}}\sum_{n = 1}^{N_{k}} f\big(X_{n}\big)
> r - \frac{M}{N_{k}}\big) \\
& = N_{1} + \sum_{k = 1}^{\infty} \big(N_{k + 1} -
N_{k}\big)\P\big(\frac{1}{N_{k}}\sum_{n = 1}^{N_{k}} f\big(X_{n}\big)
> \mu + \Delta - \frac{M}{N_{k}}\big) \\
& \leq N_{1} + \sum_{k = 1}^{\infty} \big( N_{k + 1} - N_{k} \big)
\exp \left \{ -\frac{\big(\Delta N_{k} - M\big)_{+}^2
}{8\Tf\big(\delta/2\big)N_{k}} \right \} .
\end{align*}
Our proof depends on the following simple technical lemma, whose proof
we defer to Appendix~\ref{subapp:proof-seq-seq-stopping-time-sum}.

\begin{lem}
\label{lem:seq-seq-stopping-time-sum} 
Under the conditions of Theorem~\ref{thm:seq-stopping-time-all}, we
have
\begin{align}
 \sum_{k = 1}^{\infty} \big(N_{k + 1} - N_{k}\big) \exp \left \{
 -\frac{\big(\Delta N_{k} - M\big)_{+}^2
 }{8\Tf\big(\delta/2\big)N_{k}} \right \} & \leq \big ( 1 + \xi
 \big) \big [ 1 + \int_{N_{1}}^{\infty} h(s) \der s\big],
\end{align}
where $h(s) \mydefn \exp \Big \{ -\frac {(\Delta s -
  M)_{+}^2}{8\Tf(\delta/2) s} \Big \}$.
\end{lem}

Given this lemma, we then follow the argument of~\citet{Gyo15Test} in
order to bound the integral. We have
\begin{align*}
\int_{N_{1}}^{\infty} h\big(s\big) \der s \leq \frac4
    {\Delta}\sqrt{\frac{2\Tf\big(\delta/2\big)M}{\Delta} +
      8\Tf\big(\delta/2\big)}.
\end{align*}
To conclude, note that either $r \geq \Delta$ or $1 - r \geq \Delta$,
since $0 < \mu < 1$, so that $\min\big(\frac{1}{r},~\frac{1}{1 -
  r}\big) \leq \frac{1}{\Delta}$. It follows that
\begin{align*}
 N_{1} \leq \big(1 + \xi\big)N_{0} \leq \frac{(1 + \xi) M}{\Delta}.
\end{align*}
Combining the bounds yields the desired result.


\subsubsection{Proof of Theorem~\ref{thm:seq-stopping-time-all} for $\algdiff$}

For concreteness, we may assume $H_{1}$ holds, as the $H_{0}$ case is
symmetric. We now have that
\begin{align*}
\P \left[ N \geq N_{k} \right] & \leq \P \left[
  \big|\frac{1}{N_{k}}\sum_{n = 1}^{N_k} f\big(X_{n}\big) - r\big|
  \leq \epsilon_{k} \right] \; \leq \; \P \left[ \frac{1}{N_{k}}
  \sum_{n = 1}^{N_k} f\big(X_{n}\big) \geq \mu + \Delta - \epsilon_{k}
  \right].
\end{align*}
For convenience, let us introduce the shorthand
\begin{align*}
 T_{f,k}^{+} & \mydefn \begin{cases} T_{f}\big(\frac{\Delta - \epsilon_{k}
   }{2}\big) &~ \text{if}~ \epsilon_{k} \leq \Delta, \\ 1 &~
   \text{otherwise} .
\end{cases}
\end{align*}
Applying the Hoeffding bound from Theorem~\ref{thm:hoeffding-eps2}, we
then find that
\begin{align*}
\P \left[ N \geq N_{k} \right] & \leq \exp \left \{
-\frac{N_{k}}{8T_{f,k}^{+}} \cdot \big(\Delta -
\epsilon_{k}\big)_{+}^2 \right \}.
\end{align*}
Observe further that
\begin{align*}
\E [N] & = N_{1} + \sum_{k = 1}^{\infty} \big(N_{k + 1} -
N_{k}\big)\P\big(N > N_{k}\big) \\
& \leq N_{k_{0}^{\ast} + 1} + \sum_{k = k_{0}^{\ast} + 1}^{\infty}
\big(N_{k + 1} - N_{k}\big)\P\big(N > N_{k}\big) \\
& \leq \big(1 + \xi \big) \big(N_{0}^{\ast} + 1 \big ) + \sum_{k =
  k_{0}^{\ast} + 1}^{\infty} \big(N_{k + 1} - N_{k} \big) \P\big(N >
N_{k} \big).
\end{align*}
Combining the pieces yields
\begin{align}
\label{eq:ENstar-bd-1}
\E[N] & \leq \big(1 + \xi\big)\big(N_{0}^{\ast} + 1\big) + \sum_{k =
  k_{0}^{\ast} + 1}^{\infty} \big(N_{k + 1} -
N_{k}\big)\exp\big(-\frac{N_{k}}{8T_{f,k}^{+}} \cdot \big(\Delta -
\epsilon_{k}\big)_{+}^2 \big).
\end{align}

The crux of the proof is a bound on the infinite sum, which we pull
out as a lemma for clarity.

\begin{lem}
\label{lem:seq-diff-stopping-time-sum}
The infinite sum~\eqref{eq:ENstar-bd-1} is upper bounded by
\begin{align*} 
\sum_{k = k_{0}^{\ast} + 1}^{\infty} \big(N_{k + 1} -
N_{k}\big)\exp\big(-\frac{N_{k}}{8T_{f,k}^{+}} \cdot \big(\Delta -
\epsilon_{k}\big)_{+}^2 \big) \leq \alpha\cdot \sum_{m =
  1}^{\infty}\exp\big(-m \cdot \frac{\Delta^2
}{32T_{f}\big(\frac{\Delta}{4}\big)}\big).
\end{align*}
\end{lem}
\noindent See Appendix~\ref{subapp:proof-seq-diff-stopping-time-sum}
for the proof of this claim.

Lemma~\ref{lem:seq-diff-stopping-time-sum} then implies that
\begin{align*}
\sum_{k = k_{0}^{\ast} + 1}^{\infty} \big(N_{k + 1} - N_{k}\big)
\exp\big(-\frac{N_{k}}{T_{f}\big(\frac{\Delta}{4}\big)} \cdot
\frac{\Delta^2}{32}\big) & \leq \alpha\cdot \sum_{m =
  1}^{\infty}\exp\big(-m \cdot \frac{\Delta^2
}{32T_{f}\big(\frac{\Delta}{4}\big)}\big) \\
& = \frac{\alpha \exp\big(-\frac{\Delta^2 }{32T_{f}
    \big(\frac{\Delta}{4}\big)}\big)}{1 -
  \exp\big(\frac{\Delta^2}{32T_{f}\big(\frac{\Delta
    }{4}\big)}\big)} \\ & \leq \frac{32\alpha
  T_{f}\big(\frac{\Delta}{4}\big)}{\Delta^2} .
\end{align*}
The claim now follows from equation~\eqref{eq:ENstar-bd-1}.


\section{Discussion}
\label{sec:discuss}

A significant obstacle to successful application of statistical
procedures based on Markov chains---especially MCMC---is the
possibility of slow mixing.  The usual formulation of mixing is in
terms of convergence in a distribution-level metric, such as the total
variation or Wasserstein distance. On the other hand, algorithms like
MCMC are often used to estimate equilibrium expectations over a
limited class of functions. For such uses, it is desirable to build a
theory of mixing times with respect to these limited classes of
functions and to provide convergence and concentration guarantees
analogous to those available in the classical setting, and our paper
has made some steps in this direction.

In particular, we introduced the $f$-mixing time of a function, and
showed that it can be characterized by the interaction between the
function and the eigenspaces of the transition operator. Using these
tools, we proved that the empirical averages of a function $f$
concentrate around their equilibrium values at a rate characterized by
the $f$-mixing time; in so doing, we replaced the worst-case
dependence on the spectral gap of the chain, characteristic of
previous Markov chain concentration bounds, by an adaptive dependence
on the properties of the actual target function.  Our methodology
yields sharper confidence intervals, as well as better rates for
sequential hypothesis tests in MCMC, and we have provided evidence
that the predictions made by our theory are accurate in some real
examples of MCMC and thus potentially of significant practical
relevance.

Our investigation also suggests a number of further questions.  Two important
ones concern the continuous and non-reversible cases. Both arise frequently in statistical
applications---for example, when sampling continuous parameters or when performing
Gibbs sampling with systematic scan---and are therefore of considerable interest. As uniform Hoeffding bounds 
do exist for the continuous case and, more recently, have been established for the non-reversible case, we 
believe many of our conclusions should carry over to these settings, although somewhat different methods of
analysis may be required.

Furthermore, in practical applications, it would be desirable to have methods for
estimating or bounding the $f$-mixing time based on samples.  It would
also be interesting to study the $f$-mixing times of Markov chains that
arise in probability theory, statistical physics, and applied
statistics itself. While we have shown what can be done with spectral
methods, the classical theory provides a much larger arsenal of
techniques, some of which may generalize to yield sharper $f$-mixing
time bounds. We leave these and other problems to future work.


\subsection*{Acknowledgments}
The authors thank Alan Sly and Roberto Oliveira for helpful discussions about the lower bounds and
the sharp function-specific Hoeffding bounds (respectively).
This work was partially supported by NSF grant CIF-31712-23800,
ONR-MURI grant DOD 002888, and AFOSR grant FA9550-14-1-0016. In
addition, MR is supported by an NSF Graduate Research Fellowship and a
Fannie and John Hertz Foundation Google Fellowship.


\bibliographystyle{plainnat} \bibliography{paper}

\begin{appendices}


\section{Proofs for Section~\ref{subsec:mixing-bounds}}
\label{sec:proofs-mix}

In this section, we gather the proofs of the mixing time bounds from
Section~\ref{subsec:mixing-bounds}, namely equations~\eqref{eq:mix-TV}
and~\eqref{eq:mix-gap} and Lemmas~\ref{lem:mix-gap-all}
and~\ref{lem:mix-gap-all-oracle}.


\subsection{Proof of the bound~\eqref{eq:mix-TV}}
\label{subsec:proof-mix-TV}

Recall that
\begin{align*}
d_{f}(p,~q) & = \sup_{f \colon [d] \rightarrow [0,~1]}
\left|\E_{p}\big[f\big(X\big)\big] - \E_{q}\big[ f(Y) \big] \right|.
\end{align*}
It follows from equation~\eqref{eq:mix-gap} that
\begin{align*}
d_{\TV} \big(\pi_{n},~\pi\big) & = \sup_{f \colon [d] \rightarrow
  [0,~1]} d_{f}\big(\pi_n,~\pi\big) \\
& \leq \sup_{f \colon [d] \rightarrow [0,~1]}
\big[\frac{\lambda_{f}^{n}}{\sqrt{\pimin}} \cdot
  d_{f}\big(\pi_{0},~\pi\big)\big] \\
& = \frac{1}{\sqrt{\pimin}} \cdot \lambda_{\ast}^{n} \cdot
d_{\TV}\big(\pi_0,~\pi\big),
\end{align*}
as claimed.

\subsection{Proof of equation~\eqref{eq:mix-gap}}
\label{}
Let $\Dpi = \diag\big(\sqrt{\pi}\big)$. Then the matrix \mbox{$A =
  \Dpi P \Dpi^{-1}$} is symmetric and so has an eigendecomposition of
the form \mbox{$A = \gamma_{1}\gamma_{1}^{T} + \sum_{j = 2}^{d}
  \lambda_{j} \gamma_{j}\gamma_{j}^{T}$.}  Using this decomposition,
we have
\begin{align*}
P = \ones\pi^{T} + \sum_{j = 2}^{d} \lambda_{j}h_{j}q_{j}^{T},
\end{align*}
where $h_j \mydefn \Dpi^{-1}\gamma_j$ and $q_j \mydefn
\Dpi\gamma_j$. Note that the vectors $\{q_j\}_{j=2}^d$ correspond to
the left eigenvectors associated with the eigenvalues
$\{\lambda_j\}_{j=2}^d$.

Now, if we let $\pi_0$ be an arbitrary distribution over $[d]$, we
have
\begin{align*}
d_{f}(\pi_n, ~\pi) & = \big|\pi_{0}^{T}P^{n}f - \pi^{T}P^{n}f\big| \;
\leq \; \big|(\pi_0 - \pi)^{T}P^{n}f\big| .
\end{align*}
Defining $P_f \mydefn \ones\pi^{T} + \sum_{j \in J_f}
\lambda_{j}h_{j}q_{j}^{T}$, we have $ P^{n}f = P_{f}^{n} f$.
Moreover, if we define \mbox{$\tilde{P}_{f} \mydefn \sum_{j \in J_f}
  \lambda_{j}h_{j}q_{j}^{T}$}, and correspondingly
\mbox{$\tilde{A}_{f} \mydefn \Dpi \tilde{P}_{f} \Dpi^{-1}$,} we then
have the relation $\big(\pi_0 - \pi\big)^{T}\tilde{P}_{f} = \big(\pi_0
- \pi\big)^{T}P_{f}$.  Consequently, by the definition of the operator
norm and sub-multiplicativity, we have
\begin{align*}
d_{f}\big(\pi_n,~\pi\big) & \leq \big|\big(\pi_0 -
\pi\big)^{T}\tilde{P}_{f}^{n}f\big| \\
 & \leq \opnorm{\tilde{A}_{f}}^{n} \| \Dpi f \|_2 \big \|
\Dpi^{-1}\big(\pi_0 - \pi\big) \big \|_2 \\
& = \sqrt{\E_{\pi}\big[f^2 \big] \cdot \sum_{i \in [d]}
  \frac{\big(\pi_{0,i} - \pi_i\big)^2 }{\pi_i}} \cdot
\lambda_{f}^{n}d_{f}\big(\pi_0,~\pi\big) .
\end{align*}
In order to complete the proof, let $Z \in \bits^{d}$ denote the
indicator vector $Z_{j} = \ones\big(X_{0} = j\big)$. Observe that the
function
\begin{align*}
r(z) \mydefn \sum_{i \in [d]} \frac{\big(z_{i} - \pi_i\big)^2 }{\pi_i}
\end{align*}
is convex in terms of $z$. Thus, Jensen's inequality implies that
\begin{align*}
\E_{\pi_0}\big[ r(Z)\big] \geq r\big(\E_{\pi_0}\big[Z\big]\big) =
r\big(\pi_0\big) = \sum_{i \in [d]} \frac{\big(\pi_{0,i} -
  \pi_i\big)^2 }{\pi_i} .
\end{align*}
On the other hand, for any fixed value $X_{0} = j$, corresponding to
$Z = e_j$, we have
\begin{align*}
r\big(z\big) = r\big(e_j\big) = \frac{\big(1 - \pi_j\big)^2 }{\pi_j}
+ \sum_{i \neq j} \pi_{i} = \frac{1 - \pi_j}{\pi_j} \leq
\frac{1}{\pimin}.
\end{align*}
We deduce that $d_{f}(\pi_n,~\pi) \leq
\sqrt{\frac{\E_{\pi}[f^2]}{\pimin}} \cdot \lambda_{f}^{n} \cdot
d_{f}(\pi_0,~\pi)$, as claimed.


\subsection{Proof of Lemma~\ref{lem:mix-gap-all}}

We observe that
\begin{align*}
\big |(\pi_0 - \pi)^{T} h_{J}(n) \big| & \leq \| \pi_0 -\pi \|_{1} \;
\|h_{J}(n) \|_{\infty} \\
& = 2 d_{\TV} \big(\pi_0, ~\pi \big) \; \| h_{J}(n)\|_{\infty} \\
& \leq 2d_{\TV}\big(\pi_0,~\pi\big) \; \Big \{ \sum_{j \in J}
\big|\lambda_{j}\big|^{n} \cdot \big|q_{j}^{T}f\big| \cdot \| h_{j}
\|_{\infty} \Big \} \\
& \leq 2 d_{\TV} \big(\pi_0,~\pi\big) \; \Big \{ 2 \big| J \big| \cdot
\max_{j \in J} \big| q_{j}^{T}f\big| \cdot \max_{j \in J} \| h_{j}
\|_{\infty} \Big \},
\end{align*}
as claimed.


\subsection{Proof of Lemma~\ref{lem:mix-gap-all-oracle}}

We proceed in a similar fashion as in the proof of
equation~\eqref{eq:mix-gap}. Begin with the identity proved there,
viz.
\begin{align*}
d_{f} \big(\pi_n,~\pi\big) = \big| (\pi_0 - \pi)^{T}
\tilde{P}_{f}^{n}f\big|,
\end{align*}
where $\tilde{P}_{f} = \sum_{j \in J_f}
\lambda_{j}h_{j}q_{j}^{T}$. Now decompose $\tilde{P}_{f}$ further into
\begin{align*}
P_{J} = \sum_{j \in J} \lambda_{j}h_{j}q_{j}^{T} ~~ \text{and}
~~P_{-J} = \sum_{j \in J_{f}\setminus J} \lambda_{j}h_{j}q_{j}^{T}.
\end{align*}
Note also that $\tilde{P}_{f}^{n} = P_{J}^{n} + P_{-J}^{n}$. We thus
find that
\begin{align*}
d_{f}(\pi_n,~\pi) & \leq \big| (\pi_0 - \pi)^{T} P_{J}^{n}f\big| +
\big| (\pi_0 - \pi)^{T} P_{-J}^{n} f \big|.
\end{align*}
Now observe that $P_{J}^{n}f = h_{J} (n)$, so $\big| (\pi_0 - \pi)^{T}
P_{J}^{n} f \big| = \big| (\pi_0 - \pi)^{T} h_{J} (n)|$.  On the other
hand, the second term can be bounded using the argument from the proof
of equation~\eqref{eq:mix-gap} to obtain
\begin{align*}
\big|(\pi_0 - \pi)^{T} P_{-J}^{n}f\big| \leq
\sqrt{\frac{\E_{\pi}\big[f^2 \big]}{\pimin}} \cdot
\lambda_{-J_{\delta}}^{n} \cdot \df(\pi_0,~\pi),
\end{align*}
as claimed.


\section{Proofs for Section~\ref{subsec:example-cycle}}
\label{app:proof-example-cycle}

In this section, we provide detailed proofs of the
bound~\eqref{eq:conc-random}, as well as the other claims about the
random function example on $C_{2d}$.

\begin{prop}
\label{prop:Tf-random}
Let $f \colon [d] \rightarrow [0,~1]$ with $f(i) \sim \tau$ iid from
some distribution on $[0,~1]$. There exists a universal constant
$c_{0} > 0$ such that with probability $\geq 1 -
\frac{\delta^{\ast}}{128\sqrt{d \log{d}}}$ over the randomness $f$, we
have
\begin{align*}
 T_{f}\big(\delta\big) \leq
 \frac{c_{0}d\log{d}\log{\frac{128d}{\delta}}}{\delta^2 } \qquad
 \mbox{for all $0 < \delta \leq \delta^{\ast}$.}
\end{align*}
\end{prop}
\begin{proof}
We proceed by defining a ``good event'' $\Event_{\delta}$, and then
showing that the stated bound on $\Tf(\delta)$ holds conditioned on
this event.  The final step is to show that $\P[\Event_\delta]$ is
suitably close to one, as claimed.

The event $\Event_{\delta}$ is defined in terms of the interaction
between $f$ and the eigenspaces of $P$ corresponding to eigenvalues
close to $1$. More precisely, denote the indices of these eigenvalues
by
\begin{align*}
J_{\delta} & \mydefn \Big \{ j \in \{1, \ldots, 2d-1\} \; \mid \; j
\leq 4\delta\sqrt{\frac{d}{\log{d}}}~~ \text{or} ~~ j \geq 2d -
4\delta\sqrt{\frac{d}{\log{d}}} \Big \}.
\end{align*}
The good event $\Event_\delta$ occurs when $f$ has small inner product
with all the corresponding eigenfunctions---that is
\begin{align*}
\Event_{\delta} & \mydefn \Big \{ \max_{j \in J_{\delta}} |q_j^T f|
\leq 2 \sqrt{\frac{10 \log d}{d} } \Big \}.
\end{align*}
Viewed as family of events indexed by $\delta$, these events form a
decreasing sequence.  (In particular, the associated sequence of sets
$J_{\delta}$ is increasing in $\delta$, in that whenever $\delta \leq
\delta^{\ast}$, we are guaranteed that $J_{\delta} \subset
J_{\delta^{\ast}}$.) 


\paragraph{Establishing the bound conditionally on $\Event_\delta$:}

We now exploit the spectral properties of the transition matrix to
bound $\Tf$ conditionally on the event $\Event_\delta$.  Recall that
the lazy random walk on $C_{2d}$ has eigenvalues $\lambda_{j} =
\frac{1}{2} \big( 1 + \cos(\frac{\pi j}{d}) \big)$ for $j \in [d]$,
with corresponding unit eigenvectors
\begin{align*}
v_{j}^{T} = \frac{1}{\sqrt{2d}} \begin{pmatrix} 1 & \omega_{j} &
  \cdots & \omega_{j}^{2d - 1} \end{pmatrix},~~ \omega_{j} \colon =
e^{\frac{\pi i j}{d}}.
\end{align*}
(See ~\citep{Lev08Markov} for details.)  We note that this
diagonalization allows us to write \mbox{$P = \ones\pi^{T} + \sum_{j =
    1}^{2d - 1} \lambda_{j} h_{j}q_{j}^{T}$,} where $h_{j} = \sqrt{2d}
\cdot v_{j}$ and $q_{j} = \frac{v_{j}}{\sqrt{2d}}$, where we have used
the fact that $\diag\big(\sqrt{\pi}\big) = \frac{1}{\sqrt{2d}} \cdot
I$. Note that $\|h_{j}\|_{\infty} = 1$.

Combining Lemma~\ref{lem:mix-gap-all} with the bounds
$\lambda_{J_{\delta}} \leq 1$, $\|h_{j}\|_{\infty} \leq 1$, and
$\big|J_{\delta}\big| \leq 8\delta\sqrt{\frac{d}{\log{d}}}$, we find
that
\begin{align*}
 d_{f} (\pi_n,~\pi) & \leq 16\delta\sqrt{\frac{d}{\log{d}}} \cdot
 \max_{j \in J} \big|q_{j}^{T}f\big| + \sqrt{d} \cdot
 \lambda_{-J_{\delta}}^{n}.
\end{align*}
Therefore, when the event $\Event_{\delta}$ holds, we have
\begin{align}
\label{eq:df-random-good-event-1}
d_{f}(\pi_{n},~\pi) & \leq 32\sqrt{10} \cdot \delta + \sqrt{d} \cdot
\lambda_{-J_{\delta}}^{n}.
\end{align}
In order to conclude the argument, we use the fact that
\begin{align*}
\lambda_{-J_{\delta}} = \frac{1 + \max_{j \in J_{f}\setminus
    J_{\delta}} \cos\big(\frac{\pi j}{d}\big)}2 \leq \frac{1 +
  \cos\big(\frac{\pi j_{0}}{d}\big)}2 ,
\end{align*}
where $j_{0} = 4\delta\sqrt{\frac{d}{\log{d}}}$. On the other hand, we
also have
\begin{align*}
\cos\big(\pi x\big) \leq 1 - \frac{\pi^2 x^2 }2 +
\frac{\pi^{4}x^{4}}{24} \leq 1 - \frac{\pi^2 x^2 }{12}, \qquad
\mbox{for all $|x| \leq 1$,}
\end{align*}
which implies that
\begin{align*}
\lambda_{-J_{\delta}} \leq 1 - \frac{2\pi^2 \delta^2 }{3d\log{d}} \leq
\exp\big(-\frac{2\pi^2 \delta^2 }{3d\log{d}}\big).
\end{align*}
Together with equation~\eqref{eq:df-random-good-event-1}, this bound
implies that for $n \geq \frac{3d \log{d}
  \log{\frac{d}{\delta}}}{2\pi^2 \delta^2 }$, we have $\sqrt{d}
\lambda_{-J_{\delta}}^{n} \leq \delta$, whence
\begin{align*}
d_{f}(\pi_{n},~\pi) \leq (32\sqrt{10} + 1) \, \delta \leq 128 \delta.
\end{align*}
Replacing $\delta$ by $\frac{\delta}{128}$ throughout, we conclude
that for
\begin{align*}
n \geq \frac{3 \, (128)^2 d \log{d} \log{\frac{128d}{\delta}}}{2 \pi^2
  \delta^2 } = \frac{3 \cdot 2^{13}}{\pi^2 } \cdot \frac{d\log{d} \log
  \frac{128d}{\delta}}{\delta^2},
\end{align*}
we have $d_{f}(\pi_n,~\pi)  \leq \delta$ with probability at least
$\P\left(\Event_{\delta/128}\right)$.


\paragraph{Controlling the probability of $\Event_\delta$:}
It now suffices to prove $\P(\Event_{\delta}) \geq 1 -
\frac{\delta}{\sqrt{d \log{d}}}$, since this implies that $\P
\left(\Event_{\delta/128}\right) \geq 1 - \frac{\delta}{128\sqrt{d
    \log{d}}}$, as required.  In order to do so, observe that the
vectors $\{q_j\}_{j=1}^d$ are rescaled versions of an orthonormal
collection of eigenvectors, and hence
\begin{align*}
\E\big[q_{j}^{T}f\big] = \E_{\nu}\left[\mu\right] \cdot q_{j}^{T}\ones = 0.
\end{align*}
We can write the inner product as \mbox{$q_j^T f = A_{j} + i B_{j}$,}
where $(A_j, B_j)$ are a pair of real numbers.  The triangle
inequality then guarantees that \mbox{$|q_j^T f| \leq |A_j| + |B_j|$,}
so that it suffices to control these two absolute values.

By definition, we have
\begin{align*}
A_{j} = \frac{1}{2d} \sum_{\ell = 0}^{2d - 1} f\big(\ell\big) \cdot
\cos\big(\frac{\pi j \ell}{d}\big),
\end{align*}
showing that it is the sum of sub-Gaussian random variables with
parameters $\sigma_{\ell, j}^2 = \cos^2 \big(\frac{\pi j \ell}{d}\big)
\leq 1$.  Thus, the variable $A_{j}$ is sub-Gaussian with parameter at
most $\sigma_{j}^2 \leq \frac{1}{2d}$. A parallel argument applies to
the scalar $B_{j}$, showing that it is also sub-Gaussian with
parameter at most $\sigma_j^2$.

By the triangle inequality, we have \mbox{$|q_{j}^{T}f| \leq |A_{j}| +
  |B_{j}|$,} so it suffices to bound $\big|A_{j}\big|$ and $|B_{j}|$
separately.  In order to do so, we use sub-Gaussianity to obtain
\begin{align*}
\P\big(\max_{j \in J} |A_{j}| \geq r \big) & \leq |J| \cdot
e^{-\frac{r^2 }2 } \leq 8 \delta\sqrt{\frac{d}{\log{d}}} \cdot
e^{-\frac{d r^2 }2 }.
\end{align*}
With $r \mydefn \sqrt{\frac{2 \log{16d}}{d}}$, we have
\begin{align*}
\P\big(\max_{j \in J_{\delta}} \big|A_{j}\big| \geq \sqrt{\frac{2
    \log{16d}}{d}}\big) \leq \frac{\delta}{2\sqrt{d \log{d}}}.
\end{align*}
Applying a similar argument to $B_{j}$ and taking a union bound, we
find that
\begin{align*}
\P\big(\max_{j \in J_{\delta}} \big|q_{j}^{T}f\big| \geq
2\sqrt{\frac{2 \log{16d}}{d}}\big) \leq \frac{\delta}{\sqrt{d
    \log{d}}} .
\end{align*}

Since $2\sqrt{\frac{2\log{16d}}{d}} \leq 2 \sqrt{\frac{10 \log d}{d}}$
for $d \geq 2$, we deduce that
\begin{align*}
1 - \P\left(\Event_{\delta}\right) = \P\left(\max_{j \in J_{\delta}}
\left|q_{j}^{T}f\right| \geq 2 \sqrt{\frac{10 \log d}{d}}\right) \leq
\frac{\delta}{\sqrt{d \log{d}}},
\end{align*}
as required.
\end{proof}
The concentration result now follows.


\begin{prop} 
\label{prop:conc-random}
The random function $f$ on $C_{2d}$ defined in
equation~\eqref{eq:f-random} satisfies the mixing time and tail bounds
\begin{align*}
T_{f} \big(\frac{\epsilon}{2}\big) & \leq \frac{c_{0} d
  \log{d}\big[\log{d} + \log\big(\frac{1}{\epsilon^2
    }\big)\big]}{\epsilon^{2}}, \\ \intertext{and} \P \left[
  \frac{1}{N}\sum_{n = \Tf\big(\epsilon/2\big)}^{N + \Tf
    (\epsilon/2)} f\big(X_n\big) \geq \mu + \epsilon \right] &
\leq \exp\big(-\frac{c_{1}\epsilon^{4}N}{d \log{d}\big[\log
    (\frac{1}{\epsilon}) + \log{d}\big]}\big).  \\
\end{align*}
with probability at least $1 - \frac{c_{2}\epsilon^2 }{\sqrt{d
    \log{d}}}$ over the randomness of $f$
provided $\epsilon \geq c_{3}\big(\frac{\log{d}}{d}\big)^{1/2}$, where
$c_{0}, c_{1}, c_{2}, c_{3} > 0$ are universal constants.
\end{prop}
\begin{proof}
We first note that from the proof of Proposition~\ref{prop:Tf-random},
we have the lower bound \mbox{$1 - \lambda_{-J_{\delta}} \geq
  \frac{c_{4}\delta^2 }{d\log{d}}$,} valid for all \mbox{$\delta \in
  (0,1)$.}  The proof of the previous proposition guarantees that
$\Delta_{J}^{\ast} \leq 32\sqrt{10}\delta$, so setting $\delta =
\frac{\epsilon }{128\sqrt{10}}$ yields
\begin{align*}
\frac{\epsilon}{4} = 32\sqrt{10}\delta \geq \Delta_{J}^{\ast},
\quad \mbox{and} \quad 1 - \lambda_{-J_{\delta}} \geq
\frac{c_{4}'\epsilon^{2}}{d \log{d}}.
\end{align*}
Now, by Proposition~\ref{prop:Tf-random}, there is a universal
constant $c_5 > 0$ such that, with probability at least $1 -
\frac{\delta}{128\sqrt{d \log d}}$, we have
\begin{align*}
T_{f}\big(\delta'\big) \leq \frac{c_{5} d
  \log{d}\log{d/\delta'}}{\big(\delta'\big)^2 } \qquad \mbox{for all $\delta' \geq \delta$.}
\end{align*}
In particular, we have
\begin{align*}
 T_{f}\big(\frac{\epsilon}{2}\big) \leq \frac{c_2 ' d
   \log{d}\log{d/\epsilon}}{\epsilon^{2}}
\end{align*}
with this same probability. Thus, we have this bound on $T_{f}$ with
the high probability claimed in the statement of the proposition.

We now finish by taking $\Delta = \frac{\epsilon}{4}$ in
Corollary~\ref{cor:hoeffding-derived-Jf}. Noting that
\mbox{$\Delta_{J} + \Delta = \frac{\epsilon}{2}$} and \mbox{$1 -
  \lambda_{-J} \geq \frac{c_{4}'\epsilon^{2}}{d \log{d}}$} completes
the proof.
\end{proof}


\section{Proofs for Section~\ref{subsec:confidence}}\label{app:proof-confidence}

We now prove correctness of the confidence intervals based on the
uniform Hoeffding bound~\eqref{eq:unif-hoeffding}, and the
Berry-Esseen bound~\eqref{eq:markov-clt-be}.


\subsection{Proof of claim~\eqref{eq:confidence-unif-hoeffding}}
\label{appsec:confidence-unif-hoeffding}

This claim follows directly from a modified uniform Hoeffding bound,
due to~\cite{Pau12Conc}.  In particular, for any integer \mbox{$T_{0}
  \geq 0$,} let $d_{\TV}(T_0) = \sup_{\pi_0} d_{\TV}( \pi_{0} P^{T_0},
~\pi )$ be the worst-case total variation distance from stationarity
after $T_0$ steps. Using this notation, \cite{Pau12Conc} shows that
for any starting distribution $\pi_0$ and any bounded function $f
\colon [d] \rightarrow [0,~1]$, we have
\begin{align}
\label{EqnPaulinHoeffding}
 \P\big(\big|\frac{1}{N - T_0}\sum_{n = T_0 + 1}^{N} f\big(X_{n}\big)
 - \mu\big| \geq \epsilon\big) \leq 2
 \exp\big(-\frac{\gamma_{0}}{2\big(2 - \gamma_0\big)} \cdot \epsilon^2
 N\big) + 2d_{\TV}\big(T_0\big) .
\end{align}

We now use the bound~\eqref{EqnPaulinHoeffding} to prove our
claim~\eqref{eq:confidence-unif-hoeffding}.  Recall that we have
chosen $T_0$ so that $d_{\TV}\big(T_0\big) \leq \alpha_{0}/2$.
Therefore, the bound~\eqref{EqnPaulinHoeffding} implies that
\begin{align*}
\P \left[ \big|\frac{1}{N - T_0}\sum_{n = T_0 + 1}^{N} f(X_{n}) - \mu
  \big| \geq \epsilon_{N} \right] & \leq 2 \exp \Big \{
-\frac{\gamma_{0}}{2\big(2 - \gamma_0\big)} \cdot \epsilon_{N}^2 N
\Big \} + \alpha_0 \\
& \leq 2 \cdot \frac{\alpha - \alpha_0}2 + \alpha_0 = \alpha,
\end{align*}
as required.


\subsection{Proof of the claim~\eqref{eq:confidence-clt-be}}
\label{appsec:confidence-clt-be}

We now use the result~\eqref{eq:markov-clt-be} to prove the
claim~\eqref{eq:confidence-clt-be}.

By the lower bound on $N$, we have
\begin{align*}
 \frac{e^{-\gamma_{0}N}}{3\sqrt{\pimin}} \leq \frac{\alpha}{6} ~~
 \text{and} ~~ \frac{13}{\sigfasym\sqrt{\pimin}} \cdot
 \frac{1}{\gamma_{0}\sqrt{N}} \leq \frac{\alpha}{6}.
\end{align*}
It follows from equation~\eqref{eq:markov-clt-be} that
\begin{align*}
\P \left[ \frac{1}{\sigfasym N}\sum_{n = 1}^{N} f\big(X_n\big) \geq
  \mu + \epsilon_{N} \right] & \leq \Phi\big(\epsilon_{N}\sqrt{N}\big)
+ \frac{\alpha}{3} \\ 
& \leq \exp\big(-\frac{N}2 \cdot \epsilon_{N}^2 \big) +
\frac{\alpha}{3} \\ 
& = \frac{\alpha}2 ,
\end{align*}
and since a matching bound holds for the lower tail, we get the
desired result.


\section{Proofs for Section~\ref{subsec:proofs-testing}}
\label{app:proofs-proofs-testing}

In this section, we gather the proofs of
Lemmas~\ref{lem:seq-seq-indiff-adapt-sum}--\ref{lem:seq-diff-stopping-time-sum}.


\subsection{Proof of Lemma~\ref{lem:seq-seq-indiff-adapt-sum}}

Observe that the function
\begin{align*}
g(\zeta) & \mydefn \exp \Big \{ -\log(1/\beta)\, (\zeta + \zeta^{-1})
\Big \}
\end{align*}
is increasing on $(0,~1]$ and decreasing on $[1,~\infty)$. Therefore,
  bringing $\zeta$ closer to $1$ can only increase the value of the
  function.

Now, for fixed $k \geq 1$, define
\begin{align*}
\ell_k \mydefn \begin{cases} \min\big \{ \ell \colon \big(1 +
  \xi\big)^{\ell} \geq \zeta_{k}\big \} &~ \text{if}~ \zeta_k \leq 1,
  \\ \max\big \{ \ell \colon \big(1 + \xi\Big)^{\ell} \leq
  \zeta_{k}\big \} &~ \text{otherwise.}
\end{cases} 
\end{align*}
In words, the quantity $\ell_k$ is either the smallest integer such
that $\big(1 + \xi\big)^{\ell}$ is bigger than $\zeta_k$ (if $\zeta_k
\leq 1$) or the largest integer such that $\big(1 + \xi\big)^{\ell}$
is smaller than $\zeta_k$ (if $\zeta_k \geq 1$).

With this definition, we see that $\big(1 + \xi\big)^{\ell_k}$ always
lies between $\zeta_k$ and $1$, so that we are guaranteed that
$g\big(\big(1 + \xi\big)^{\ell_k}\big) \geq g\big(\zeta_k\big)$, and
hence
\begin{align*}
\sum_{k = 1}^{\infty} g\big(\zeta_k\big) \leq \sum_{k = 1}^{\infty}
g\big( (1 + \xi)^{\ell_k} \big).
\end{align*}
Thus, it suffices to show that at most two distinct values of $k$ map
to a single $\ell_k$. Indeed, when this mapping condition holds, we
have
\begin{align*}
 \sum_{k = 1}^{\infty} g\big(\zeta_k\big) \leq 2\sum_{\ell =
   -\infty}^{\infty} g\big(\big(1 + \xi\big)^{\ell}\big) \leq
 4\sum_{\ell = 0}^{\infty} g\big(\big(1 + \xi\big)^{\ell}\big) .
\end{align*}

In order to prove the stated mapping condition, note first that
$\ell_k$ is clearly nondecreasing in $k$, so that we need to prove
that $\ell_{k + 2} > \ell_{k}$ for all $k \geq 1$.  It is sufficent to
show that $\zeta_{k + 2} \geq \big(1 + \xi\big)\zeta_{k}$, since this
inequality implies that $\ell_{k + 2} \geq \ell_{k} + 1$.

We now exploit the fact that $\zeta_k = an_{k}$ for some absolute
constant $a$, where $n_{k} = \lfloor n_{0}\big(1 +
\xi\big)^{k}\rfloor$. For this, let $b = n_{0}\big(1 + \xi\big)^{k}$,
so that $n_k = \lfloor b \rfloor$. Since $n_{k + 1} > n_{k}$, we have
$\big(1 + \xi\big)b \geq \lfloor \big(1 + \xi\big)b\rfloor \geq
\lfloor b \rfloor + 1$, and hence
\begin{align*}
\frac{n_{k + 2}}{n_{k}} \; = \; \frac{\lfloor \big(1 + \xi\big)^2
  b\rfloor}{\lfloor b\rfloor} & \geq \frac{\big(1 + \xi\big)^2 b -
  1}{\lfloor b \rfloor} \\ 
& \geq \frac{\big(1 + \xi\big)\big[\lfloor b \rfloor + 1\big] -
  1}{\lfloor b \rfloor} \\
& \geq  1 + \xi,
\end{align*}
as required.\footnote{We thank Daniel Paulin for suggesting this
  argument as an elaboration on the shorter proof
  in~\citet{Gyo15Test}.}


\subsection{Proof of Lemma~\ref{lem:seq-seq-indiff-adapt-group}}
\label{subapp:proof-seq-seq-indiff-adapt-group}

When $c = 0$ and $\ell = 0$, we note that the claim obviously holds
with equality. On the other hand, the left hand side is increasing in
$\ell$, so that the $c = 0$ case follows immediately.

Turning to the case $c > 0$, we first note that it is equivalent to
show that
\begin{align*}
(1 + \xi)^{2\ell} - 2 (c + 1) (1 + \xi)^{\ell} + 1 \geq 0 \qquad
  \mbox{for all $\ell \in ( \frac{9c}{5\xi}, \frac{9 (c + 1)}{5\xi}
    )$.}
\end{align*}
It suffices to show that $(1 + \xi)^{\ell}$ is at least as large as
the largest root of the the quadratic equation $z^2 - 2\big(c +
1\big)z + 1 = 0$.  This largest root is given by 
\begin{align*}
z^* = c + 1 + \sqrt{c \, (c + 2)} \leq 2(c + 1) .
\end{align*}
Consequently, it is enough to show that $\ell \geq \frac{\log{2\left(c + 1\right)}}{\log (1 + \xi)}$.  Since
$\frac{9 c}{5 \xi}$ is a lower bound on $\ell$, we need to verify that
\begin{align*}
\frac{9c}{5\xi} \geq \frac{\log{2\left(c + 1\right)}}{\log (1 + \xi )}.
\end{align*}
In order to verify this claim, note first that since $\xi \leq \frac{2}{5}$,
we have $\log (1 + \xi) \geq \xi - \frac{1}2 \xi^2 \geq
\frac{4}{5}\xi$, whence
\begin{align*}
 \frac{\log{2\left(c + 1\right)}}{\log (1 + \xi )} & \leq \frac{5 \log{2\left(c + 1\right)}}{4\xi} .
\end{align*}
Differentiating the upper bound in $c$, we find that its derivative is
\begin{align*}
\frac{5}{4\left(c + 1\right)\xi} \leq \frac{5}{8\xi} \leq \frac{9}{5\xi},
\end{align*}
so it actually suffices to verify the claim for $c = 1$, which can be done by 
checking numerically that $\frac{5 \log{4}}{4} \leq \frac{9}{5}$. 


\subsection{Proof of Lemma~\ref{lem:seq-seq-stopping-time-sum}}
\label{subapp:proof-seq-seq-stopping-time-sum}

Our strategy is to split the infinite sum into two parts: one
corresponding to the range of $s$ where $h$ is constant and equal to
$1$ and the other to the range of $s$ where $h$ is decreasing. In
terms of the $N_{k}$, these two parts are obtained by splitting the
sum into terms with $k < k_{0}$ and $k \geq k_{0}$, where $k_{0} \geq
1$ is minimal such that $M \leq \Delta N_{k}$ for $k \geq k_{0}$.

For convenience in what follows, let us introduce the convenient shorthand
\begin{align*}
T_k & \mydefn \exp\big(-\frac{\big(\Delta N_{k} - M\big)_{+}^2
}{2\tau_{f}\big(\delta/2\big)N_{k}}\big).
\end{align*}
Now, if $k_{0} = 1$, we note that $h$ must then be decreasing for $s
\geq N_{1}$, so that
\begin{align*}
\sum_{k = 1}^{\infty} \big(N_{k + 1} - N_{k}\big) T_k \leq
\int_{N_{1}}^{\infty} h\big(s\big) \der s .
\end{align*}
Otherwise, if $k_{0} > 1$, we have
\begin{align*}
 \sum_{k = k_{0}}^{\infty} \big(N_{k + 1} - N_{k}\big) T_k & \leq
 \int_{N_{k_{0}}}^{\infty} h\big(s\big) \der s.
\end{align*}
For $k < k_0$, we have $T_k = 1$, so that when $k < k_{0} - 1$, we
have
\begin{align*}
\big(N_{k + 1} - N_{k}\big)\exp\big(-\frac{\big(\Delta N_{k} -
  M\big)_{+}^2 }{2\tau_{f}\big(\delta/2\big)N_{k}}\big) =
\int_{N_{k}}^{N_{k + 1}} h\big(s\big)~ \der s .
\end{align*}
Thus
\begin{align*}
\sum_{k = 1}^{k_{0} - 1} \big(N_{k + 1} - N_{k}\big)
\exp\big(-\frac{\big(\Delta N_{k} - M\big)_{+}^2
}{2\tau_{f}\big(\delta/2\big)N_{k}}\big) = \int_{N_{1}}^{N_{k_{0} -
    1}} h\big(s\big) \der s.
\end{align*}
Note that this implies
\begin{align*}
 \int_{N_{1}}^{\infty} \exp\big(-\frac{\big(\Delta s - M\big)_{+}^2
 }{2\tau_{f}\big(\delta/2\big)s}\big)~ \der s \geq N_{k_{0} - 1} .
\end{align*}
 Finally, we observe that $N_{k + 1} \leq \big(1 + \xi\big)N_{k} + 1 +
 \xi$, so that \mbox{$N_{k_{0}} - N_{k_{0} - 1} \leq \xi N_{k_{0} - 1}
   + 1 + \xi$.} Putting together the pieces, we have
\begin{align*} 
\big(N_{k_{0}} - N_{k_{0} - 1}\big)\exp\big(-\frac{\big(\Delta
  N_{k_{0} - 1} - M\big)_{+}^2 }{2\tau_{f}\big(\delta/2\big)N_{k_{0} -
    1}}\big) \leq 1 + \xi + \xi\int_{N_{1}}^{\infty}
\exp\big(-\frac{\big(\Delta s - M\big)_{+}^2
}{2\tau_{f}\big(\delta/2\big)s}\big)~ \der s,
\end{align*}
and hence
\begin{align*}
 \sum_{k = 1}^{\infty} \big(N_{k + 1} - N_{k}\big)
 \exp\big(-\frac{\big(\Delta N_{k} - M\big)_{+}^2
 }{2\tau_{f}\big(\delta/2\big)N_{k}}\big) \leq 1 + \xi + \big(1 +
 \xi\big)\int_{N_{1}}^{\infty} h\big(s\big) \der s .
\end{align*}


\subsection{Proof of Lemma~\ref{lem:seq-diff-stopping-time-sum}}
\label{subapp:proof-seq-diff-stopping-time-sum}

Observe that for $k > k_{0}^{\ast}$, we have $\Delta - \epsilon_{k} \geq
\frac{\Delta}2 $. It follows that for $k > k_{0}^{\ast}$, we have $T_{f,k}^{+} \leq
T_{f}\big(\frac{\Delta}{4}\big)$. Thus, we can bound each term
in the sum by
\begin{align*}
 \big(N_{k + 1} - N_{k}\big)\exp\big(-\frac{N_{k}}{8T_{f,k}^{+}} \cdot
 \big(\Delta - \epsilon_{k}\big)_{+}^2 \big) \leq
 \underbrace{\big(N_{k + 1} - N_{k}\big)
   \exp\big(-\frac{N_{k}}{T_{f}\big(\frac{\Delta}{4}\big)}
   \cdot \frac{\Delta^2}{32} \big) }_{\SPECIAL_k}.
\end{align*}
Furthermore, the exponential in the definition of $\SPECIAL_k$ is a
decreasing function of $N_{k}$, so we further bound the overall sum as
\begin{align*}
\sum_{k = k_{0}^{\ast} + 1}^{\infty} \SPECIAL_k & \leq \sum_{n =
  N_{0}^{\ast} + 1}^{\infty} \exp\big(-n \cdot \frac{\Delta^2
}{32\Tf\big(\frac{\Delta}{4}\big)}\big) \\
& = \exp\big(-N_{0}^{\ast} \cdot \frac{\Delta^2
}{32\Tf\big(\frac{\Delta}{4}\big)}\big) \times \sum_{m =
  1}^{\infty}\exp\big(-m \cdot \frac{\Delta^2
}{32\Tf\big(\frac{\Delta}{4}\big)}\big) \\ 
& =
\exp\big(-\frac{N_{0}^{\ast}}{8T_{f}\big(\frac{\big(\Delta/2\big)
  }{2}\big)} \cdot \big(\frac{\Delta}2 \big)^2 \big) \times \sum_{m =
  1}^{\infty}\exp\big(-m \cdot \frac{\Delta^2
}{32T_{f}\big(\frac{\Delta}{4}\big)}\big) .
\end{align*}
On the other hand, by the definition of $N_{0}^{\ast}$,
$\epsilon_{k_{0}^{\ast}} \leq \frac{\Delta}2 $, so
\begin{align*}
 T_{f}\big( \frac{\big(\Delta/2\big)}{2}\big) \leq T_{f}\big(
 \frac{\epsilon_{k_{0}^{\ast}}}{2} \big).
\end{align*}
By the definition of $\epsilon_{k}$, however, we know that
\begin{align*}
\frac{\epsilon_{k}^{2}}{8T_{f}\big(\frac{\epsilon_{k}}{2}\big)}
\geq \frac{\log\big(1/\alpha\big) + 1 + 2 \log{k}}{N_{k}} \geq
\frac{\log\big(1/\alpha\big)}{N_{k}},
\end{align*}
which implies that $(\Delta/2)^2 N_{0}^{\ast} \geq \log \big (1/\alpha
\big) \; 8 T_{f} \big( \frac{(\Delta/2)}{2}\big)$.  Re-arranging
yields the claim.

\end{appendices}
\end{document}